\documentclass[12pt,twoside,a4paper]{article}
\usepackage[utf8]{inputenc}
\usepackage{amsmath, amssymb, mathtools, enumerate}
\usepackage{mathrsfs} 
\usepackage{amsfonts, amsthm, stmaryrd, fullpage}
\usepackage{authblk}
\usepackage{tikz-cd}
\usepackage[german, english]{babel}
\usepackage[nottoc,numbib]{tocbibind}

\usepackage{color}
\usepackage{hyperref}
\hypersetup{colorlinks=true,
linktoc=all,
citecolor=blue,
filecolor=black,
linkcolor=blue,
urlcolor=blue}

\newtheorem{thm}{Theorem}[section]
\newtheorem{lemma}[thm]{Lemma}
\newtheorem{cor}[thm]{Corollary}
\newtheorem{prop}[thm]{Proposition}

\theoremstyle{definition} 
\newtheorem{mydef}[thm]{Definition}
  
\newtheorem{example}[thm]{Example}
\newtheorem{question}[thm]{Question}

\theoremstyle{remark}
\newtheorem{rmk}[thm]{Remark}

\newcommand\Ann{{\rm Ann}}
\newcommand\arc{{\rm arc}}

\newcommand\wAss{{\rm wAss}}

\newcommand\con{{\rm con}}

\newcommand\ddiv{{\rm div}}

\newcommand\ev{{\rm ev}}

\newcommand\Frac{{\rm Frac}}

\newcommand\gen{{\rm gen}}
\newcommand\height{{\rm ht}}

\newcommand\Max{{\rm Max}}

\newcommand\Min{{\rm Min}}

\newcommand\ord{{\rm ord}}

\newcommand\spc{{\rm sp}}

\newcommand\Spa{{\rm Spa}}

\newcommand\Spec{{\rm Spec}} 
\newcommand\Spf{{\rm Spf}}

\newcommand\Zar{{\rm Zar}}

\begin{document}

\title{On Shilov boundaries, Rees valuations and integral extensions}
\author{Dimitri Dine}
\date{}
\maketitle

\begin{abstract}We explore an analogy between, on one hand, the notions of integral closure of ideals and Rees valuations in commutative algebra and, on the other hand, the notions of spectral seminorm and Shilov boundary in nonarchimedean geometry. For any Tate ring $\mathcal{A}$ with a Noetherian ring of definition $\mathcal{A}_{0}$ and pseudo-uniformizer $\varpi\in\mathcal{A}_{0}$, we prove that the Shilov boundary for $\mathcal{A}$ naturally coincides with the set of Rees valuation rings of the principal ideal $(\varpi)_{\mathcal{A}_{0}}$ of $\mathcal{A}_{0}$. Furthermore, we characterize the Shilov boundary for a wide class of Tate rings by means of minimal open prime ideals in the subring of power-bounded elements. For affinoid algebras, in the sense of Tate, whose underlying rings are integral domains, this recovers a well-known result of Berkovich. Moreover, under some mild assumptions, we prove stability of our characterization of the Shilov boundary under (completed) integral extensions. In particular, for every mixed-characteristic Noetherian domain $R$, we obtain a description of the Shilov boundary for the Tate ring $\widehat{R^{+}}[p^{-1}]$, where $\widehat{R^{+}}$ is the $p$-adic completion of the absolute integral closure of the domain $R$. \end{abstract}



\tableofcontents

\section{Introduction}\label{sec:introduction}

The contents of this paper lie at the interface between commutative algebra and nonarchimedean geometry.

\subsection{Background}

The starting point for this research was the following classical relation between the notion of integral closure of an ideal in a Noetherian ring $A$ and the so-called asymptotic Samuel function, an invariant which has recently come to play a role in constructive resolution of singularities (see \cite{Benito-Bravo-Encinas}).
\begin{prop}[\cite{Swanson-Huneke} Cor.~6.9.1 and Lemma~6.9.2]\label{Starting point}For any ideal $I$ in a Noetherian ring $A$ and any $f\in A$ the limit \begin{equation*}\overline{\nu}_{I}(f)\coloneqq\lim_{m\to\infty}\frac{\ord_{I}(f^{m})}{m},\end{equation*}where \begin{equation*}\ord_{I}(f)\coloneqq \max\{\, n\mid f\in I^{n}\,\},\end{equation*}exists. Moreover, for any integer $n>0$, an element $f\in A$ belongs to the integral closure $\overline{I^{n}}$ of $I^{n}$ if and only if $\overline{\nu}_{I}(f)\geq n$.\end{prop}Setting $\lVert f\rVert_{I}\coloneqq 2^{-\ord_{I}(f)}$ for every $f\in A$ and \begin{equation*}\vert f\vert_{\spc,I}\coloneqq\lim_{m\to\infty}\lVert f^{m}\rVert_{I}^{1/m},\end{equation*}we can rephrase the above proposition as follows.
\begin{prop}[Reformulation of Proposition~\ref{Starting point}]\label{Starting point 2}Let $I$ be an ideal in a Noetherian ring $A$. For every $f\in A$ the limit \begin{equation*}\vert f\vert_{\spc,I}\coloneqq\lim_{m\to\infty}\lVert f^{m}\rVert_{I}^{1/m}\end{equation*}exists in $\mathbb{R}_{\geq0}$. Moreover, for every integer $n>0$, an element $f\in A$ belongs to the integral closure $\overline{I^{n}}$ of $I^{n}$ if and only if $\vert f\vert_{\spc,I}\leq2^{-n}$.\end{prop}Now, the function $\lVert\cdot\rVert_{I}: A\to\mathbb{R}_{\geq0}$ is a seminorm, known as the $I$-adic seminorm, and the limit \begin{equation*}\vert f\vert_{\spc}\coloneqq \lim_{m\to\infty}\lVert f^{m}\rVert^{1/m}\end{equation*}is known to exist in $\mathbb{R}_{\geq0}$ for any seminorm $\lVert\cdot\rVert$ on a ring $A$ and for any element $f\in A$. The function $\vert\cdot\vert_{\spc}: f\mapsto \vert f\vert_{\spc}$ is again a seminorm on $A$, called the spectral seminorm (derived from the seminorm $\lVert\cdot\rVert$) and plays an important role in nonarchimedean analytic geometry (see, for example, the book \cite{Berkovich} by Berkovich or \S2.8 in the work \cite{Kedlaya-Liu} of Kedlaya-Liu). Thus Proposition~\ref{Starting point 2} provides a link between the notion of integral closure of ideals in commutative algebra and nonarchimedean geometry. 

By a theorem of Berkovich (\cite{Berkovich}~Theorem 1.3.1), the spectral seminorm $\vert\cdot\vert_{\spc}$ on any seminormed ring $\mathcal{A}$ can also be described as \begin{equation*}\vert f\vert_{\spc}=\sup_{v\in\mathcal{M}(\mathcal{A})}v(f)=\max_{v\in\mathcal{M}(\mathcal{A})}v(f),\end{equation*}where $\mathcal{M}(\mathcal{A})$ is the Berkovich spectrum of $\mathcal{A}$ (called Gelfand spectrum in \cite{Berkovich}), defined as the space of all multiplicative seminorms on $\mathcal{A}$ bounded above by the given seminorm $\lVert\cdot\rVert$ and endowed with the weakest topology with respect to which all evaluation functions \begin{equation*}\ev_{f}: \mathcal{M}(\mathcal{A})\to\mathbb{R}_{\geq0}, v\mapsto v(f),\end{equation*}for $f\in \mathcal{A}$, are continuous. By \cite{Berkovich}~Theorem 1.2.1, $\mathcal{M}(\mathcal{A})$ is always a compact Hausdorff space, so the above supremum is indeed a maximum. It is then natural to look for the smallest possible compact subsets of $\mathcal{M}(\mathcal{A})$ on which this maximum is achieved for all elements $f\in \mathcal{A}$.
\begin{mydef}[Boundary, Shilov boundary]\label{Boundary, Shilov boundary}Let $\mathcal{A}$ be a seminormed ring. A subset $\mathcal{S}$ of the Berkovich spectrum $\mathcal{M}(\mathcal{A})$ of $\mathcal{A}$ is called a boundary for $\mathcal{A}$ if every element $f\in \mathcal{A}$, viewed as a continuous function $\mathcal{M}(\mathcal{A})\to\mathbb{R}_{\geq0}$, attains its maximum on $\mathcal{S}$. The subset $\mathcal{S}$ is called the Shilov boundary for $\mathcal{A}$ if it is the smallest closed boundary for $\mathcal{A}$.\end{mydef}It has been shown by Guennebaud in \cite{Guennebaud} and by Escassut-Maïnetti in \cite{EM1} that the Shilov boundary always exists. Moreover, it follows from \cite{Guennebaud}~Ch.~I, Proposition 4, that the Shilov boundary is the same thing as a minimal closed boundary. It can be shown that the notions of spectral seminorm, Berkovich spectrum, boundary and Shilov boundary depend only on the seminorm topology of $\mathcal{A}$ (rather than on a particular choice of seminorm) as long as one restricts one's attention to seminorms having the same value on a topologically nilpotent unit $\varpi$ of $\mathcal{A}$ which is moreover a multiplicative element for those seminorms (see Lemma \ref{Boundedness vs. continuity 2}). This allows us to speak of the Shilov boundary of a Tate ring $\mathcal{A}$ without specifying a seminorm.
\begin{mydef}[Shilov boundary for a Tate ring]\label{Shilov boundary, introduction}Let $\mathcal{A}$ be a Tate ring with topologically nilpotent unit $\varpi$. The Berkovich spectrum $\mathcal{M}_{\varpi}(\mathcal{A})=\mathcal{M}(\mathcal{A})$ of $\mathcal{A}$ relative to $\varpi$ is the Berkovich spectrum of the seminormed ring obtained by equipping $\mathcal{A}$ with any seminorm defining the topology on $\mathcal{A}$ such that \begin{enumerate}[(1)]\item $\lVert\varpi\rVert=\frac{1}{2}$ and \item $\lVert \varpi f\rVert=\lVert \varpi\rVert\lVert f\rVert$ for all $f\in\mathcal{A}$.\end{enumerate}A boundary (respectively, the Shilov boundary) for this seminormed ring (with any seminorm as above) is called a boundary (respectively, the Shilov boundary) for the Tate ring $\mathcal{A}$ relative to $\varpi$.\end{mydef}
We suppress the words 'relative to $\varpi$' whenever the topologically nilpotent unit $\varpi$ is understood from the context. Note that a seminorm $\lVert\cdot\rVert$ defining the topology and satisfying (1) and (2) can always be found for any Tate ring $\mathcal{A}$: For example, choose a ring of definition $\mathcal{A}_{0}$ of $\mathcal{A}$ containing $\varpi$ and let $\lVert\cdot\rVert$ be the canonical extension to $\mathcal{A}$ of the $\varpi$-adic seminorm on $\mathcal{A}_{0}$.   

\subsection{First main theorem: Rees valuations}

Returning to commutative algebra, there is the fairly classical notion of Rees valuations of an ideal $I$ in a Noetherian ring $A$: These form the smallest possible (up to renormalizing) finite family $\mathcal{RV}(I)$ of discrete valuations of rank $1$ on $A$ such that, for every $n>0$, the integral closure $\overline{I^{n}}$ of $I^{n}$ is 'cut out' by $\mathcal{RV}(I)$, in the sense that \begin{equation*}\overline{I^{n}}=\bigcap_{v\in \mathcal{RV}(I)}\{\, f\in A\mid v(f)\leq v(g)~\textrm{for some}~g\in I^{n}\,\}.\end{equation*}By a theorem of Rees (\cite{Swanson-Huneke}~Theorem 10.2.2), every ideal in a Noetherian ring has a set of Rees valuations. Then, Proposition \ref{Starting point 2} suggests an analogy between the Rees valuations of an ideal $I$ and the notion of Shilov boundary for the seminormed ring $(A, \lVert\cdot\rVert_{I})$. The first main result of this paper makes this analogy precise in the case when $I$ is the principal ideal generated by a non-zero-divisor $\varpi$.
\begin{thm}[Theorem~\ref{Rees valuations}]\label{Main theorem 0}Let $\varpi$ be a non-zero-divisor and a non-unit in a Noetherian ring $A$. Then the Shilov boundary for the Tate ring $\mathcal{A}=A[\varpi^{-1}]$ is given by the set $\mathcal{RV}(\varpi)$ of $\varpi$-normalized Rees valuations of $(\varpi)_{A}$.\end{thm}
Since the set of Rees valuations of any ideal $I$ in a Noetherian ring $A$ is finite by definition, this immediately yields the following finiteness result for the Shilov boundary.
\begin{cor}[Corollary \ref{Finite Shilov boundary}]\label{Corollary of Main theorem 0}For any Tate ring $\mathcal{A}$ which admits a Noetherian ring of definition, the Shilov boundary is finite.\end{cor}

\subsection{Second main theorem: Berkovich's description of the Shilov boundary}

For general Tate rings $\mathcal{A}$, little appears to be known on the relation between the Shilov boundary and ring-theoretic properties of $\mathcal{A}$ or its rings of definition. By contrast, for affinoid algebras $\mathcal{A}$ (in the sense of Tate) over a nonarchimedean field $K$, Berkovich obtained an explicit description of the Shilov boundary in terms of the minimal prime ideals of $\mathcal{A}^{\circ}/(\varpi)_{\mathcal{A}^{\circ}}$, where $\mathcal{A}^{\circ}$ is the subring of power-bounded elements of $\mathcal{A}$.

For any Tate ring $\mathcal{A}$, we consider the continuous specialization map \begin{equation*}\spc: \Spa(\mathcal{A}, \mathcal{A}^{\circ})\to \Spf(\mathcal{A}^{\circ})=\Spec(\mathcal{A}^{\circ}/(\varpi)_{\mathcal{A}^{\circ}})\end{equation*}which sends a continuous valuation $v$ on $\mathcal{A}$ to the open prime ideal \begin{equation*}\spc(v)\coloneqq\{\, f\in \mathcal{A}^{\circ}\mid v(f)<1\,\}.\end{equation*}We can identify the Berkovich spectrum $\mathcal{M}(\mathcal{A})$ for the Tate ring $\mathcal{A}$ with the set of rank-$1$ points of $\Spa(\mathcal{A}, \mathcal{A}^{\circ})$. One can then state Berkovich's result as follows (the reformulation using the adic spectrum is due to Bhatt and Hansen in \cite{Bhatt-Hansen}).
\begin{prop}[Berkovich \cite{Berkovich} Prop.~2.4.4, see also Bhatt-Hansen \cite{Bhatt-Hansen} Prop.~2.2]\label{Berkovich's result}Let $\mathcal{A}$ be an affinoid algebra (in the sense of Tate) over a nonarchimedean field $K$. Then the pre-image $\spc^{-1}(\{\mathfrak{p}\})$ of the generic point $\mathfrak{p}$ of any irreducible component of $\Spf(\mathcal{A}^{\circ})$ is a single rank-$1$ point and the Shilov boundary $\mathcal{S}$ for the Tate ring $\mathcal{A}$ is equal to the set $\spc^{-1}(\Spf(\mathcal{A}^{\circ})_{\gen})$ of pre-images of generic points. In particular, the Shilov boundary is finite. \end{prop}
This result of Berkovich motivates the following definition.
\begin{mydef}[Berkovich's description of the Shilov boundary, Definition \ref{Berkovich's description of the Shilov boundary}]We say that a Tate ring $\mathcal{A}$ satisfies Berkovich's description of the Shilov boundary if for every irreducible component of $\Spf(\mathcal{A}^{\circ})$ with generic point $\mathfrak{p}$ the pre-image $\spc^{-1}(\{\mathfrak{p}\})$ consists of exactly one point, which is of rank one, and the closure in $\mathcal{M}(\mathcal{A})$ of the subset \begin{equation*}\spc^{-1}(\Spf(\mathcal{A}^{\circ})_{\gen})\subseteq\mathcal{M}(\mathcal{A})\end{equation*}is the Shilov boundary for $\mathcal{A}$. We say that $\mathcal{A}$ satisfies Berkovich's description of the Shilov boundary in the strong sense if, moreover, $\spc^{-1}(\Spf(\mathcal{A}^{\circ})_{\gen})$ is itself closed in $\mathcal{M}(\mathcal{A})$ and thus equal to the Shilov boundary for $\mathcal{A}$.\end{mydef}
Motivated by a result of Bhatt-Hansen (\cite{Bhatt-Hansen}~Proposition~2.2), our next theorem characterizes the Tate rings which satisfy Berkovich's description of the Shilov boundary algebraically. Note that, by definition, the Shilov boundary for a Tate ring $\mathcal{A}$ with topologically nilpotent unit $\varpi$ coincides with the Shilov boundary for the uniform Tate ring $\mathcal{A}^{u}=\mathcal{A}^{\circ}[\varpi^{-1}]$, with pair of definition $(\mathcal{A}^{\circ}, \varpi)$, so it is harmless to restrict our attention to uniform Tate rings. For a uniform Tate ring $\mathcal{A}$ with topologically nilpotent unit $\varpi$, we denote by $\vert\cdot\vert_{\spc,\varpi}$ the spectral seminorm derived from any seminorm $\lVert\cdot\rVert$ defining the topology on $\mathcal{A}$ and satisfying the assumption of Definition \ref{Shilov boundary, introduction} (by Lemma \ref{Boundedness vs. continuity 2}, this does not depend on a particular choice of $\lVert\cdot\rVert$).  
\begin{thm}[Theorem \ref{Strongly Shilov}]\label{Main theorem 0.5}Let $\mathcal{A}$ be a uniform Tate ring with pseudo-uniformizer $\varpi$ and let $A=\mathcal{A}^{\circ}$.
\begin{itemize}\item Suppose that for every weakly associated prime ideal $\mathfrak{p}$ of $\varpi$ in $A$, the $\varpi$-adic completion of the local ring $A_{\mathfrak{p}}$ is a valuation ring of rank one. Then $\mathcal{A}$ satisfies Berkovich's description of the Shilov boundary. \item Conversely, if $\mathcal{A}$ satisfies Berkovich's description of the Shilov boundary and, moreover, $A/(\varpi)_{A}$ has only finitely many minimal prime ideals, then, for every weakly associated prime ideal $\mathfrak{p}$ of $\varpi$ in $A$, the $\varpi$-adic completion of $A_{\mathfrak{p}}$ is a valuation ring of rank one. \end{itemize}\end{thm}
We also show that in many cases of interest it suffices to check the above condition for minimal prime ideals of $\varpi$.
\begin{thm}[Theorem~\ref{Spectral seminorm and S1}]\label{Spectral seminorm and S1, introduction}Let $\mathcal{A}$ be a uniform Tate ring with topologically nilpotent unit $\varpi$ and let $A=\mathcal{A}^{\circ}$. If the spectral seminorm $\vert\cdot\vert_{\spc,\varpi}$ on $\mathcal{A}$ satisfies \begin{equation*}\vert \mathcal{A}\vert_{\spc,\varpi}\subseteq \sqrt{\vert \mathcal{A}^{\times,m}\vert_{\spc,\varpi}}\cup\{0\}\end{equation*}and $A/(\varpi)_{A}$ has only finitely many minimal prime ideals, then every weakly associated prime ideal $\mathfrak{p}$ of $\varpi$ in $A$ is minimal over $\varpi$.\end{thm}
The notation $\mathcal{A}^{\times,m}$ refers to the group of invertible elements $c$ of $\mathcal{A}$ which are multiplicative for the seminorm $\vert\cdot\vert_{\spc,\varpi}$, i.e., $\vert cf\vert_{\spc,\varpi}=\vert c\vert_{\spc,\varpi}\vert f\vert_{\spc,\varpi}$ for all $f\in\mathcal{A}$. The main result of this paper, stated below, is a list of rather general sufficient conditions for a Tate ring to satisfy Berkovich's description of the Shilov boundary. By a Tate domain we mean a Tate ring whose underlying ring is an integral domain. Note that a uniform Tate ring is automatically reduced if it is Hausdorff, but there is no need for us to assume all our Tate rings are Hausdorff. We call a ring normal if it is reduced and integrally closed in its total ring of fractions.  
\begin{thm}[Theorem~\ref{Reduced Noetherian rings are strongly Shilov}, Proposition~\ref{Multiplicative norms}, Theorem~\ref{Big theorem}, Proposition~\ref{Large value groups and integral extensions}, Corollary~\ref{Flat integral extensions 3}, Theorem~\ref{Integral extensions of Noetherian rings 2}]\label{Main theorem 1}Let $\mathcal{A}$ be a reduced uniform Tate ring with topologically nilpotent unit $\varpi\in \mathcal{A}$. Suppose that one of the following is true: \begin{enumerate}[(1)]\item $\mathcal{A}^{\circ}$ is an integral extension of a reduced Noetherian ring $\mathcal{A}_{0}$ with $\varpi\in\mathcal{A}_{0}$ having the same total ring of fractions as $\mathcal{A}$, or \item $\mathcal{A}$ is normal and $\mathcal{A}^{\circ}$ is a torsion-free integral extension of a Noetherian domain $\mathcal{A}_{0}$ with $\varpi\in \mathcal{A}_{0}$, or \item There is only one minimal prime ideal of $\varpi$ in $\mathcal{A}^{\circ}$ and the spectral seminorm $\vert\cdot\vert_{\spc,\varpi}$ on $\mathcal{A}$ satisfies \begin{equation*}(\ast)~\ ~ \ ~\vert \mathcal{A}\vert_{\spc,\varpi}\subseteq \sqrt{\vert \mathcal{A}^{\times,m}\vert_{\spc,\varpi}}\cup\{0\},\end{equation*}or \item $\mathcal{A}^{\circ}$ is a coherent ring, $\mathcal{A}^{\circ}/(\varpi)_{\mathcal{A}^{\circ}}$ has only finitely many minimal prime ideals and the spectral seminorm $\vert\cdot\vert_{\spc,\varpi}$ on $\mathcal{A}$ satisfies $(\ast)$, or \item The spectral seminorm on $\mathcal{A}$ satisfies $(\ast)$, $\mathcal{A}^{\circ}/(\varpi)_{\mathcal{A}^{\circ}}$ has only finitely many minimal prime ideals and $\mathcal{A}^{\circ}$ is a torsion-free integral extension of $\mathcal{A}^{\prime\circ}$, where $\mathcal{A}^{\prime}$ is a normal uniform Tate domain with topologically nilpotent unit $\varpi$ satisfying Berkovich's description of the Shilov boundary, or \item $\mathcal{A}^{\circ}$ is a flat integral extension of $\mathcal{A}^{\prime\circ}$, where $\mathcal{A}^{\prime}$ is a normal uniform Tate domain with topologically nilpotent unit $\varpi$ such that $\mathcal{A}^{\prime}$ satisfies Berkovich's description of the Shilov boundary and $\mathcal{A}^{\prime\circ}/(\varpi)_{\mathcal{A}^{\prime\circ}}$ has only finitely many minimal prime ideals. \end{enumerate}Then $\mathcal{A}$ satisfies Berkovich's description of the Shilov boundary. Moreover, in all of these cases except possibly in case (6), $\mathcal{A}$ satisfies Berkovich's description of the Shilov boundary in the strong sense.\end{thm}
By Noether normalization for affinoid algebras, cases (3) and (5) of the above theorem imply Berkovich's original result on affinoid algebras (Proposition~\ref{Berkovich's result} above) for those Tate-affinoid algebras whose underlying rings are integral domains. We also note that our theorem implies two other important special cases of Berkovich's result. Case (1) of the theorem recovers Berkovich's result for affinoid algebras in the sense of Tate over a discretely valued nonarchimedean field, while case (4) recovers the same result for affinoid algebras in the sense of Tate over an algebraically closed nonarchimedean field (Remark~\ref{Comparison with Berkovich's result}). A different generalization of Berkovich's description of the Shilov boundary, to affinoid algebras in the sense of Berkovich, was found by Temkin in \cite{Temkin04}~Proposition~3.3, where the concept of so-called graded reductions is used to avoid a hypothesis on values of the spectral norm. It would be interesting to see whether our approach to Theorem~\ref{Main theorem 1}(4) can be combined with Temkin's notion of graded reductions to obtain a common generalization of both results.

In another direction, case (1) of Theorem~\ref{Main theorem 1} includes all uniform Tate domains which have a Noetherian ring of definition.
\begin{cor}\label{Corollary 1 of Main theorem 1}Let $\mathcal{A}$ be a uniform Tate domain and suppose that $\mathcal{A}$ admits a Noetherian ring of definition $\mathcal{A}_{0}$. Then $\mathcal{A}$ satisfies Berkovich's description of the Shilov boundary in the strong sense.\end{cor}
Besides affinoid algebras, we note that Theorem~\ref{Main theorem 1} and Corollary~\ref{Corollary 1 of Main theorem 1} also apply to many Tate rings for which a description of the Shilov boundary was not previously known. For example, pseudoaffinoid algebras in the sense of Lourenço \cite{Lourenco17} have Noetherian rings of definition and thus Theorem~\ref{Main theorem 1}(1) applies to all reduced pseudoaffinoid algebras. Moreover, cases (2), (5) and (6) of Theorem~\ref{Main theorem 1} also yield results on many Tate Banach rings which are completions of integral extensions of Noetherian Tate rings (note that Berkovich's description of the Shilov boundary is readily seen to be preserved under completion). We exemplify this by the following corollary, which might have some applications in commutative algebra (but see also Theorem~\ref{Main theorem 2} below).
\begin{cor}[Corollary~\ref{Absolute integral closure} \footnote{The cited corollary proves that $\widehat{R^{+}}[\varpi^{-1}]$ satisfies Berkovich's description of the Shilov boundary. That it also satisfies Berkovich's description in the strong sense then follows from Lemma~\ref{Integral extensions and closed subsets} applied to the integral extension $\overline{R}\hookrightarrow R^{+}$, where $\overline{R}$ is the integral closure of $R$ in its field of fractions.}]\label{Corollary 2 of Main theorem 1}Let $R$ be a Noetherian domain, let $R^{+}$ be the absolute integral closure of $R$ and let $\widehat{R^{+}}$ be the $\varpi$-adic completion of $R^{+}$ for some non-zero non-unit $\varpi\in R$. Then the uniform Tate ring $\widehat{R^{+}}[\varpi^{-1}]$ satisfies Berkovich's description of the Shilov boundary in the strong sense.  \end{cor}

\subsection{Third main theorem: Completed integral extensions}

Since the introduction of perfectoid spaces by Scholze in \cite{Scholze}, there has been an increased interest in the study of non-Noetherian Tate Banach rings, both in commutative algebra and in $p$-adic arithmetic geometry. This is due to the fact that perfectoid Tate rings, which form the local building blocks for the theory of perfectoid spaces, are almost never Noetherian (more precisely, a perfectoid Tate ring is Noetherian if and only if it is isomorphic to a finite product of perfectoid fields, by \cite{Kedlaya17}~Corollary~2.9.3). Still, many examples of perfectoid Tate rings $\mathcal{A}$ arising in applications (such as in André's and Bhatt's proofs \cite{Andre18-2}, \cite{Bhatt} of the direct summand conjecture, or in $p$-adic Hodge theory) satisfy the following finiteness property: They are completions, with respect to a suitably chosen (semi)norm, of an integral extension of some Noetherian Tate ring $\mathcal{A}'$.

Note that cases (2), (5) and (6) of Theorem~\ref{Main theorem 1} already imply that many Tate Banach rings of this kind satisfy Berkovich's description of the Shilov boundary. Indeed, case (2) of the theorem implies that this is true if the Noetherian Tate ring $\mathcal{A}^{\prime}$ is a domain, its corresponding integral extension is a normal domain and $\mathcal{A}^{\prime}$ has a Noetherian ring of definition (or is the normalization of a Tate ring which has a Noetherian ring of definition). For example, these assumptions are satisfied for \begin{equation*}\mathcal{A}=\widehat{R^{+}}[\varpi^{-1}]\end{equation*}in Corollary~\ref{Corollary 2 of Main theorem 1}, or for the perfectoid Tate rings arising in André's and Bhatt's work on the direct summand conjecture. Similarly, case (5) of the theorem yields Berkovich's description of the Shilov boundary for $\mathcal{A}$ if $\mathcal{A}$ is torsion-free over $\mathcal{A}^{\prime}$, the Noetherian subring $\mathcal{A}^{\prime}$ is an affinoid algebra in the sense of Tate over some nonarchimedean field, $\mathcal{A}^{\circ}$ has only finitely many minimal prime ideals of $\varpi$ and the spectral norm on $\mathcal{A}$ satisfies the assumption $(\ast)$. Our last main theorem removes the assumption on the spectral norm in Theorem \ref{Main theorem 1}(5) at the cost of assuming that the integral extension $\mathcal{B}$ is also a domain.
\begin{thm}[Theorem \ref{Strongly Shilov rings and integral extensions}]\label{Main theorem 2}Let $\mathcal{A}\hookrightarrow \mathcal{B}$ be an integral extension of uniform Tate domains such that $\mathcal{A}$ is a normal domain, $\mathcal{A}^{\circ}/(\varpi)_{\mathcal{A}^{\circ}}$ has only finitely many minimal prime ideals and $\mathcal{B}^{\circ}$ is integral over $\mathcal{A}^{\circ}$. If $\mathcal{A}$ satisfies Berkovich's description of the Shilov boundary, then so does $\mathcal{B}$ (and hence also its completion $\widehat{\mathcal{B}}$). \end{thm}
As noted earlier, preservation of Berkovich's description of the Shilov boundary under completion is automatic, so the actual content of this theorem is preservation under integral extensions. Combining Theorem~\ref{Main theorem 2} with Noether normalization for affinoid algebras, we obtain the following result which covers many examples of perfectoid Tate rings appearing "in nature".
\begin{cor}[Corollary~\ref{Integral extensions}]\label{Corollary 1 of Main theorem 2}Let $\mathcal{B}$ be a uniform Tate Banach ring. Suppose that there exists a dense subring $\mathcal{B}^{\prime}$ of $\mathcal{B}$ such that $\mathcal{B}^{\prime}$ is a domain and $\mathcal{B}^{\prime\circ}$ is an integral extension of $\mathcal{A}^{\circ}$ for some affinoid algebra $\mathcal{A}$ (in the sense of Tate) over a nonarchimedean field. Then $\mathcal{B}$ satisfies Berkovich's description of the Shilov boundary in the strong sense.\end{cor}

\subsection{Main ideas behind the proofs}

Throughout most of the proof of Theorem~\ref{Main theorem 1}, we work with the algebraic characterization of Tate rings satisfying Berkovich's description of the Shilov boundary given in Theorem~\ref{Main theorem 0.5}. Moreover, to deduce the conclusion of Theorem~\ref{Main theorem 1}(4) from the coherence of the ring of power-bounded elements (plus the assumption on the values of the spectral seminorm) we use a ``$\varpi$-local" analog of the notion of $\ast$-operations from the multiplicative ideal theory of integral domains. We hope that the resulting formalism of $\varpi$-$\ast$-operations (which is actually a special case of a very general notion of $\ast$-operations on a ring extension introduced by Knebusch and Kaiser in \cite{Knebusch-Zhang2}~Ch.~3) and our related results on what we call $\varpi$-v-multiplication rings (such as Theorem~\ref{Houston-Malik-Mott}) could be of independent interest. 

We emphasize that, while we formulated most of the results in this introduction in terms of Tate rings, in the body of the paper we mostly take the commutative algebra point of view. That is, we fix a ring $A$, playing the role of a ring of definition, and a non-zero-divisor and non-unit $\varpi$ in $A$ and consider the Tate ring $\mathcal{A}=A[\varpi^{-1}]$ whose topology is determined by the pair of definition $(A, \varpi)$. We can then talk about the Berkovich spectrum and the Shilov boundary for this Tate ring, as we did above. Now we give an overview of the individual sections of the paper.

\subsection{Outline of the paper}

In Section~\ref{sec:power-multiplicative seminorms} we discuss the basics of power-multiplicative seminorms on commutative (multiplicative) monoids, recall results of Guennebaud on the Shilov boundary for a seminormed monoid and give a generalization of the theory of gauges of subsets of vector spaces over a nonarchimedean field to the setting of power-multiplicative monoid seminorms. In particular, we prove that any power-multiplicative function $\nu: X\to\mathbb{R}_{\geq0}$ on a monoid $X$ which admits a 'multiplicative topologically nilpotent unit' $\varpi$ can be described as a generalized gauge of a suitable submonoid (see Proposition~\ref{Description of seminorms} for the precise statement). The additional generality of seminormed monoids (as opposed to seminormed rings) as well as our results on generalized gauges are necessary for our study of integral extensions in Section~\ref{sec:integral extensions}. In particular, they are used in the proof of Proposition~\ref{Spectral seminorms on integral extensions 0}, which generalizes an analogous result of Guennebaud (\cite{Guennebaud}~Ch.~II, Proposition~2) from the case of seminormed algebras over a nonarchimedean field to general seminormed rings.

In Section~\ref{sec:Rees valuations} we introduce, for every non-zero-divisor and non-unit $\varpi$ in a ring $A$, an operation $I\mapsto I^{\arc_{\varpi}}$ on $A$-submodules of $A[\varpi^{-1}]$, called the $\arc_{\varpi}$-operation or $\arc_{\varpi}$-closure, which has the same relationship to the spectral seminorm $\vert\cdot\vert_{\spc,A,\varpi}$ derived from the $\varpi$-adic seminorm on $A$ as the operation of integral closure in the case of Noetherian rings (see Proposition~\ref{Arc-closure and the spectral seminorm} and compare it to Proposition~\ref{Starting point 2}). We then use this relationship to give a criterion for a closed subset of the Berkovich spectrum of $(A[\varpi^{-1}], \lVert\cdot\rVert)$, where $\lVert\cdot\rVert$ is a seminorm defining the topology on the Tate ring $A[\varpi^{-1}]$ such that $\lVert\varpi\rVert=\frac{1}{2}$ and $\varpi$ is multiplicative with respect to $\lVert\cdot\rVert$, to be a boundary (Theorem~\ref{Boundaries and arc-closure}). As an application of this criterion, we prove Theorem~\ref{Main theorem 0} (Theorem~\ref{Rees valuations}). 

In Section~\ref{sec:valuative rings} we gather some facts on $\varpi$-fractional ideals (a '$\varpi$-local' analog of fractional ideals). We also study the notion of $\varpi$-valuative rings, introduced by Fujiwara and Kato in \cite{FK}, and relate it to the notion of Prüfer extensions studied by Knebusch and Zhang in \cite{Knebusch-Zhang}. In our study of Shilov boundaries of Tate rings, the class of $\varpi$-valuative local rings plays an important role since they are precisely those local rings whose $\varpi$-adic completion is a valuation ring (equivalently, whose $\varpi$-adically separated quotient is a valuation ring; see Proposition~\ref{Valuative rings and valuation rings} and Corollary~\ref{Valuative rings and completion}).

Section~\ref{sec:star-operations} is devoted to the study of so-called $\varpi$-$\ast$-operations, the $\varpi$-local analogs of (semi-)star-operations on fractional ideals of an integral domain. Some examples of $\varpi$-$\ast$-operations are the $\arc_{\varpi}$-operation considered in Section~\ref{sec:Rees valuations} as well as the $\varpi$-v- and $\varpi$-t-operations, which are the coarsest strict $\varpi$-$\ast$-operation and the coarsest strict $\varpi$-$\ast$-operation of finite type, respectively (see Proposition~\ref{v- and t-operations are star-operations}). As we noted earlier, our notion of $\varpi$-$\ast$-operations is a special case of the notion of $\ast$-operations on $A$-submodules of a ring extension $A\subseteq R$ introduced by Knebusch and Kaiser in \cite{Knebusch-Zhang2}~Ch.~3 (namely, the special case when $R=A[\varpi^{-1}]$ for a non-zero-divisor $\varpi\in A$), and we freely use results from Knebusch's and Kaiser's book. However, we also prove a series of results whose analogs for $\ast$-operations on the set of fractional ideals of a domain have been known for a long time, but which appear to be new in the $\varpi$-local setting. For example, we prove that, for any $A$ and any non-zero-divisor $\varpi$, weakly associated prime ideals of $\varpi$ are $\varpi$-t-ideals (Lemma~\ref{Weakly associated primes and maximal t-ideals}, Proposition~\ref{Prime ideals minimal over a star-ideal}), a fact which, together with Theorem~\ref{Main theorem 0.5}, motivates our use of $\varpi$-$\ast$-operations in the study of the Shilov boundary. 

Section~\ref{sec:heart of the paper} is the heart of the paper. For a ring $A$ with a non-zero-divisor and non-unit $\varpi\in A$, we study the condition on local rings at weakly associated primes of $\varpi$ which appears in Theorem~\ref{Main theorem 0.5}. We call rings $A$ satisfying this condition for some non-zero-divisor and non-unit $\varpi$ strongly $\varpi$-Shilov. We discuss some basic examples of strongly $\varpi$-Shilov rings (such as Krull domains) and prove Theorem~\ref{Main theorem 0.5} (Theorem~\ref{Strongly Shilov}), Theorem~\ref{Spectral seminorm and S1, introduction} (Theorem~\ref{Spectral seminorm and S1}) as well as Theorem~\ref{Main theorem 1}(1) (Theorem~\ref{Reduced Noetherian rings are strongly Shilov}) and Theorem~\ref{Main theorem 1}(3) (Proposition~\ref{Multiplicative norms}). Our proof of Theorem~\ref{Main theorem 0.5} is based on the main result of Section~\ref{sec:Rees valuations}, Theorem~\ref{Boundaries and arc-closure}. We also introduce the notion of strongly $\varpi$-Krull rings which is more restrictive than the notion of strongly $\varpi$-Shilov rings (see Example~\ref{Example of Heitmann-Ma}, based on the work of Heitmann and Ma \cite{Heitmann-Ma26}), but which still includes many important examples such as Krull domains and reduced Noetherian rings $A$ integrally closed in $A[\varpi^{-1}]$.

In Section~\ref{sec:coherence} we use our results on $\varpi$-$\ast$-operations (in particular, on the $\varpi$-v-operation and the $\varpi$-t-operation) from Section~\ref{sec:star-operations} to establish sufficient conditions for a ring $A$ to be strongly $\varpi$-Shilov with respect to some fixed $\varpi$. As a consequence, we establish case (4) of Theorem~\ref{Main theorem 1}. Our approach to proving this result is inspired by the proofs of several classical characterizations of Prüfer v-multiplication domains, due to Houston-Malik-Mott and Zafrullah (see \cite{Houston-Malik-Mott}~Theorem~1.1, and \cite{Zafrullah78}~Theorem~2, respectively). Hereby, the role played in the classical setting by valuation domains is now occupied by Fujiwara's and Kato's notion of $\varpi$-valuative rings which we studied in Section~\ref{sec:valuative rings}. 

Finally, in Section~\ref{sec:integral extensions} we prove cases (2), (5) and (6) of Theorem~\ref{Main theorem 1} (Theorem~\ref{Integral extensions of Noetherian rings 2}, Proposition~\ref{Large value groups and integral extensions} and Corollary~\ref{Flat integral extensions 3}) as well as Theorem~\ref{Main theorem 2} (Theorem~\ref{Strongly Shilov rings and integral extensions}). The main technical ingredients in the proof of Theorem~\ref{Main theorem 2} are Proposition~\ref{Spectral seminorms on integral extensions 0} and Proposition~\ref{Spectral seminorms and the minimal polynomial 1.5} whose proofs follow the proofs of analogous results from Guennebaud's thesis (\cite{Guennebaud}~Ch.~II, Propositions~2 and 3), but rely on our results on generalized gauges from Section~\ref{sec:power-multiplicative seminorms} to avoid the hypothesis that the seminormed rings in question be seminormed algebras over a nonarchimedean field. This section, along with Section~\ref{sec:power-multiplicative seminorms}, is the most nonarchimedean-analytic in the paper. 

\subsection{Notation and terminology}

We use the same terminological conventions as in \cite{Dine22} (see the subsection "Notation and terminology" in the introduction to that paper). Moreover, in this paper we use the multiplicative (not additive) notation for valuations. For a ring $A$ and a non-unit non-zero-divisor $\varpi\in A$, the expression 'the Tate ring $A[\varpi^{-1}]$' means the ring $A[\varpi^{-1}]$ equipped with the topology defined by the pair of definition $(A, \varpi)$. In this situation, the seminorm \begin{equation*}\lVert f\rVert_{A,\varpi}=\inf\{ 2^{-n}\mid n\in\mathbb{Z}, f\in \varpi^{n}A\}, f\in A[\varpi^{-1}],\end{equation*}on $A[\varpi^{-1}]$ (which defines this topology) is called the canonical extension of the $\varpi$-adic seminorm on $A$.

\subsection{Acknowledgements}

I am much indebted to my advisor, Kiran Kedlaya, for all of his support, guidance and encouragement as well as for his feedback on multiple preliminary versions of this paper.

I would like to express my gratitude to Linquan Ma for some very helpful conversations on the ring $\widehat{R^{+}}$, the $p$-adic completion of the absolute integral closure of a Noetherian complete local domain $R$ of mixed characteristic $(0, p)$, which served as one of the motivating examples for this project, for his comments on a draft of my paper, as well as for sharing with me a preliminary version of his work with Raymond Heitmann \cite{Heitmann-Ma26}, which, among other results, shows that $\widehat{R^{+}}$ is an example of a strongly $p$-Shilov ring which is not strongly $p$-Krull (\cite{Heitmann-Ma26}~Proposition~4.10). This result is discussed in Example~\ref{Example of Heitmann-Ma} below and I also hope to study other connections between the strong $p$-Shilov property and the work of Heitmann and Ma in the future.

Moreover, I would like to thank Linquan Ma for the crucial idea of using prime avoidance to verify condition S1, an argument that I used repeatedly in this paper, for example, in the proofs of Proposition~\ref{Minimal prime ideals and boundaries}, Theorem~\ref{Spectral seminorm and S1} and Theorem~\ref{Flat integral extensions}.

Furthermore, I am grateful to Konstantin Ardakov for pointing out an error in a previous version of this paper, and to Kazuma Shimomoto for helpful comments and, in particular, for bringing to my attention the analogy between Proposition \ref{Complete integral closure and integral extensions} and Krull's classical work \cite{Krull36}.

I would like to thank Ryo Ishizuka for a fruitful discussion of the relationship between the strong $p$-Shilov property and perfectoid towers generated from prisms as in his work \cite{Ishizuka25}, and Jack J Garzella for many enlightening conversations on various topics in nonarchimedean geometry, including on topics close to the contents of this paper.   

Finally, I would like to express my sincere gratitude to the anonymous referee for her/his careful reading of the paper and numerous insightful comments and suggestions for improvement. In particular, I would like to thank the referee for pointing out a mistake in the previous version of Proposition~\ref{Minimal prime ideals and boundaries} and for suggesting a generalization and more conceptual proof of Theorem~\ref{Reduced Noetherian rings are strongly Shilov}. That latter suggestion not only improved Theorem~\ref{Reduced Noetherian rings are strongly Shilov} and its proof, but ultimately helped me discover the statement and proof of Theorem~\ref{Spectral seminorm and S1}. 

\subsection{AI disclosure}The results of this paper pre-date the AI-in-mathematics era. All results and proofs in the paper are the author's own. However, during the revision, the author used Claude Fable 5 and Claude Fable 5.1 to proofread the paper and point out some (mostly formal) mistakes, which were subsequently corrected by the author. Claude Fable was also used for searching the literature at one point during the revision.                 

\section{Power-multiplicative seminorms}\label{sec:power-multiplicative seminorms}

In this paper the words 'semigroup' and 'monoid' will always mean a commutative semigroup and a commutative monoid and the word 'ring' will mean a commutative unital ring. Let $\Gamma$ be a totally ordered (commutative) monoid such that every subset of $\Gamma$ which is bounded below has an $\inf$ and every subset of $\Gamma$ which is bounded above has a $\sup$. We write $\Gamma$ multiplicatively, except where explicitly stated otherwise. We consider $\Gamma^{\prime}=\Gamma\cup\{0\}$ as a totally ordered monoid by letting $0\cdot\gamma=0$ and $0\leq\gamma$ for all $\gamma$. 
\begin{mydef}Given a (commutative) totally ordered monoid $\Gamma$, we call a function \begin{equation*}\mu: X\to \Gamma\end{equation*}on a (commutative) monoid $X$ (which we also write multiplicatively) a submultiplicative function if \begin{equation*}\mu(fg)\leq \mu(f)\mu(g)\end{equation*}for all $f, g\in X$. We call the function $\mu$ power-multiplicative if $\mu(f^{m})=\mu(f)^{m}$ for all integers $m>0$ and we call it multiplicative if $\mu(fg)=\mu(f)\mu(g)$.\end{mydef} 
The most important special cases of these notions are when $\mu$ is a ring seminorm $\lVert\cdot\rVert$ on a (commutative) ring $\mathcal{A}$. However, in this section we do not restrict our attention to this case. 
\begin{mydef}[Seminormed monoid]\label{Seminormed monoid}Let $\Gamma$ be a totally ordered monoid.\begin{itemize}\item A $\Gamma$-seminorm (more precisely, a $\Gamma$-valued monoid seminorm) on a (commutative) monoid $X$ is a submultiplicative function $\nu: X\to\Gamma$ with $\nu(1)=1$. If $\Gamma=\mathbb{R}_{\geq0}=\mathbb{R}_{>0}\cup\{0\}$, where $\mathbb{R}_{>0}$ is the multiplicative monoid of positive real numbers, we call a $\Gamma$-valued monoid seminorm $\nu$ on $X$ simply a (monoid) seminorm on $X$. 
\item A pair $(X, \mu)$ consisting of a commutative monoid $X$ and a $\Gamma$-seminorm $\mu$ on $X$ is called a $\Gamma$-seminormed monoid. If $\Gamma=\mathbb{R}_{\geq0}$, such a pair is simply called a seminormed monoid. \item A morphism of $\Gamma$-seminormed monoids (or bounded monoid homomorphism) \begin{equation*}\varphi: (X, \mu)\to (X^{\prime}, \mu^{\prime})\end{equation*}is a morphism of monoids $\varphi: X\to X^{\prime}$ such that there exists an element $C\in\Gamma$ which is not the least element in $\Gamma$ satisfying \begin{equation*}\mu^{\prime}(\varphi(f))\leq C\mu(f)\end{equation*}for all $f\in X$. The morphism of $\Gamma$-seminormed monoids is called submetric if the above constant $C$ can be chosen to be equal to $1$.\end{itemize}\end{mydef}
\begin{example}[Seminormed abelian group]If $X$ is an abelian group and $\Gamma$ is the additive monoid $(\mathbb{R}_{\geq0}, +)$ of non-negative real numbers, with the binary operation written additively and with neutral element $0$, then we recover the usual notion of a (not necessarily nonarchimedean) abelian group seminorm.\end{example}
\begin{example}[Seminormed ring]If $\mathcal{A}$ is a ring, then a (not necessarily nonarchimedean) ring seminorm on $\mathcal{A}$ is precisely a function $\lVert\cdot\rVert: \mathcal{A}\to \mathbb{R}_{\geq0}$ which is an $(\mathbb{R}_{\geq0}, +)$-seminorm on the underlying additive abelian group of $\mathcal{A}$ and at the same time also a monoid seminorm (i.e., an $(\mathbb{R}_{\geq0}, \cdot)$-seminorm) on the multiplicative monoid of $\mathcal{A}$. Note, however, that in this paper we mostly restrict our attention to seminormed rings which are nonarchimedean, i.e., whose seminorm satisfies the stronger nonarchimedean triangle inequality $\lVert f+g\rVert\leq\max(\lVert f\rVert, \lVert g\rVert)$ for all $f, g\in \mathcal{A}$.\end{example}
\begin{example}[Closed unit ball]For a seminormed monoid $(X, \mu)$, consider the closed unit ball \begin{equation*}X_{\mu\leq1}=\{\, f\in X\mid \mu(f)\leq1\,\}.\end{equation*}Then the restriction of the seminorm $\mu$ endows $X_{\mu\leq1}$ with the structure of a seminormed monoid.\end{example}
\begin{example}[Punctured closed unit ball]If $X$ is the underlying multiplicative monoid of an integral domain $\mathcal{A}$ and $\mu$ is a monoid seminorm on $X$, the restriction of $\mu$ also endows $X_{\mu\leq1}\setminus\{0\}$ with the structure of a seminormed monoid.\end{example}
For a seminormed monoid $(X, \mu)$ there is the following analog of the notion of Berkovich spectrum of a seminormed ring, which was studied by Guennebaud in his thesis \cite{Guennebaud} and which we therefore call the Guennebaud spectrum.
\begin{mydef}[Guennebaud spectrum]\label{Guennebaud spectrum}Fix an (as always, commutative) totally ordered monoid $\Gamma$ and set $\Gamma^{\prime}=\Gamma\cup\{0\}$. For a $\Gamma^{\prime}$-seminormed monoid $(X, \mu)$, the Guennebaud spectrum $M(X, \mu)$ of $(X, \mu)$ is the space of all multiplicative functions $\nu: X\to \Gamma^{\prime}$ satisfying $\nu\leq\mu$, endowed with the weakest topology making all evaluation maps $M(X, \mu)\to\Gamma^{\prime}$, $\nu\mapsto \nu(f)$, for $f\in X$, continuous, for the order topology on $\Gamma^{\prime}$.\end{mydef}
It is readily seen that the topology on $M(X, \mu)$ is the topology of pointwise convergence and that it is equal to the subspace topology induced from the canonical inclusion \begin{equation*}M(X, \mu)\subseteq \Gamma^{\prime X}.\end{equation*}If $\Gamma^{\prime}=\mathbb{R}_{\geq0}$, this topology on $M(X, \mu)$ turns out to have quite reasonable properties.
\begin{thm}[Guennebaud \cite{Guennebaud}~Ch.~I, second paragraph of \S4]\label{The Guennebaud spectrum is compact}For every seminormed monoid $(X, \mu)$, the Guennebaud spectrum $M(X, \mu)$ is a compact Hausdorff space.\end{thm}
\begin{proof}It is readily seen from the definition that $M(X, \mu)$ is a subspace of the product \begin{equation*}\prod_{f\in X}[0, \mu(f)].\end{equation*}By Tychonov's theorem, this product is a compact Hausdorff space, so it suffices to show that $M(X, \mu)$ is closed in $\prod_{f\in X}[0, \mu(f)]$ or, equivalently, closed in $\mathbb{R}_{\geq0}^{X}$. But the pointwise limit of a net of multiplicative functions $X\to\mathbb{R}_{\geq0}$ is again multiplicative and the pointwise limit of functions bounded above by $\mu$ is again bounded above by $\mu$. \end{proof}     
In what follows we restrict our attention to the case when the totally ordered monoid $\Gamma$ is the multiplicative monoid of positive real numbers $\mathbb{R}_{>0}$ (so, $\Gamma^{\prime}=\mathbb{R}_{\geq0}$). For any seminormed monoid $(X, \mu)$ there exists a largest power-multiplicative seminorm $\vert\cdot\vert_{\spc, (X, \mu)}$ on $X$ bounded with respect to the given seminorm $\mu$, namely, 
\begin{equation*}\vert f\vert_{\spc, (X, \mu)}=\inf_{n}\mu(f^{n})^{1/n}=\lim_{n\to\infty}\mu(f^{n})^{1/n}, f\in X.\end{equation*}The seminorm $\vert\cdot\vert_{\spc, (X, \mu)}$ is called the spectral seminorm of the seminormed monoid $(X, \mu)$ or the spectral seminorm derived from the seminorm $\mu$ on $X$. When the seminormed monoid $(X, \mu)$ is understood from the context, we simply write $\vert\cdot\vert_{\spc}$ instead of $\vert\cdot\vert_{\spc, (X, \mu)}$. It is clear that a seminorm $\mu$ on a monoid $X$ is equal to its spectral seminorm if and only if it is itself power-multiplicative. By a result of Guennebaud, the spectral seminorm can be alternatively described as follows.
\begin{thm}[Guennebaud \cite{Guennebaud}~Ch.~I, Proposition~10]\label{Spectral seminorm}For any seminormed monoid $(X, \mu)$, the spectral seminorm $\vert\cdot\vert_{\spc, (X, \mu)}$ of $(X, \mu)$ is given by\begin{equation*}\vert f\vert_{\spc, (X, \mu)}=\sup_{\nu\in M(X, \mu)}\nu(f), f\in X,\end{equation*}where $M(X, \mu)$ is the Guennebaud spectrum of $(X, \mu)$.\end{thm}
Note that, by compactness of the Guennebaud spectrum, the supremum in the above description of the spectral seminorm is actually a maximum (an actually stronger statement follows from \cite{Guennebaud}~Ch.~I, Corollaire du Théorème~1).

Power-multiplicative seminorms enjoy the following useful property, whose (straightforward) proof we recall for the reader's convenience.
\begin{lemma}[Cf.~\cite{BGR}~Proposition~1.3.1/2, in the case of seminormed rings]\label{Bounded implies submetric}For any morphism of seminormed monoids $\varphi: (X, \mu)\to (X^{\prime}, \mu^{\prime})$ we have\begin{equation*}\vert\varphi(f)\vert_{\spc, (X^{\prime}, \mu^{\prime})}\leq \vert f\vert_{\spc, (X, \mu)}\end{equation*}for all $f\in X$, where $\vert\cdot\vert_{\spc, (X, \mu)}$, $\vert\cdot\vert_{\spc, (X^{\prime}, \mu^{\prime})}$ are the spectral seminorms derived from $\mu$, $\mu^{\prime}$. In particular, if $\mu^{\prime}$ is power-multiplicative, $\varphi$ is submetric.\end{lemma}
\begin{proof}Choose a constant $C>0$ such that $\mu^{\prime}(\varphi(f))\leq C\mu(f)$ for all $f\in X$. Fix some $f\in X$. Then, for all $n\geq1$, we have \begin{equation*}\vert \varphi(f)\vert_{\spc, (X^{\prime}, \mu^{\prime})}^{n}=\vert\varphi(f^{n})\vert_{\spc, (X^{\prime}, \mu^{\prime})}\leq\mu^{\prime}(\varphi(f^{n}))\leq C\mu(f^{n})\end{equation*}or, equivalently, $\vert \varphi(f)\vert_{\spc, (X^{\prime}, \mu^{\prime})}\leq C^{1/n}\mu(f^{n})^{1/n}$ for all $n$. Taking limits on both sides yields $\vert\varphi(f)\vert_{\spc, (X^{\prime}, \mu^{\prime})}\leq \vert f\vert_{\spc, (X, \mu)}$, as desired.\end{proof}
Using the concept of the spectral seminorm, one can prove the following criterion for a given seminorm to be power-multiplicative.
\begin{lemma}\label{Power-multiplicative}Let $(X, \mu)$ be a seminormed monoid. The seminorm $\mu$ is power-multiplicative if and only if there exists an integer $k\geq2$ such that $\mu(f^{k})=\mu(f)^{k}$ for all $f\in X$.\end{lemma}
\begin{proof}If there exists an integer $k\geq2$ as above, then $\mu(f^{k^{n}})=\mu(f)^{k^{n}}$ for all $n\geq1$ and $f\in X$. But $(k^{n})_{n}$ is a cofinal subset of $\mathbb{Z}_{>1}$, so we have $\vert f\vert_{\spc}=\lim_{n}\mu(f^{n})^{1/n}=\lim_{n}\mu(f^{k^{n}})^{1/k^{n}}=\mu(f)$. \end{proof}
When the seminormed monoid in question is the underlying multiplicative monoid of a seminormed ring, there is the following more well-known analog of the Guennebaud spectrum, which was introduced (for Banach rings) by Berkovich in \cite{Berkovich}.
\begin{mydef}[Berkovich spectrum]\label{Berkovich spectrum}For a (nonarchimedean) seminormed ring $(\mathcal{A}, \lVert\cdot\rVert)$, the Berkovich spectrum $\mathcal{M}((\mathcal{A}, \lVert\cdot\rVert))$ is the space of bounded multiplicative ring seminorms $\phi: \mathcal{A}\to \mathbb{R}_{\geq0}$ on $\mathcal{A}$ endowed with the weakest topology making all evaluation maps $\phi\mapsto \phi(f)$, for $f\in \mathcal{A}$, continuous.\end{mydef}
By Lemma~\ref{Bounded implies submetric}, every bounded multiplicative ring seminorm $\phi$ actually satisfies $\phi\leq \lVert\cdot\rVert$, so the Berkovich spectrum is naturally a subspace of the Guennebaud spectrum of the multiplicative seminormed monoid of $(\mathcal{A}, \lVert\cdot\rVert)$. Again, the topology on the Berkovich spectrum can be described as the topology of pointwise convergence or, equivalently, the subspace topology induced from the inclusion \begin{equation*}\mathcal{M}((\mathcal{A}, \lVert\cdot\rVert))\subseteq \mathbb{R}_{\geq0}^{\mathcal{A}}.\end{equation*}Similar to Theorem~\ref{The Guennebaud spectrum is compact}, there is the following result due to Berkovich.
\begin{thm}[Berkovich \cite{Berkovich}~Theorem~1.2.1]\label{The Berkovich spectrum is compact}For every seminormed ring $(\mathcal{A}, \lVert\cdot\rVert)$, the Berkovich spectrum $\mathcal{M}((\mathcal{A}, \lVert\cdot\rVert))$ is a non-empty compact Hausdorff space.\end{thm}
For a seminormed ring $(\mathcal{A}, \lVert\cdot\rVert)$, the spectral seminorm $\vert\cdot\vert_{\spc}$ derived from $\lVert\cdot\rVert$ is a ring seminorm and we have the following analog of Theorem~\ref{Spectral seminorm}.
\begin{thm}[Berkovich \cite{Berkovich}~Theorem~1.3.1]\label{Spectral seminorm 2}Let $(\mathcal{A}, \lVert\cdot\rVert)$ be a seminormed ring. Then the spectral seminorm $\vert\cdot\vert_{\spc}$ derived from $\lVert\cdot\rVert$ is given by \begin{equation*}\vert f\vert_{\spc}=\sup_{\phi\in\mathcal{M}((\mathcal{A}, \lVert\cdot\rVert))}\phi(f), f\in \mathcal{A}.\end{equation*}\end{thm}
By compactness of the Berkovich spectrum, the supremum in the above theorem is again a maximum. It is natural to study subsets on which this maximum is attained.
\begin{mydef}[Boundary, Shilov boundary]Let $(X, \mu)$ (respectively, $(\mathcal{A}, \lVert\cdot\rVert)$) be a seminormed monoid (respectively, a seminormed ring). A subset $\mathcal{S}$ of $M(X, \mu)$ (respectively, of $\mathcal{M}((\mathcal{A}, \lVert\cdot\rVert))$) is called a boundary for $(X, \mu)$ (respectively, for $(\mathcal{A}, \lVert\cdot\rVert)$) if for every $f\in X$ (respectively, for every $f\in \mathcal{A}$) there exists some $\nu\in\mathcal{S}$ with \begin{equation*}\vert f\vert_{\spc}=\nu(f).\end{equation*}If there exists a smallest closed boundary $\mathcal{S}$ for $(X, \mu)$ (respectively, for $(\mathcal{A}, \lVert\cdot\rVert)$), it is called the Shilov boundary for $(X, \mu)$ (respectively, for $(\mathcal{A}, \lVert\cdot\rVert)$).\end{mydef}
We also sometimes speak of a boundary or of the Shilov boundary of $M(X, \mu)$ (respectively, of $\mathcal{M}((\mathcal{A}, \lVert\cdot\rVert))$) when we mean a boundary or the Shilov boundary for the seminormed monoid $(X, \mu)$ (respectively, the seminormed ring $(\mathcal{A}, \lVert\cdot\rVert)$). Following Guennebaud in \cite{Guennebaud}, we denote by $\Min(X, \mu)$ the set of multiplicative monoid seminorms $\nu\in M(X, \mu)$ which are constructible from $\mu$ (see \cite{Guennebaud}~Ch.~I, \S2, or the introduction of \cite{EM1} for the latter notion). For simplicity in referring to it, we give a name to this subspace of $M(X, \mu)$.
\begin{mydef}[Constructible spectrum]\label{Constructible spectrum}For a seminormed monoid $(X, \mu)$, we call the subspace $\Min(X, \mu)$ of multiplicative monoid seminorms on $X$ which are constructible from $\mu$ in the sense of Guennebaud the constructible spectrum of $(X, \mu)$. If $(\mathcal{A}, \lVert\cdot\rVert)$ is a seminormed ring, we write $\Min(\mathcal{A}, \lVert\cdot\rVert)$ for the constructible spectrum of the underlying multiplicative seminormed monoid of $(\mathcal{A}, \lVert\cdot\rVert)$ and call it the constructible spectrum of the seminormed ring $(\mathcal{A}, \lVert\cdot\rVert)$.\end{mydef}
It turns out that, for a seminormed ring $(\mathcal{A}, \lVert\cdot\rVert)$, the constructible spectrum is a subset of the Berkovich spectrum.
\begin{prop}[See \cite{Guennebaud}~Ch.~I, Prop.~2 and Exemple~1]\label{Constructible spectrum and Berkovich spectrum}If $(\mathcal{A}, \lVert\cdot\rVert)$ is a seminormed ring, then every multiplicative monoid seminorm on $\mathcal{A}$ which is constructible from $\lVert\cdot\rVert$ is actually a nonarchimedean ring seminorm. In symbols: \begin{equation*}\Min(\mathcal{A}, \lVert\cdot\rVert)\subseteq \mathcal{M}((\mathcal{A}, \lVert\cdot\rVert)).\end{equation*}\end{prop}
We can now recall a result of Guennebaud \cite{Guennebaud} and Escassut-Maïnetti \cite{EM1} which ensures the existence of the Shilov boundary for every seminormed monoid and every seminormed ring.
\begin{thm}[Guennebaud \cite{Guennebaud}~Ch.~I, Cor.~du Théorème~1 on p.~8, Théorème 3 on p.~14, Escassut-Maïnetti \cite{EM1}~Theorem~C]\label{Guennebaud's theorem}For every seminormed monoid $(X, \mu)$, the constructible spectrum $\Min(X, \mu)$ is a boundary for $(X, \mu)$ (and, in particular, non-empty). Moreover, the Shilov boundary for $(X, \mu)$ exists and is equal to the (topological) closure of $\Min(X, \mu)$ inside $M(X, \mu)$. In particular, for every seminormed ring $(\mathcal{A}, \lVert\cdot\rVert)$, the Shilov boundary for $(\mathcal{A}, \lVert\cdot\rVert)$ exists and is equal to the closure of $\Min(\mathcal{A}, \lVert\cdot\rVert)$ inside $\mathcal{M}((\mathcal{A}, \lVert\cdot\rVert))$.\end{thm}
We also note the following related result from Guennebaud's thesis. 
\begin{thm}[Guennebaud \cite{Guennebaud}~Ch.~I, Théorème~2 on p.~8]Let $(X, \mu)\hookrightarrow (X^{\prime}, \mu^{\prime})$ be an isometry of seminormed monoids. Then every element of $\Min(X, \mu)$ is the restriction to $X$ of some element of $\Min(X^{\prime}, \mu^{\prime})$.\end{thm}  
In the rest of this section, we aim to give, under some mild conditions, an explicit description of any power-multiplicative ring seminorm which will be useful in our study of integral extensions in Section~\ref{sec:integral extensions}.

For a seminormed ring $(\mathcal{A}, \lVert\cdot\rVert)$ we call an element $\varpi\in \mathcal{A}$ seminorm-multiplicative (or multiplicative with respect to $\lVert\cdot\rVert$) if $\lVert\varpi f\rVert=\lVert \varpi\rVert\lVert f\rVert$ for all $f\in \mathcal{A}$. We denote by $(\mathcal{A}, \lVert\cdot\rVert)^{m,\times}\subseteq \mathcal{A}^{\times}$ the subgroup of seminorm-multiplicative units in a seminormed ring $(\mathcal{A}, \lVert\cdot\rVert)$ and we also write $\mathcal{A}^{\times,m}$ instead of $(\mathcal{A}, \lVert\cdot\rVert)^{m,\times}$ when the seminorm $\lVert\cdot\rVert$ is understood from the context. The presence of a seminorm-multiplicative topologically nilpotent unit $\varpi$ in a seminormed ring $(\mathcal{A}, \lVert\cdot\rVert)$ ensures that, up to bounded-equivalence, the seminorm on $\mathcal{A}$ is determined by the topology it defines and its value on $\varpi$.
\begin{lemma}\label{Boundedness vs. continuity}Let $(\mathcal{A}, \lVert\cdot\rVert)$ be a seminormed ring and let $\phi$ be a seminorm on $\mathcal{A}$ which is continuous with respect to the topology on $\mathcal{A}$ defined by $\lVert\cdot\rVert$. Suppose that there exists a topologically nilpotent unit $\varpi\in \mathcal{A}$ which is multiplicative with respect to both $\lVert\cdot\rVert$ and $\phi$. If $\phi(\varpi)=\lVert\varpi\rVert$, then the seminorm $\phi$ is bounded with respect to $\lVert\cdot\rVert$. In particular, every continuous multiplicative seminorm $\phi$ on $\mathcal{A}$ which satisfies $\phi(\varpi)=\lVert\varpi\rVert$ is bounded.\end{lemma}
\begin{proof}Follows from the proof of \cite{Dine22}~Lemma 2.11.\end{proof}
We need the following special case of the above lemma.
\begin{lemma}\label{Boundedness vs. continuity 2}Let $\mathcal{A}$ be a Tate ring with topologically nilpotent unit $\varpi$ and let $\lVert\cdot\rVert$, $\lVert\cdot\rVert'$ be seminorms defining the topology on $\mathcal{A}$ such that $\lVert\varpi\rVert=\lVert\varpi\rVert'=\frac{1}{2}$ and such that $\varpi$ is multiplicative with respect to both $\lVert\cdot\rVert$ and $\lVert\cdot\rVert'$. Then $\lVert\cdot\rVert$, $\lVert\cdot\rVert'$ are bounded-equivalent and their respective spectral seminorms are equal.\end{lemma}
\begin{proof}The statement that the seminorms are bounded-equivalent is a special case of Lemma~\ref{Boundedness vs. continuity}. The equality of spectral seminorms follows from this and Lemma~\ref{Bounded implies submetric}.\end{proof}
\begin{mydef}[Spectral seminorm on a Tate ring]For a Tate ring $\mathcal{A}$ and a topologically nilpotent unit $\varpi\in\mathcal{A}$, we call the spectral seminorm $\vert\cdot\vert_{\spc, \varpi}$ derived from any seminorm on $\mathcal{A}$ satisfying the assumptions of Lemma~\ref{Boundedness vs. continuity 2} the spectral seminorm on $\mathcal{A}$ (relative to $\varpi$).\end{mydef} 
Recall that a Tate ring (or, more generally, a Huber ring) $\mathcal{A}$ is called uniform if its subring of power-bounded elements $\mathcal{A}^{\circ}$ is bounded.
\begin{lemma}\label{Uniform Tate rings}Let $\mathcal{A}$ be a uniform Tate ring and let $\varpi\in \mathcal{A}$ be a topologically nilpotent unit. Then the spectral seminorm $\vert\cdot\vert_{\spc,\varpi}$ defines the topology on $\mathcal{A}$ and is the unique power-multiplicative seminorm defining the topology on $\mathcal{A}$ for which $\varpi$ is seminorm-multiplicative and $\vert\varpi\vert_{\spc, \varpi}=\frac{1}{2}$. In particular, for any ring of definition $\mathcal{A}_{0}$ of $\mathcal{A}$ with $\varpi\in \mathcal{A}_{0}$ the spectral seminorm derived from the canonical extension $\lVert\cdot\rVert_{\mathcal{A}_{0}, \varpi}$ of the $\varpi$-adic seminorm on $\mathcal{A}_{0}$ is equal to $\vert\cdot\vert_{\spc, \varpi}$.\end{lemma}
\begin{proof}Suppose that there is a power-multiplicative seminorm $\vert\cdot\vert_{\spc, \varpi}$ defining the topology on $\mathcal{A}$ for which $\varpi$ is seminorm-multiplicative and $\vert\varpi\vert_{\spc, \varpi}=\frac{1}{2}$. Then uniqueness of such a seminorm immediately follows from Lemma~\ref{Boundedness vs. continuity 2}. Thus it remains to prove existence. Since $\mathcal{A}$ is uniform, $\mathcal{A}^{\circ}$ is a ring of definition of $\mathcal{A}$ and thus the canonical extension $\lVert\cdot\rVert_{\mathcal{A}^{\circ}, \varpi}$ of the $\varpi$-adic seminorm on $\mathcal{A}^{\circ}$ defines the topology on the Tate ring $\mathcal{A}$. Let $\vert\cdot\vert_{\spc, \varpi}$ be the spectral seminorm derived from $\lVert\cdot\rVert_{\mathcal{A}^{\circ}, \varpi}$. Since $\varpi$ is multiplicative with respect to $\lVert\cdot\rVert_{\mathcal{A}^{\circ}, \varpi}$ and $\lVert\varpi\rVert_{\mathcal{A}^{\circ}, \varpi}=\frac{1}{2}$, we must have $\vert\varpi\vert_{\spc, \varpi}=\frac{1}{2}$. Since $(\mathcal{A}^{\circ}, \varpi)$ is a pair of definition of $\mathcal{A}$, to prove that $\vert\cdot\vert_{\spc, \varpi}$ defines the topology on $\mathcal{A}$ it suffices to prove that for every $n\geq1$ the ideal $\varpi^{n}\mathcal{A}^{\circ}$ of $\mathcal{A}^{\circ}$ contains some open $\vert\cdot\vert_{\spc, \varpi}$-ball. It is readily seen from the definition of the spectral seminorm that $\varpi$ being multiplicative with respect to $\lVert\cdot\rVert_{\mathcal{A}^{\circ}, \varpi}$ implies that $\varpi$ is also multiplicative with respect to $\vert\cdot\vert_{\spc, \varpi}$. Hence it suffices to prove that $\mathcal{A}^{\circ}$ itself contains some open $\vert\cdot\vert_{\spc, \varpi}$-ball. However, if $\vert f\vert_{\spc, \varpi}<1$, then $\lVert f^{n}\rVert_{\mathcal{A}^{\circ}, \varpi}<1$ for all large enough integers $n$ and thus $f$ is topologically nilpotent and, a fortiori, power-bounded.\end{proof}
We can also use Lemma~\ref{Boundedness vs. continuity} to define the Berkovich spectrum and Shilov boundary of a Tate ring $\mathcal{A}$ with a pseudo-uniformizer $\varpi$ independently of a choice of seminorm.
\begin{mydef}[Berkovich spectrum of a Tate ring]\label{Berkovich spectrum of a Tate ring}For a Tate ring $\mathcal{A}$ with topologically nilpotent unit $\varpi$, the Berkovich spectrum \begin{equation*}\mathcal{M}(\mathcal{A})=\mathcal{M}_{\varpi}(\mathcal{A})\end{equation*}of $\mathcal{A}$ relative to $\varpi$ is the Berkovich spectrum of the seminormed ring $(\mathcal{A}, \lVert\cdot\rVert)$, where $\lVert\cdot\rVert$ is any seminorm defining the topology on $\mathcal{A}$ such that $\lVert\varpi\rVert=\frac{1}{2}$ and $\varpi$ is multiplicative with respect to $\lVert\cdot\rVert$. Equivalently, \begin{equation*}\mathcal{M}_{\varpi}(\mathcal{A})=\mathcal{M}((\mathcal{A}, \vert\cdot\vert_{\spc,\varpi})).\end{equation*}We just speak of the Berkovich spectrum $\mathcal{M}(\mathcal{A})$ of the Tate ring $\mathcal{A}$ when a choice of topologically nilpotent unit $\varpi\in\mathcal{A}$ is understood.\end{mydef}
\begin{mydef}[Shilov boundary for a Tate ring]\label{Shilov boundary for a Tate ring}For a Tate ring $\mathcal{A}$ with topologically nilpotent unit $\varpi$, a subset $\mathcal{S}$ of $\mathcal{M}_{\varpi}(\mathcal{A})$ is called a boundary (respectively, the Shilov boundary) for $\mathcal{A}$ relative to $\varpi$ if it is a boundary (respectively, the Shilov boundary) for the seminormed ring $(\mathcal{A}, \lVert\cdot\rVert)$, where $\lVert\cdot\rVert$ is any seminorm defining the topology on $\mathcal{A}$ such that $\lVert\varpi\rVert=\frac{1}{2}$ and such that $\varpi$ is multiplicative with respect to $\lVert\cdot\rVert$.

Equivalently, $\mathcal{S}$ is a boundary (respectively, the Shilov boundary) relative to $\varpi$ if it is a boundary (respectively, the Shilov boundary) for the seminormed ring $(\mathcal{A}, \vert\cdot\vert_{\spc,\varpi})$. We simply speak of the Shilov boundary for the Tate ring $\mathcal{A}$ when the choice of a topologically nilpotent unit $\varpi$ is understood from the context.\end{mydef}
Motivated by \cite{Berkovich}~Proposition~2.4.4, we introduce the following definition which plays a central role in this paper.
\begin{mydef}[Berkovich's description of the Shilov boundary]\label{Berkovich's description of the Shilov boundary}We say that a Tate ring $\mathcal{A}$ satisfies Berkovich's description of the Shilov boundary if for every irreducible component of $\Spf(\mathcal{A}^{\circ})$ with generic point $\mathfrak{p}$ the pre-image $\spc^{-1}(\{\mathfrak{p}\})$ consists of exactly one point, which is of rank one, and the closure in $\mathcal{M}(\mathcal{A})$ of the subset \begin{equation*}\spc^{-1}(\Spf(\mathcal{A}^{\circ})_{\gen})\subseteq\mathcal{M}(\mathcal{A})\end{equation*}is the Shilov boundary for $\mathcal{A}$. We say that $\mathcal{A}$ satisfies Berkovich's description of the Shilov boundary for $\mathcal{A}$ in the strong sense if, moreover, $\spc^{-1}(\Spf(\mathcal{A}^{\circ})_{\gen})$ is the Shilov boundary.\end{mydef}
In the presence of a seminorm-multiplicative topologically nilpotent unit, we can actually describe any power-multiplicative seminorm in explicit terms. We give this description in the more general setting of monoids. The following definition is inspired by the theory of Tate rings.
\begin{mydef}[Weak pair of definition of a monoid]\label{Weak pair of definition of a monoid}A weak pair of definition of a monoid $X$ is a pair $(X_{0}, \varpi)$ consisting of a subsemigroup $X_{0}$ of $X$ and an element $\varpi\in X_{0}$ which is a unit in $X$ but whose inverse $\varpi^{-1}$ does not belong to $X_{0}$.\end{mydef}
\begin{mydef}[Pair of definition of a monoid]\label{Pair of definition of a monoid}A weak pair of definition $(X_{0}, \varpi)$ of a monoid $X$ is called a pair of definition if \begin{equation*}X=X_{0}[\varpi^{-1}]=\bigcup_{n\in\mathbb{Z}}\varpi^{n}X_{0}.\end{equation*}\end{mydef}
\begin{mydef}[Generalized gauge]\label{Generalized gauge}Let $(X_{0}, \varpi)$ be a pair of definition of a monoid $X$. We define a function $p_{X_{0},\varpi}: X\to\mathbb{R}_{\geq0}$ by \begin{equation*}p_{X_{0},\varpi}(f)=\inf\{\, 2^{-n/m}\mid n\in\mathbb{Z}, m\in\mathbb{Z}_{>0}, f^{m}\in\varpi^{n}X_{0}\,\}\end{equation*}and call it the generalized gauge of $(X_{0}, \varpi)$.\end{mydef}
\begin{rmk}Since for a (weak) pair of definition $(X_{0}, \varpi)$ the element $\varpi^{-1}$ does not belong to $X_{0}$, we always have $p_{X_{0},\varpi}(1)=1$.\end{rmk}
\begin{prop}\label{Properties of the generalized gauge}For every pair of definition $(X_{0}, \varpi)$ of a monoid $X$, the generalized gauge $p_{X_{0},\varpi}$ is a monoid seminorm (i.e., the function $p_{X_{0},\varpi}$ is submultiplicative).\end{prop}
\begin{proof}For any $f, g\in X$ we compute: \begin{align*}p_{X_{0},\varpi}(f)p_{X_{0},\varpi}(g)\\ =\inf\{\, 2^{-n/m}\mid n, m\in\mathbb{Z}, m>0, f^{m}\in\varpi^{n}X_{0}\,\}\cdot\inf\{\, 2^{-n'/m'}\mid n', m'\in\mathbb{Z}, m'>0, g^{m'}\in\varpi^{n'}X_{0}\,\} \\=\inf\{\, 2^{-n/m}2^{-n'/m'}\mid m, n, m', n'\in\mathbb{Z}, m, m'>0, f^{m}\in\varpi^{n}X_{0}, g^{m'}\in\varpi^{n'}X_{0}\,\} \\ =\inf\{\, 2^{-(nm'+n'm)/mm'}\mid n, m, n', m'\in\mathbb{Z}, m, m'>0, f^{m}\in\varpi^{n}X_{0}, g^{m'}\in\varpi^{n'}X_{0}\,\} \\\geq\inf\{\, 2^{-k/l}\mid k, l\in\mathbb{Z}, l>0, (fg)^{l}\in\varpi^{k}X_{0}\,\}=p_{X_{0},\varpi}(fg).\end{align*}where in the last inequality we used that the conditions $f^{m}\in\varpi^{n}X_{0}, g^{m'}\in\varpi^{n'}X_{0}$ imply \begin{equation*}(fg)^{mm'}=(f^{m})^{m'}(g^{m'})^{m}\in (\varpi^{nm'}X_{0})\cdot(\varpi^{n'm}X_{0})=\varpi^{nm'+n'm}X_{0},\end{equation*}for any $f, g\in X$.\end{proof}
We now consider functions on a monoid $X$ which can be described as generalized gauges of pairs of definition of $X$. Note that if $\varpi\in X^{\times}$ is a unit in a monoid $X$ and $\nu: X\to \mathbb{R}_{\geq0}$ is a function such that $\nu(1)=1$ and \begin{equation*}\nu(\varpi^{n}f)\leq2^{-n}\nu(f)\end{equation*}for all $f\in X$ and $n\in\mathbb{Z}$, then we actually have \begin{equation*}\nu(\varpi^{n}f)=2^{-n}\nu(f)\end{equation*}for all $f\in X$ and all integers $n$ because $\nu(f)=\nu(\varpi^{-n}\varpi^{n}f)\leq 2^{n}\nu(\varpi^{n}f)\leq \nu(f)$. In particular, the open unit ball \begin{equation*}X_{\nu<1}=\{\, f\in X\mid \nu(f)<1\,\}\end{equation*}and the closed unit ball \begin{equation*}X_{\nu\leq1}=\{\, f\in X\mid \nu(f)\leq1\,\}\end{equation*}both contain $\varpi$ but do not contain $\varpi^{-1}$. It follows that $(X_{\nu\leq1}, \varpi)$ and $(X_{\nu<1}, \varpi)$ are weak pairs of definition of the monoid $X$ whenever $X_{\nu<1}$ and $X_{\nu\leq1}$ are subsemigroups of $X$ (for example, if $\nu$ is submultiplicative). Moreover, for every $f\in X$ we can find $n\in\mathbb{Z}$ with $\nu(f)<2^{-n}$ and then $\nu(\varpi^{-n}f)=2^{n}\nu(f)<1$, showing that the weak pairs of definition $(X_{\leq1}, \varpi)$ and $(X_{\nu<1}, \varpi)$ are pairs of definition.
\begin{rmk}\label{Pair of definition associated with a weak pair of definition}Note that for any pair of definition $(X_{0}, \varpi)$ of a monoid $X$, it is evident from the definition of the generalized gauge $p_{X_{0},\varpi}$ that $p_{X_{0},\varpi}(\varpi^{n}f)=2^{-n}p_{X_{0},\varpi}(f)$ for all $f\in X$ and all integers $n$. In particular, the above observations apply to $p_{X_{0},\varpi}$, so, for any pair of definition $(X_{0}, \varpi)$, the pairs $(X_{p_{X_{0},\varpi}<1}, \varpi)$ and $(X_{p_{X_{0},\varpi}\leq1}, \varpi)$ are pairs of definition.\end{rmk}       
\begin{lemma}\label{Explicit power-multiplicative seminorms 1}Let $\varpi\in X^{\times}$ be a unit in a monoid $X$ and let $\nu: X\to \mathbb{R}_{\geq0}$ be a function with $\nu(1)=1$ satisfying the following properties \begin{enumerate}[(1)]\item $\nu(\varpi^{n}f)\leq 2^{-n} \nu(f)$ for all $n\in \mathbb{Z}$, $f\in X$. \item $\nu(f)^{m}\leq \nu(f^{m})$ for all $f\in X$ and all integers $m\geq1$. \item The open unit ball $X_{\nu<1}$ (respectively, the closed unit ball $X_{\nu\leq1}$) is a subsemigroup of $X$. \end{enumerate}Then \begin{equation*}\nu(f)\leq p_{X_{\nu<1},\varpi}(f)\end{equation*}for all $f\in X$ (respectively, $\nu(f)\leq p_{X_{\nu\leq1},\varpi}(f)$ for all $f\in X$). \end{lemma}
\begin{proof}We prove the assertions for the open unit ball, the case of the closed unit ball being analogous. Let $f\in X$ and let $s>0$ such that $2^{-n/m}<s$ and $f^{m}\in \varpi^{n}X_{\nu<1}$ for some integers $n, m>0$. Then $\nu(f^{m})<2^{-n}$ by property (1). By property (2) this implies that $\nu(f)<2^{-n/m}$. The desired inequality follows from this since $s>2^{-n/m}$ was arbitrary.\end{proof}
\begin{lemma}\label{Explicit power-multiplicative seminorms 2}Let $\varpi\in X^{\times}$ be a unit in a monoid $X$ and let $\nu: X\to \mathbb{R}_{\geq0}$ be a function with $\nu(1)=1$ satisfying: \begin{enumerate}[(1)] \item $\nu(\varpi^{n}f)\leq 2^{-n}\nu(f)$ for all $n\in\mathbb{Z}$ and $f\in X$. \item $\nu(f^{m})\leq \nu(f)^{m}$ for all $f\in X$ and all integers $m>0$. \item The open unit ball $X_{\nu<1}$ is a subsemigroup of $X$.\end{enumerate}Then \begin{equation*}\nu(f)\geq p_{X_{\nu<1},\varpi}(f)\end{equation*}for all $f\in X$ and analogously for the closed unit ball $X_{\nu\leq1}$ instead of $X_{\nu<1}$ if $X_{\nu\leq1}$ is a subsemigroup.\end{lemma}
\begin{proof}Suppose that $\nu(f)<s$ for some $s>0$. Choose integers $n, m\in\mathbb{Z}$, $m>0$, such that \begin{equation*}\nu(f)<2^{-n/m}<s.\end{equation*}Using property (2) we obtain $\nu(f^{m})<2^{-n}$. Applying property (1), we conclude that $\nu(\varpi^{-n}f^{m})\leq 2^{n}\nu(f^{m})<1$ and, consequently, $f^{m}=\varpi^{n}(\varpi^{-n}f^{m})\in \varpi^{n}X_{\nu<1}$. The assertion follows.\end{proof}
Putting the two lemmas together, we obtain the following statement about power-multiplicative functions on a monoid $X$.
\begin{prop}\label{Description of seminorms}Let $\varpi\in X^{\times}$ be a unit in a monoid $X$ and let $\nu: X\to\mathbb{R}_{\geq0}$ be a power-multiplicative function on $X$ satisfying $\nu(1)=1$ and $\nu(\varpi^{n}f)\leq2^{-n}\nu(f)$ for all $f\in X$ and all integers $n$. Then \begin{equation*}\nu(f)=p_{X_{\nu<1},\varpi}(f)=p_{X_{\nu\leq1},\varpi}(f)\end{equation*}for all $f\in X$ provided that $X_{\nu<1}$ and $X_{\nu\leq1}$ are subsemigroups of $X$. 

In particular, any two power-multiplicative functions $\nu_{1}, \nu_{2}: X\to\mathbb{R}_{\geq0}$ satisfying $\nu_{i}(1)=1$ and $\nu_{i}(\varpi^{n}f)\leq2^{-n}\nu_{i}(f)$ for all $f\in X$, $n\in\mathbb{Z}$, $i=1, 2$, which have the same open or closed unit ball that is a subsemigroup of $X$ must be equal.\end{prop}
\begin{proof}Follows from Lemma~\ref{Explicit power-multiplicative seminorms 1} and Lemma~\ref{Explicit power-multiplicative seminorms 2}.\end{proof}
Conversely, we can show that the generalized gauge of any pair of definition $(X_{0}, \varpi)$ of a monoid $X$ is a power-multiplicative function on $X$.
\begin{prop}\label{Properties of the generalized gauge 2}For every pair of definition $(X_{0}, \varpi)$ of a monoid $X$, the generalized gauge $p_{X_{0},\varpi}$ of $(X_{0},\varpi)$ is power-multiplicative.\end{prop}
\begin{proof}By Proposition~\ref{Properties of the generalized gauge}, $p_{X_{0},\varpi}$ is submultiplicative, so, in particular, $p_{X_{0},\varpi}(f^{m})\leq p_{X_{0},\varpi}(f)^{m}$ for all $f\in X$, $m\in\mathbb{Z}_{>0}$. Conversely, let $m\in\mathbb{Z}_{>0}$, $f\in X$ and $s>0$ such that $p_{X_{0},\varpi}(f^{m})<s$. Then there exist $k, l\in\mathbb{Z}$, $l>0$, such that $f^{ml}\in\varpi^{k}X_{0}$ and $2^{-k/l}<s$. Then, by definition, $p_{X_{0},\varpi}(f)\leq 2^{-k/ml}<s^{1/m}$. \end{proof}
\begin{lemma}\label{Bounds and gauges}Let $\varpi$ be a unit in a monoid $X$ and let $X_{0}, X_{0}^{\prime}$ be two subsemigroups of $X$ such that $(X_{0}, \varpi)$ and $(X_{0}^{\prime}, \varpi)$ are pairs of definition. If there exists an integer $k\geq0$ such that \begin{equation*}X_{0}^{\prime}\subseteq \varpi^{-k}X_{0},\end{equation*}then $p_{X_{0},\varpi}\leq p_{X_{0}^{\prime},\varpi}$.\end{lemma}
\begin{proof}Let $f\in X$ and let $k\geq0$ be an integer with $X_{0}^{\prime}\subseteq \varpi^{-k}X_{0}$. Choose sequences $(m_{i})_{i}\subseteq \mathbb{Z}_{>0}$ and $(n_{i})_{i}\subseteq \mathbb{Z}$ such that $f^{m_{i}}\in \varpi^{n_{i}}X_{0}^{\prime}$ for all $i$ and $2^{-n_{i}/m_{i}}\to p_{X_{0}^{\prime},\varpi}(f)$ from above as $i\to \infty$.  We can arrange (without changing the sequence $(\frac{n_{i}}{m_{i}})_{i}$) that $m_{i}\to \infty$. Then $2^{k/m_{i}}\to 1$ and thus $2^{-(n_{i}-k)/m_{i}}=2^{-n_{i}/m_{i}}2^{k/m_{i}}\to p_{X_{0}^{\prime},\varpi}(f)$. Since $X_{0}^{\prime}\subseteq \varpi^{-k}X_{0}$, we have $f^{m_{i}}\in \varpi^{n_{i}-k}X_{0}$ for all $i$ and, therefore, $p_{X_{0},\varpi}(f)\leq 2^{-(n_{i}-k)/m_{i}}$ for all $i$. It follows that $p_{X_{0},\varpi}(f)\leq \lim_{i\to\infty}2^{-(n_{i}-k)/m_{i}}=p_{X_{0}^{\prime},\varpi}(f)$. \end{proof}
\begin{prop}\label{Properties of the generalized gauge 3}Let $\varpi\in A$ be a non-zero-divisor and non-unit in a ring $A$ and let $I$ be an ideal of $A$ containing $\varpi$. Regard $(A, \varpi)$ as a pair of definition of the underlying multiplicative monoid of $\mathcal{A}=A[\varpi^{-1}]$. Then $p_{I,\varpi}=p_{A,\varpi}=\vert\cdot\vert_{\spc,\varpi}$. In particular, $p_{I,\varpi}$ is a ring seminorm.\end{prop}
\begin{proof}Since $I$ contains $\varpi$, the pair $(I, \varpi)$ is a pair of definition of the underlying multiplicative monoid of $\mathcal{A}$, so $p_{I,\varpi}$ is well-defined. Let $f\in \mathcal{A}$ with $p_{A,\varpi}(f)<1$. Then there exist $n\in\mathbb{Z}$, $m\in \mathbb{Z}_{>0}$ such that $f^{m}\in \varpi^{n}A$ and $2^{-n/m}<1$. Then $\vert f\vert_{\spc,\varpi}\leq \vert \varpi\vert_{\spc,\varpi}^{n/m}=2^{-n/m}<1$. Conversely, if $\vert f\vert_{\spc,\varpi}<1$, then $\lVert f^{m}\rVert_{A,\varpi}<1$ for $m$ large enough, where $\lVert\cdot\rVert_{A,\varpi}$ is (as usual) the canonical extension to $\mathcal{A}$ of the $\varpi$-adic seminorm on $A$. By definition of the $\varpi$-adic seminorm, there exists $n\in \mathbb{Z}_{>0}$ such that $f^{m}\in \varpi^{n}A$, so $p_{A,\varpi}(f)<1$. This shows that $p_{A,\varpi}$ and $\vert\cdot\vert_{\spc,\varpi}$ have the same open unit ball. By Proposition~\ref{Properties of the generalized gauge 2}, $p_{A,\varpi}$ is a power-multiplicative function. We can then conclude by Proposition~\ref{Description of seminorms} that $p_{A,\varpi}=\vert\cdot\vert_{\spc,\varpi}$. Moreover, since $I$ is an ideal of $A$ and $\varpi\in I$, we have \begin{equation*}I\subseteq A\subseteq \varpi^{-1}I.\end{equation*}By Lemma~\ref{Bounds and gauges}, this implies the equality $p_{I,\varpi}=p_{A,\varpi}$. \end{proof}
We conclude this section with some remarks on complete integral closure and power-bounded elements in a Tate ring. Let $A\subseteq B$ be a ring extension. Recall that an element $f\in B$ is said to be almost integral over $A$ if there exists a finitely generated $A$-submodule $I$ of $B$ such that all powers $f^{n}$ of $f$ belong to $I$ (equivalently, there exists a finitely generated $A$-submodule $I$ of $B$ with $A[f]\subseteq I$). If $B=A[\varpi^{-1}]$ for some non-zero-divisor $\varpi$, this is equivalent to the existence of an integer $m\geq0$ such that $f^{n}\in \varpi^{-m}A$ for all integers $n\geq0$, i.e., an element $f\in A[\varpi^{-1}]$ is almost integral over $A$ if and only if it is power-bounded. The set of all elements of $B$ which are almost integral over $A$ is called the complete integral closure of $A$ in $B$; this is a subring of $B$ containing $A$. The ring $A$ is said to be completely integrally closed in $B$ if $A$ is equal to its complete integral closure in $B$. From the definitions, we obtain the following lemma.
\begin{lemma}\label{Completely integrally closed}Let $A$ be a ring and let $\varpi\in A$ be a non-zero-divisor and non-unit. Then the complete integral closure of $A$ in $A[\varpi^{-1}]$ is equal to the ring of power-bounded elements $\mathcal{A}^{\circ}$ of the Tate ring $\mathcal{A}=A[\varpi^{-1}]$. In particular, if $A$ is completely integrally closed in $\mathcal{A}=A[\varpi^{-1}]$, then $A=\mathcal{A}^{\circ}$.\end{lemma}
\begin{rmk}\label{Andre's example}Just as in the classical theory for integral domains, the complete integral closure $\mathcal{A}^{\circ}$ of $A$ in $\mathcal{A}=A[\varpi^{-1}]$ need not in general be completely integrally closed. Since the closed unit ball with respect to any power-multiplicative seminorm on $\mathcal{A}$ is readily seen to be completely integrally closed, this is related to the phenomenon that $\mathcal{A}^{\circ}$ need not coincide with the closed unit ball of the spectral seminorm $\vert\cdot\vert_{\spc,\varpi}$ on $\mathcal{A}$. An example of both phenomena has been exhibited by André in \cite{Andre18}~Remarques~2.3.3(2). On the other hand, if the Tate ring $\mathcal{A}=A[\varpi^{-1}]$ (with topology defined by the pair $(A, \varpi)$) is uniform, then $\mathcal{A}^{\circ}$ coincides with the closed unit ball of $(\mathcal{A}, \vert\cdot\vert_{\spc, \varpi})$, by \cite{Dine22}~Lemma~2.24; thus $\mathcal{A}^{\circ}$ is completely integrally closed in $\mathcal{A}$.\end{rmk}  

\section{The $\arc_{\varpi}$-operation, Shilov boundary and Rees valuations}\label{sec:Rees valuations}

Fix a ring $A$ and an element $\varpi\in A$ which is a non-unit and a non-zero-divisor. We consider the Tate ring $\mathcal{A}=A[\varpi^{-1}]$ which has $(A, \varpi)$ for a pair of definition. If not explicitly stated otherwise, by the Berkovich spectrum $\mathcal{M}(\mathcal{A})$ of $\mathcal{A}$ and by the Shilov boundary for $\mathcal{A}$ we always mean the Berkovich spectrum and Shilov boundary relative to the given non-zero-divisor $\varpi$ of $A$ (which is a topologically nilpotent unit in $\mathcal{A}$). With this set-up, we use the symbols $\mathcal{A}$ and $A[\varpi^{-1}]$ interchangeably. Recall that in this paper we write all valuations on a ring multiplicatively, not additively. Moreover, starting in this section, we view elements of the Berkovich spectrum $\mathcal{M}(\mathcal{A})$ as rank $1$ valuations on $\mathcal{A}$ and denote them by letters $v, w$ to emphasize the valuative point of view. For convenience, we also introduce the following piece of terminology.
\begin{mydef}[$\varpi$-normalized valuation]Let $A$ and $\varpi$ be as above. A rank $1$ valuation $v$ on $A$ (or a rank $1$ valuation $v$ on $\mathcal{A}=A[\varpi^{-1}]$) is called $\varpi$-normalized if $v(\varpi)=\frac{1}{2}$.\end{mydef}
The following lemma shows in particular that every rank $1$ valuation $v$ on $A[\varpi^{-1}]$ which belongs to the Berkovich spectrum $\mathcal{M}(A[\varpi^{-1}])=\mathcal{M}_{\varpi}(A[\varpi^{-1}])$ is $\varpi$-normalized.  
\begin{lemma}\label{Continuous means normalized}For a rank $1$ valuation $v$ on $\mathcal{A}=A[\varpi^{-1}]$ the following are equivalent: \begin{enumerate}[(1)]\item $v(f)\leq1$ for all $f\in A$ and $v(\varpi)=\frac{1}{2}$; \item $v$ is continuous (with respect to the topology on $\mathcal{A}$ defined by the pair of definition $(A, \varpi)$) and $v(\varpi)=\frac{1}{2}$; \item $v$ belongs to $\mathcal{M}(\mathcal{A})$. \end{enumerate}\end{lemma}
\begin{proof}(3)$\Rightarrow$(1): Let $\lVert\cdot\rVert$ be the canonical extension to $\mathcal{A}$ of the $\varpi$-adic seminorm on $A$; note that this is a seminorm on $\mathcal{A}$ satisfying the conditions in Definition~\ref{Berkovich spectrum of a Tate ring}. Then $v(f)\leq\lVert f\rVert\leq1$ for every $v\in\mathcal{M}(\mathcal{A})$ and every $f\in A$. Moreover, for $v\in\mathcal{M}(\mathcal{A})$, the assumption $\lVert \varpi\rVert=\frac{1}{2}$ implies $v(\varpi)\leq\lVert\varpi\rVert=\frac{1}{2}$ and the assumption that $\varpi$ is seminorm-multiplicative implies $\lVert\varpi^{-1}\rVert=\lVert\varpi\rVert^{-1}$, so  $v(\varpi)^{-1}=v(\varpi^{-1})\leq \lVert\varpi\rVert^{-1}=2$. Consequently, $\frac{1}{2}\leq v(\varpi)$ and we conclude that $v(\varpi)=\frac{1}{2}$. 

(2)$\Rightarrow$(3): Follows from Lemma~\ref{Boundedness vs. continuity}. 

(1)$\Rightarrow$(2): To prove that every rank $1$ valuation $v$ satisfying the assumptions in $(1)$ is continuous, it suffices to prove that, for every $r>0$, the subset $\{\, f\in \mathcal{A}\mid v(f)<r\,\}$ is open with respect to the topology on $\mathcal{A}$ defined by the pair of definition $(A, \varpi)$. Given $r>0$, choose $n\geq1$ such that $\lVert\varpi^{n}\rVert=2^{-n}<r$. Then $v(\varpi^{n})=v(\varpi)^{n}=2^{-n}<r$, so the condition $v(f)\leq1$ for all $f\in A$ implies \begin{equation*}\varpi^{n}A\subseteq \{\, f\in \mathcal{A}\mid v(f)<r\,\}.\end{equation*}It follows that $\{\, f\in \mathcal{A}\mid v(f)<r\,\}$ is an open subset of $\mathcal{A}$ for every $r>0$, as desired.\end{proof}
We say that a valuation ring $V$ which is not a field has pseudo-uniformizer $\varpi$ if $\varpi$ is a non-zero non-unit in $V$ and $V$ is $\varpi$-adically separated. By \cite{FK}~Ch.~0, Proposition~6.7.2, this is equivalent to the field of fractions $\Frac(V)$ of $V$ being equal to $V[\varpi^{-1}]$ and then $\varpi$ is a pseudo-uniformizer of the Tate ring $\Frac(V)$ in our usual sense. Recall that every valuation ring $V$ determines (and is determined by) an equivalence class of valuations on the fraction field $\Frac(V)$ of $V$. If $\varpi$ is a pseudo-uniformizer of $V$ and $V$ is of rank $1$, this equivalence class contains a $\varpi$-normalized representative.
\begin{lemma}\label{Valuations and valuation rings}For every valuation $v$ on $A$ with $v(f)\leq1$ for all $f\in A$, there exists a valuation ring $A\to V$ over $A$ such that $v$ is the restriction along $A\to V\hookrightarrow\Frac(V)$ of the (suitably normalized) valuation of $\Frac(V)$ determined by $V$. Moreover, if $v$ extends to a continuous valuation on $A[\varpi^{-1}]$, then there exists $A\to V$ as above such that $\varpi$ is a pseudo-uniformizer of $V$, and if $v$ is a valuation of rank $1$, we can in addition assume that $V$ is a valuation ring of rank $1$.\end{lemma}
\begin{proof}The kernel $\ker(v)$ is a prime ideal and $v$ induces a valuation $\overline{v}$ on the integral domain $A/\ker(v)$; we can then extend this induced valuation to $\Frac(A/\ker(v))$, which corresponds to a valuation ring $V$ of $\Frac(A/\ker(v))$ containing $A/\ker(v)$. If $v$ is of rank $1$, then so is $\overline{v}$ and then the corresponding valuation ring $V$ of $\Frac(A/\ker(v))$ has rank $1$. If $v$ extends to a continuous valuation on $A[\varpi^{-1}]$, then $0<v(\varpi)<1$ and thus $0<\overline{v}(\varpi)<1$. It follows that $\varpi$ belongs to the maximal ideal of the valuation ring $V$ corresponding to the valuation $\overline{v}$ on $\Frac(A/\ker(v))$. We can then set $V'=V/\bigcap_{n>0}\varpi^{n}V$ and then $V'$ is a valuation ring with pseudo-uniformizer $\varpi$ such that the (suitably normalized) valuation on $\Frac(V')$ determined by $V'$ pulls back to the valuation $v$ on $A$.\end{proof}
For $A$ and $\varpi$ as before, we denote by $J_{\varpi}(A)$ the set of all $A$-submodules of $A[\varpi^{-1}]$. For any $I\in J_{\varpi}(A)$ and any valuation $v$ on $A[\varpi^{-1}]$ we write \begin{equation*}I^{v}:=\{\, f\in A[\varpi^{-1}]\mid v(f)\leq v(g)~\textrm{for some}~g\in I\,\}.\end{equation*}Note that $I\subseteq I^{v}$, for every $I\in J_{\varpi}(A)$. This $A$-submodule of $A[\varpi^{-1}]$ can be equivalently described by means of valuation rings over $A$.     
\begin{lemma}\label{Valuative ideals}Let $I\in J_{\varpi}(A)$ and let $v$ be a continuous valuation on $A[\varpi^{-1}]$. Let $A\to V$ be a valuation ring over $A$ with pseudo-uniformizer $\varpi$ such that $v$ is the restriction along $A[\varpi^{-1}]\to \Frac(V)$ of the valuation on $\Frac(V)$ determined by $V$. Then $f\in I^{v}$ if and only if the image of $f$ in $\Frac(V)$ belongs to the $V$-submodule of $\Frac(V)$ generated by the image of $I$.\end{lemma}
\begin{proof}It suffices to prove that for a fractional ideal $I$ of a valuation ring $V$ and for $v$ the valuation on $\Frac(V)$ determined by $V$ we have $f\in I$ if and only if $v(f)\leq v(g)$ for some $g\in I$. But, in this case, if $v(f)\leq v(g)$ for some non-zero $g\in I$, then $v(\frac{f}{g})\leq1$ which means that $\frac{f}{g}\in V$. In particular, $f\in (g)_{V}\subseteq I$.\end{proof}
The following closure operation for $A$-submodules of $A[\varpi^{-1}]$ is an analog of the notion of integral closure of ideals. It is related to the spectral seminorm on $A[\varpi^{-1}]$ in the same way as the integral closure of ideals in a Noetherian ring is related to the asymptotic Samuel function (cf.~\cite{Swanson-Huneke}~Corollary~6.9.1). 

Our name for the closure operation was inspired by the notion of $\arc_{\varpi}$-covers due to Bhatt and Mathew (\cite{Bhatt-Mathew21}~Definition 6.14).
\begin{mydef}[$\arc_{\varpi}$-closure]We define a map \begin{equation*}J_{\varpi}(A)\to J_{\varpi}(A), I\mapsto I^{\arc_{\varpi}},\end{equation*}by setting \begin{equation*}I^{\arc_{\varpi}}\coloneqq\bigcap_{v\in\mathcal{M}(A[\varpi^{-1}])}I^{v}.\end{equation*}For any $I\in J_{\varpi}(A)$, we call $I^{\arc_{\varpi}}$ the $\arc_{\varpi}$-closure of $I$, and we call the map $I\mapsto I^{\arc_{\varpi}}$ the $\arc_{\varpi}$-operation. We say that an $A$-submodule $I$ of $A[\varpi^{-1}]$ is $\arc_{\varpi}$-closed if $I^{\arc_{\varpi}}=I$.\end{mydef}
The $\arc_{\varpi}$-operation is an example of a $\varpi$-$\ast$-operation in the sense of Section~\ref{sec:star-operations} of this paper. However, we do not need this observation right now.
\begin{example}[The $\arc_{\varpi}$-closure of the ring $A$]\label{Complete integral closure and arc-closure}If $I=A$, then \begin{equation*}I^{v}=A^{v}=\{\, f\in \mathcal{A}=A[\varpi^{-1}]\mid v(f)\leq1\,\}\end{equation*}for any $v\in\mathcal{M}(\mathcal{A})$ and thus $A^{\arc_{\varpi}}$ is precisely the closed unit ball $\mathcal{A}_{\vert\cdot\vert_{\spc,\varpi}\leq1}$ of the spectral seminorm $\vert\cdot\vert_{\spc,\varpi}$ on $\mathcal{A}$, by virtue of \cite{Berkovich}~Theorem~1.3.1. If the Tate ring $\mathcal{A}$ is uniform, then this closed unit ball is equal to $\mathcal{A}^{\circ}$, by \cite{Dine22}~Lemma 2.24, or, equivalently, to the complete integral closure of $A$ inside $\mathcal{A}$. In particular, if $A$ is completely integrally closed in $\mathcal{A}$, then $A^{\arc_{\varpi}}=A$. \end{example}
One can think of the notion of $\arc_{\varpi}$-closure as a coarsening of the classical notion of integral closure for finitely generated ideals.
\begin{lemma}\label{Integral closure is contained in the arc-closure}For any ideal $I$ of $A$, the integral closure $\overline{I}$ of $I$ in $A$ is contained in the $\arc_{\varpi}$-closure $I^{\arc_{\varpi}}$.\end{lemma}
\begin{proof}It is a classical fact that $\overline{I}=\bigcap_{V}\varphi_{V}^{-1}(IV)$, where $\varphi_{V}: A\to V$ ranges over all valuation rings over $A$ (when $A$ is a domain, this is \cite{Swanson-Huneke}~Proposition~6.8.2; one can then reduce to the case of domains using \cite{Swanson-Huneke}~Proposition~1.1.5(2)). By Lemma~\ref{Valuative ideals}, this description of integral closure translates into $\overline{I}=\bigcap_{v}I^{v}$, where $v$ ranges over the valuations $v$ on $A$ with $v(f)\leq1$ for all $f\in A$ (if $v(\varpi)=0$, we can still define $I_{v}$ by $I_{v}=\{\, f\in A\mid v(f)\leq v(g)~\textrm{for some}~g\in I \,\}$). Since every element $v\in\mathcal{M}(A[\varpi^{-1}])$ satisfies $v(f)\leq1$ for $f\in A$, we see that $\overline{I}\subseteq I^{\arc_{\varpi}}$. \end{proof}
In the world of Noetherian rings (and under a mild additional assumption), the $\arc_{\varpi}$-closure of ideals coincides with integral closure.
\begin{prop}\label{Integral closure and arc-closure}If $A$ is Noetherian and $\varpi$-adically Zariskian (i.e., $\varpi$ belongs to the Jacobson radical of $A$), then for every ideal $I$ of $A$ we have $I^{\arc_{\varpi}}\cap A=\overline{I}$, the integral closure of the ideal $I$.\end{prop}
\begin{proof}In view of Lemma~\ref{Integral closure is contained in the arc-closure}, there remains only one inclusion to verify. By \cite{Swanson-Huneke}~Proposition~1.1.5, $f\in\overline{I}$ if and only if the image of $f$ in $A/\mathfrak{p}$ is in the integral closure of $(I+\mathfrak{p})/\mathfrak{p}$ for every minimal prime ideal $\mathfrak{p}$ of $A$. Since $\varpi$ is a non-zero-divisor in $A$, it does not belong to any minimal prime ideal $\mathfrak{p}$ of $A$. Since $A$ is $\varpi$-adically Zariskian, so is $A/\mathfrak{p}$ for every minimal prime $\mathfrak{p}$ of $A$. In particular, $\varpi$ is a non-zero non-unit in $A/\mathfrak{p}$ for every $\mathfrak{p}$. Let $\mathfrak{p}$ be some minimal prime ideal of $A$. Since $A/\mathfrak{p}$ is Noetherian, the integral closure of $(I+\mathfrak{p})/\mathfrak{p}$ in $A/\mathfrak{p}$ is equal to the intersection \begin{equation*}\bigcap_{V\in\mathcal{V}^{(\mathfrak{p})}}((I+\mathfrak{p})/\mathfrak{p})V\cap (A/\mathfrak{p}),\end{equation*}where $\mathcal{V}^{(\mathfrak{p})}$ is the family of all discrete rank $1$ valuation rings $V$ between $A/\mathfrak{p}$ and its field of fractions with the property that the maximal ideal of $V$ contracts to a maximal ideal of $A/\mathfrak{p}$ (see \cite{Swanson-Huneke}~Proposition~6.8.4). Since $A/\mathfrak{p}$ is $\varpi$-adically Zariskian, $\varpi$ is a non-unit in every $V\in\mathcal{V}^{(\mathfrak{p})}$. Moreover, since all $V\in\mathcal{V}^{(\mathfrak{p})}$ are of rank $1$, they are $a$-adically separated for every non-unit $a\in V$. In particular, $\varpi$ is a pseudo-uniformizer of each $V\in\mathcal{V}^{(\mathfrak{p})}$. By Lemma~\ref{Valuative ideals}, for any minimal prime ideal $\mathfrak{p}$ and any $V\in\mathcal{V}^{(\mathfrak{p})}$, the pre-image of $((I+\mathfrak{p})/\mathfrak{p})V$ in $A[\varpi^{-1}]$ is equal to $I^{v_{V}}$, where $v_{V}$ is the continuous $\varpi$-normalized rank $1$ valuation on $A$ obtained by pulling back the $\varpi$-normalized rank $1$ valuation on $\Frac(V)$ corresponding to the valuation ring $V$. We conclude that \begin{equation*}\overline{I}=\bigcap_{\mathfrak{p}\in\Min(A)}\bigcap_{V\in\mathcal{V}^{(\mathfrak{p})}}I^{v_{V}}\cap A,\end{equation*}and, in particular, $I^{\arc_{\varpi}}\cap A\subseteq \overline{I}$. \end{proof}  
\begin{prop}\label{Arc-closure of ideals and rings}For any $f\in A$, $n\geq1$, we have \begin{equation*}f\in (\varpi^{n})_{A}^{\arc_{\varpi}}\end{equation*}if and only if \begin{equation*}\frac{f}{\varpi^{n}}\in \mathcal{A}_{\vert\cdot\vert_{\spc,\varpi}\leq1}.\end{equation*}\end{prop}
\begin{proof}For any $v\in\mathcal{M}(\mathcal{A})$ we have $v(f)\leq v(\varpi^{n})$ if and only if $v(\frac{f}{\varpi^{n}})\leq1$. Now apply \cite{Berkovich}~Theorem~1.3.1.\end{proof}
\begin{cor}\label{Arc-closure of ideals and rings 2}The ring $A$ is completely integrally closed in $\mathcal{A}$ if and only if \begin{equation*}(\varpi^{n})_{A}^{\arc_{\varpi}}=(\varpi^{n})_{A}\end{equation*}for some (or equivalently all) $n\geq1$.\end{cor}
\begin{proof}By Proposition~\ref{Arc-closure of ideals and rings}, $f\in (\varpi^{n})_{A}^{\arc_{\varpi}}$ if and only if $\frac{f}{\varpi^{n}}\in \mathcal{A}_{\vert\cdot\vert_{\spc,\varpi}\leq1}$. Thus it remains to prove that $A$ is completely integrally closed in $\mathcal{A}$ if and only if $A$ is equal to the closed unit ball of the spectral seminorm $\vert\cdot\vert_{\spc,\varpi}$ on $\mathcal{A}$.

If $A=\mathcal{A}_{\vert\cdot\vert_{\spc,\varpi}\leq1}$, then it is readily seen that $A$ is completely integrally closed since the spectral seminorm $\vert\cdot\vert_{\spc,\varpi}$ is power-multiplicative. Conversely, if $A$ is completely integrally closed, then $A=\mathcal{A}^{\circ}$ by Lemma~\ref{Completely integrally closed}, so the Tate ring $\mathcal{A}$ is uniform. It then follows from \cite{Dine22}~Lemma~2.24, that $A=\mathcal{A}^{\circ}$ is equal to the closed unit ball of $\vert\cdot\vert_{\spc,\varpi}$.  \end{proof} 
This is similar to the following basic property of integral closure, an analog of \cite{Swanson-Huneke}~Proposition~1.5.2.
\begin{prop}\label{Rings of integral elements}Let $A$, $\varpi$ be as before. Then for every (not necessarily proper) ideal $I$ of $A$ and every $n\geq1$, an element $f\in A$ belongs to the integral closure $\overline{\varpi^{n}\cdot I}$ of $\varpi^{n}\cdot I$ in $A$ if and only if $\frac{f}{\varpi^{n}}$ belongs to the integral closure of $I$ in $A[\varpi^{-1}]$. In particular, $A$ is integrally closed in $A[\varpi^{-1}]$ if and only if $(\varpi^{n})_{A}$ is integrally closed in $A$ for some or equivalently every $n\geq1$.\end{prop}
\begin{proof}Assume that $f$ is integral over $\varpi^{n}\cdot I$. Let \begin{equation*}f^{m}+b_{1}\varpi^{n} f^{m-1}+\dots+b_{m-1}\varpi^{n(m-1)}f+b_{m}\varpi^{nm}=0\end{equation*}be an equation of integral dependence of $f$ over $\varpi^{n}\cdot I$ with $b_{i}\in I^{i}$. Dividing by $\varpi^{nm}$ we obtain \begin{equation*}(\frac{f}{\varpi^{n}})^{m}+b_{1}(\frac{f}{\varpi^{n}})^{m-1}+\dots+b_{m-1}(\frac{f}{\varpi^{n}})+b_{m}=0,\end{equation*}an equation of integral dependence of $\frac{f}{\varpi^{n}}\in A[\varpi^{-1}]$ over $I$. 

For the converse, suppose that $g=\frac{f}{\varpi^{n}}$ is in the integral closure of $I$ in $A[\varpi^{-1}]$. Let \begin{equation*}g^{m}+b_{1}g^{m-1}+\dots+b_{m-1}g+b_{m}=0\end{equation*}be an equation of integral dependence of $g$ over $I$ (i.e., $b_{i}\in I^{i}$). Multiplying by $\varpi^{nm}$, we obtain an equation of integral dependence of $f=\varpi^{n}\cdot g$ over $\varpi^{n}\cdot I$. \end{proof} 
The following corollary of Proposition~\ref{Arc-closure of ideals and rings} can be thought of as analogous to \cite{Swanson-Huneke}~Corollary 6.8.11, or to loc.~cit., Corollary~6.8.12.
\begin{cor}\label{Uniform Tate rings and arc-closure}Suppose that the Tate ring $\mathcal{A}=A[\varpi^{-1}]$ is uniform. For $f\in \mathcal{A}$, $n\geq1$, the element $f$ belongs to the $\arc_{\varpi}$-closure of $(\varpi^{n})_{A}$ if and only if there exists some integer $k\geq0$ such that \begin{equation*}f^{m}\in (\varpi^{nm-k})_{A}\end{equation*}for all integers $m>k$.\end{cor}
\begin{proof}By the assumption that $\mathcal{A}$ is uniform and \cite{Dine22}~Lemma~2.24, the $\arc_{\varpi}$-closure $A^{\arc_{\varpi}}=\mathcal{A}_{\vert\cdot\vert_{\spc,\varpi}\leq1}$ of $A$ coincides with the ring of power-bounded elements $\mathcal{A}^{\circ}$. The assertion follows from this by Proposition \ref{Arc-closure of ideals and rings} and by the definition of a power-bounded element.\end{proof}
We think of the following proposition as analogous to \cite{Swanson-Huneke}~Corollary~6.9.1. While simple to prove, this proposition was our main motivation for introducing the $\arc_{\varpi}$-operation: In fact, we will use it to get an algebraic description of the notion of Shilov boundary for $A[\varpi^{-1}]$ in Theorem~\ref{Boundaries and arc-closure}. 
\begin{prop}\label{Arc-closure and the spectral seminorm}Let $A$, $\varpi$ be arbitrary. For $n\geq1$, an element $f\in \mathcal{A}=A[\varpi^{-1}]$ belongs to the $\arc_{\varpi}$-closure of $(\varpi^{n})_{A}$ if and only if $\vert f\vert_{\spc,\varpi}\leq2^{-n}$.\end{prop}
\begin{proof}By \cite{Berkovich}~Theorem~1.3.1, $\vert f\vert_{\spc,\varpi}=\sup_{v\in\mathcal{M}(\mathcal{A})}v(f)$ for all $f\in \mathcal{A}$. By Lemma~\ref{Continuous means normalized}, $v(\varpi^{n})=v(\varpi)^{n}=2^{-n}$ for all $v\in\mathcal{M}(\mathcal{A})$ and $n\geq1$. It follows that $\vert f\vert_{\spc,\varpi}\leq2^{-n}$ if and only if $v(f)\leq v(\varpi^{n})$ for all $v\in\mathcal{M}(\mathcal{A})$. \end{proof}
We can now state and prove the promised algebraic description of the notion of Shilov boundary. 
\begin{thm}\label{Boundaries and arc-closure}Let $\varpi$ be a non-zero-divisor and a non-unit in a ring $A$ and consider the Tate ring $\mathcal{A}=A[\varpi^{-1}]$. For a subset $\mathcal{S}$ of the Berkovich spectrum $\mathcal{M}(\mathcal{A})$, choose a set of rank $1$ valuation rings over $A$ \begin{equation*}(\varphi_{v}: A\to V_{v})_{v\in\mathcal{S}}\end{equation*}with pseudo-uniformizer $\varpi$, indexed by $\mathcal{S}$, such that each $v$ is the pullback along the map $A\to V_{v}\hookrightarrow\Frac(V_{v})$ of the (suitably normalized) valuation of $\Frac(V_{v})$ determined by $V_{v}$ (cf.~Lemma~\ref{Valuations and valuation rings}). Then the following are equivalent: \begin{enumerate}[(1)]\item The closure of $\mathcal{S}$ in $\mathcal{M}(\mathcal{A})$ is a boundary for $\mathcal{A}$, \item \begin{equation*}\mathcal{A}_{\vert\cdot\vert_{\spc,\varpi}\leq1}=\bigcap_{v\in\mathcal{S}}\varphi_{v}^{-1}(V_{v}),\end{equation*}\item \begin{equation*}(\varpi^{n})_{A}^{\arc_{\varpi}}=\bigcap_{v\in\mathcal{S}}\varphi_{v}^{-1}(\varpi^{n}V_{v})\end{equation*}holds for some integer $n\geq0$, \item \begin{equation*}(\varpi^{n})_{A}^{\arc_{\varpi}}=\bigcap_{v\in\mathcal{S}}\varphi_{v}^{-1}(\varpi^{n}V_{v})\end{equation*}holds for all integers $n\geq0$. \end{enumerate}\end{thm}
\begin{proof}The implication $(1)\Rightarrow (2)$ holds by the definition of a boundary and the implication $(2)\Rightarrow (3)$ is trivial. It remains to prove that (2) implies (4), that (3) implies (2) and that (4) implies (1).

(2)$\Rightarrow$(4): Suppose that (2) holds and that $f\in \mathcal{A}$ satisfies $\varphi_{v}(f)\in \varpi^{n}V_{v}$ for all $v\in\mathcal{S}$, for some $n\geq1$. Then $\varphi_{v}(\frac{f}{\varpi^{n}})\in V$ for all $v\in\mathcal{S}$ and thus $\vert\frac{f}{\varpi^{n}}\vert_{\spc,\varpi}\leq1$ by (2). This means that $v(f)\leq v(\varpi^{n})$ for all $v\in\mathcal{M}(\mathcal{A})$, i.e., $f\in (\varpi^{n})_{A}^{\arc_{\varpi}}$.

(3)$\Rightarrow$(2): Suppose that (3) holds for some $n\geq0$ and let $f\in \mathcal{A}$ satisfy $\varphi_{v}(f)\in V_{v}$ for all $v\in\mathcal{S}$. Then $\varphi_{v}(\varpi^{n}f)\in \varpi^{n}V_{v}$ for all $v\in \mathcal{S}$. By (3), this implies that $\varpi^{n}f\in (\varpi^{n})_{A}^{\arc_{\varpi}}$. By Proposition~\ref{Arc-closure of ideals and rings}, this is equivalent to $f\in \mathcal{A}_{\vert\cdot\vert_{\spc,\varpi}\leq1}$, proving (2).

(4)$\Rightarrow$(1): Suppose that $\mathcal{S}$ satisfies condition (4) in the theorem. We want to prove that the closure $\overline{\mathcal{S}}$ of $\mathcal{S}$ is a boundary for $\mathcal{A}$. That is, we want to prove that every element of $\mathcal{A}$, viewed as a function $\mathcal{M}(\mathcal{A})\to\mathbb{R}_{\geq0}$ attains its maximum on $\overline{\mathcal{S}}$. It suffices to prove that \begin{equation*}\vert f\vert_{\spc,\varpi}=\sup_{v\in \mathcal{S}}v(f)\end{equation*}for all $f\in \mathcal{A}$. It suffices to prove this for all $f\in A$. To this end, consider $f\in A$ with $\vert f\vert_{\spc,\varpi}\neq0$ and let $r>0$ be such that $\vert f\vert_{\spc,\varpi}>r$. We prove that there exists $v\in\mathcal{S}$ such that $v(f)>r$. Choose $n\in \mathbb{Z}$, $m\in\mathbb{Z}_{>0}$ such that \begin{equation*}2^{-n/m}\in (r, \vert f\vert_{\spc, \varpi}).\end{equation*}Then $\vert f^{m}\vert_{\spc,\varpi}>2^{-n}$. By Proposition~\ref{Arc-closure and the spectral seminorm}, this means that \begin{equation*}f^{m}\not\in (\varpi^{n})_{A}^{\arc_{\varpi}}.\end{equation*}By our assumption on $\mathcal{S}$ this implies that there exists some $v\in\mathcal{S}$ such that \begin{equation*}\varphi_{v}(f^{m})\not\in \varpi^{n} V_{v}.\end{equation*}But, in view of Lemma~\ref{Continuous means normalized}, this means that the valuation $v$ on $\mathcal{A}$ satisfies $v(f^{m})>v(\varpi^{n})=2^{-n}$ and thus $v(f)>2^{-n/m}>r$, as claimed.\end{proof}
Recall (for example, from \cite{Swanson-Huneke}~Ch.~10) the notions of Rees valuation rings and Rees valuations of an ideal in a ring. By a theorem of Rees (\cite{Swanson-Huneke}~Theorem~10.2.2), every ideal in a Noetherian ring has a (necessarily finite) set of Rees valuation rings and, by \cite{Swanson-Huneke}~Theorem~10.1.6, the set of Rees valuation rings is uniquely determined by the ideal in question. Theorem~\ref{Boundaries and arc-closure} allows us to show that, for the ideal generated by $\varpi$, this set can be naturally identified with the Shilov boundary for $\mathcal{A}=A[\varpi^{-1}]$.    
\begin{thm}\label{Rees valuations}Let $\varpi$ be a non-zero-divisor and a non-unit in a Noetherian ring $A$. Then the Shilov boundary for the Tate ring $\mathcal{A}=A[\varpi^{-1}]$ is given by the set $\mathcal{RV}(\varpi)$ of $\varpi$-normalized Rees valuations of $(\varpi)_{A}$.\end{thm}
\begin{proof}Let $A^{\Zar}$ be the localization of $A$ at the multiplicative subset $1+(\varpi)_{A}$. Since every $v\in \mathcal{RV}(\varpi)$ contains $\varpi$ in its center, the maps $A\to V_{v}$ (where $V_{v}$ is the valuation ring of $v\in \mathcal{RV}(\varpi)$) factor through $A\to A^{\Zar}$. Hence the map $A\to A^{\Zar}$ induces a bijection between sets of ($\varpi$-normalized) Rees valuations of $(\varpi)_{A}$ and $(\varpi)_{A^{\Zar}}$, respectively. Since the $\varpi$-adic completion map $A\to \widehat{A}$ factors through $A\to A^{\Zar}$, we also see that the Berkovich spectrum and the Shilov boundary for $\mathcal{A}=A[\varpi^{-1}]$ are invariant under replacing $\mathcal{A}$ with $\mathcal{A}^{\Zar}=A^{\Zar}[\varpi^{-1}]$. Consequently, we may replace $A$ with $A^{\Zar}$ in our proof and assume that $A$ is $\varpi$-adically Zariskian. 

Since $A$ is Noetherian, Proposition~\ref{Integral closure and arc-closure} then implies that $(\varpi^{n})_{A}^{\arc_{\varpi}}\cap A$ coincides with the integral closure of $(\varpi^{n})_{A}$ in $A$, for all $n\geq1$. By Theorem~\ref{Boundaries and arc-closure} and the definition of Rees valuation rings, we then know that $\mathcal{RV}(\varpi)$ is a boundary. Since $\mathcal{RV}(\varpi)$ is a finite set, it must be a closed boundary, $\mathcal{M}(\mathcal{A})$ being a Hausdorff space. But no proper subset of $\mathcal{RV}(\varpi)$ can be a boundary, by Theorem~\ref{Boundaries and arc-closure} and the minimality property of the set of Rees valuation rings (property (3) in \cite{Swanson-Huneke}~Definition~10.1.1). We conclude by \cite{Guennebaud}~Ch.~1, Proposition~4, that $\mathcal{RV}(\varpi)$ is the Shilov boundary for $\mathcal{A}$.\end{proof}
\begin{cor}\label{Finite Shilov boundary}If $\mathcal{A}$ is a Tate ring which admits a Noetherian ring of definition, then the Shilov boundary for $\mathcal{A}$ is finite.\end{cor}
\begin{proof}By assumption, $\mathcal{A}$ is of the form $\mathcal{A}=A[\varpi^{-1}]$ for some pair of definition $(A, \varpi)$ with $A$ Noetherian. Hence the assertion follows from Theorem~\ref{Rees valuations} and finiteness of the set of Rees valuations of any ideal in a Noetherian ring.\end{proof}
The assertion of Theorem~\ref{Boundaries and arc-closure} also motivates us to consider the following more general notion of bounding sets for Tate Huber pairs.
\begin{mydef}[Bounding set for a Tate Huber pair]\label{Bounding set}Let $(\mathcal{A}, \mathcal{A}^{+})$ be a Tate Huber pair and let $\varpi\in \mathcal{A}^{+}$ be a topologically nilpotent unit in $\mathcal{A}$ contained in $\mathcal{A}^{+}$. A subset $\mathcal{S}$ of the adic spectrum $\Spa(\mathcal{A}, \mathcal{A}^{+})$ is called a bounding set for $(\mathcal{A}, \mathcal{A}^{+})$ (relative to $\varpi$) if for every $f\in \mathcal{A}$ and $n\in\mathbb{Z}$ with $f\not\in\varpi^{n}\mathcal{A}^{+}$, there exists a continuous valuation $v\in\mathcal{S}$ with $v(f)>v(\varpi^{n})$. If there exists a unique minimal constructible-closed bounding set for $(\mathcal{A}, \mathcal{A}^{+})$ (relative to $\varpi$), we call it the Shilov boundary for $(\mathcal{A}, \mathcal{A}^{+})$ (relative to $\varpi$). \end{mydef}
We use the term ``bounding set'' rather than ``boundary'' here since for sets of rank-$1$ valuations $\mathcal{S}$ the above condition corresponds to every $f\in \mathcal{A}$ assuming its supremum on $\mathcal{S}$, but not necessarily a maximum. In the same vein, if $\mathcal{A}$ is uniform, all dense subsets of $\mathcal{M}(\mathcal{A})$, viewed as subsets of $\Spa(\mathcal{A}, \mathcal{A}^{\circ})$, are bounding sets for $(\mathcal{A}, \mathcal{A}^{\circ})$ in the sense of the above definition. For the smallest \textit{constructible-closed} bounding set, we still chose the term ``Shilov boundary'' to emphasize the analogy with the usual notion of Shilov boundary. 

It turns out that the term ``relative to $\varpi$'' can be omitted from the above definition.
\begin{lemma}\label{Boundaries and subrings}For a subset $\mathcal{S}$ of the adic spectrum $\Spa(\mathcal{A}, \mathcal{A}^{+})$ of a Tate Huber pair $(\mathcal{A}, \mathcal{A}^{+})$, the following are equivalent: \begin{enumerate}[(1)]\item For every $f\in \mathcal{A}$ with $f\not\in\mathcal{A}^{+}$, there exists a continuous valuation $v\in\mathcal{S}$ with $v(f)>1$. \item $\mathcal{S}$ is a bounding set for $(\mathcal{A}, \mathcal{A}^{+})$ relative to some topologically nilpotent unit $\varpi$ of $\mathcal{A}$ contained in $\mathcal{A}^{+}$. \item $\mathcal{S}$ is a bounding set for $(\mathcal{A}, \mathcal{A}^{+})$ relative to any topologically nilpotent unit $\varpi$ of $\mathcal{A}$ contained in $\mathcal{A}^{+}$.\end{enumerate}\end{lemma}
\begin{proof}It suffices to show that (1) implies (3). Assume (1) and let $\varpi$ be a topologically nilpotent unit of $\mathcal{A}$ contained in $\mathcal{A}^{+}$ and let $f\in\mathcal{A}$ and $n\in\mathbb{Z}$ be such that $f\not\in \varpi^{n}\mathcal{A}^{+}$. Then $\varpi^{-n}f\not\in\mathcal{A}^{+}$ and, by (1), there exists some $v\in\mathcal{S}$ with $v(\varpi^{-n})v(f)=v(\varpi^{-n}f)>1$. But this is equivalent to $v(f)>v(\varpi^{n})$.\end{proof}
\begin{prop}\label{Minimal closed boundaries}For every Tate Huber pair $(\mathcal{A}, \mathcal{A}^{+})$, there exist minimal constructible-closed bounding sets for $(\mathcal{A}, \mathcal{A}^{+})$. \end{prop}
\begin{proof}We imitate the existence proof in \cite{Gelfand-Raikov-Shilov}~Chapter~II, \S11, proof of Theorem~1. Note that the set of constructible-closed bounding sets is always non-empty, since $\Spa(\mathcal{A}, \mathcal{A}^{+})$ is a constructible-closed bounding set. By Zorn's lemma, it suffices to prove that for every strictly decreasing net of constructible-closed bounding sets $(\mathcal{S}_{n})_{n}$ the intersection $\mathcal{S}=\bigcap_{n}\mathcal{S}_{n}$ is a bounding set (for notational ease, we assume that the index set is the positive integers, but the proof for more general directed sets is virtually the same). Note that $\mathcal{S}$ is non-empty since the constructible topology on the spectral space $\Spa(\mathcal{A}, \mathcal{A}^{+})$ is compact Hausdorff. For an element $f\in\mathcal{A}$ with $f\notin\mathcal{A}^{+}$ choose elements $v_{n}\in\mathcal{S}_{n}$ such that $v_{n}(f)>1$ for all $n$. If one of $v_{n}$ lies in $\mathcal{S}$, we are done. If not, then, for every $n$, there exists $m_{n}$ such that $v_{n}\in\mathcal{S}_{n}\setminus\mathcal{S}_{n+m_{n}}$. Possibly after passing to subsequences of $(\mathcal{S}_{n})_{n}$ and $(v_{n})_{n}$, we obtain a sequence $(v_{n})_{n}$ such that $v_{n}(f)>1$ and $v_{n}\in\mathcal{S}_{n}\setminus\mathcal{S}_{n+1}$ for all $n$. Up to passing to a subsequence again, we may assume that $(v_{n})_{n}$ converges to some $v\in \Spa(\mathcal{A}, \mathcal{A}^{+})$ with respect to the constructible topology on $\Spa(\mathcal{A}, \mathcal{A}^{+})$. Note that $v\in\mathcal{S}$ (if not, there exists $n$ with $v\not\in\mathcal{S}_{n}$, but then the complement of $\mathcal{S}_{n}$ is an open neighborhood of $v$ which does not contain $v_{m}$ for any $m\geq n$, a contradiction). If $v(f)\leq1$, then $\{\, w\mid w(f)\leq1\,\}$ is an open neighborhood of $v$ and thus must contain some $v_{n}$, which contradicts the assumption that $v_{n}(f)>1$ for all $n$. It follows that $v$ is an element of $\mathcal{S}$ with $v(f)>1$, which proves that $\mathcal{S}$ is a bounding set.\end{proof}
We do not know whether a minimal constructible-closed boundary for a Tate Huber pair $(\mathcal{A}, \mathcal{A}^{+})$ is always unique (the proof of uniqueness in \cite{Gelfand-Raikov-Shilov}~Chapter~II, \S11, proof of Theorem~1, does not seem to carry over to our situation). We conclude this section by proving an analog of Theorem~\ref{Boundaries and arc-closure} for integral closure.
\begin{thm}\label{Boundaries and integral closure}Let $\varpi$ be a non-zero-divisor and non-unit in a ring $A$ and consider the Tate Huber pair $(\mathcal{A}, \mathcal{A}^{+})$, where $\mathcal{A}$ is the Tate ring $A[\varpi^{-1}]$ and where $\mathcal{A}^{+}$ is the integral closure of $A$ inside $\mathcal{A}$. For a subset $\mathcal{S}$ of the adic spectrum $\Spa(\mathcal{A}, \mathcal{A}^{+})$ choose a set of valuation rings over $A$ \begin{equation*}(\varphi_{v}: A\to V_{v})_{v\in\mathcal{S}}\end{equation*}with pseudo-uniformizer $\varpi$, indexed by $\mathcal{S}$, such that each $v$ is the pullback along the map $A\to V_{v}\hookrightarrow \Frac(V_{v})$ of the (suitably normalized) valuation of $\Frac(V_{v})$ determined by $V_{v}$ (cf.~Lemma~\ref{Valuations and valuation rings}). Then the following are equivalent: \begin{enumerate}[(1)]\item $\mathcal{S}$ is a bounding set for $(\mathcal{A}, \mathcal{A}^{+})$, \item \begin{equation*}\mathcal{A}^{+}=\bigcap_{v\in\mathcal{S}}\varphi_{v}^{-1}(V_{v}),\end{equation*}\item \begin{equation*}\overline{(\varpi^{n})_{A}}=\bigcap_{v\in\mathcal{S}}\varphi_{v}^{-1}(\varpi^{n}V_{v})\cap A\end{equation*}(where the bar denotes integral closure of the ideal) holds for all integers $n\geq1$. \end{enumerate}\end{thm}
\begin{proof}The equivalence of (1) and (2) follows from Lemma~\ref{Boundaries and subrings}. For the rest of the proof, we will use the equality \begin{equation*}\overline{(\varpi^{n})_{A}}=(\varpi^{n})_{\mathcal{A}^{+}}\cap A, n>0,\end{equation*}which is a consequence of Proposition \ref{Rings of integral elements} and \cite{Swanson-Huneke}~Proposition~1.6.1.

(2)$\Rightarrow$(3): Suppose that (2) holds and let \begin{equation*}f\in \bigcap_{v\in\mathcal{S}}\varphi_{v}^{-1}(\varpi^{n}V_{v})\cap A\end{equation*}for some integer $n\geq1$. Then $\varpi^{-n}f\in \mathcal{A}^{+}$ and $f\in A$. It follows that $f$ is an element of $(\varpi^{n})_{\mathcal{A}^{+}}\cap A=\overline{(\varpi^{n})_{A}}$. 

(3)$\Rightarrow$(2): Assume (3) and let $f\in \bigcap_{v\in\mathcal{S}}\varphi_{v}^{-1}(V_{v})$. Let $n\geq1$ be such that $\varpi^{n}f\in A$. By (3), $\varpi^{n}f\in \overline{(\varpi^{n})_{A}}\subseteq (\varpi^{n})_{\mathcal{A}^{+}}$. In particular, $f\in \mathcal{A}^{+}$, as desired.\end{proof}  

\section{Fractional ideals and $\varpi$-valuative rings}\label{sec:valuative rings}

Fix a ring $A$ and a non-zero-divisor $\varpi$ of $A$.
\begin{mydef}[$\varpi$-fractional ideal]\label{Fractional ideals}A $\varpi$-fractional ideal $I$ of $A$ is an open and bounded $A$-submodule $I$ of the Tate ring $A[\varpi^{-1}]$. In other words, $I$ is an $A$-submodule of $A[\varpi^{-1}]$ such that there exist integers $n, m>0$ with \begin{equation*}\varpi^{n}\in I\subseteq \varpi^{-m}A.\end{equation*}A $\varpi$-fractional ideal is called integral if it is contained in $A$ (and is thus just a usual ideal of $A$ which is open with respect to the $\varpi$-adic topology).\end{mydef}
\begin{rmk}The above definition is a special case of the general notion of $R$-fractional ideals of a ring extension $A\subseteq R$ introduced by Knebusch and Kaiser (\cite{Knebusch-Zhang2}~Ch.~3, \S4, Definition 5). Indeed, for a ring extension $A\subseteq R$, Knebusch and Kaiser call an $A$-submodule $I$ of $R$ an $R$-fractional ideal if there exists an ideal $\mathfrak{a}$ of $A$ and an $A$-submodule $\mathfrak{a}^{-1}$ of $R$ such that \begin{equation*}\mathfrak{a}\mathfrak{a}^{-1}=A\end{equation*}and\begin{equation*}\mathfrak{a}\subseteq I\subseteq \mathfrak{a}^{-1}.\end{equation*}But in the special case $R=A[\varpi^{-1}]$ for a non-zero-divisor $\varpi\in A$, any $A$-submodule $\mathfrak{a}$ of $A[\varpi^{-1}]$ for which $\mathfrak{a}^{-1}$ exists is necessarily $\varpi$-adically open (as $\mathfrak{a}A[\varpi^{-1}]\supseteq \mathfrak{a}\mathfrak{a}^{-1}=A$ and, consequently, $\mathfrak{a}A[\varpi^{-1}]=A[\varpi^{-1}]$) and thus its inverse \begin{equation*}\mathfrak{a}^{-1}=A:_{A[\varpi^{-1}]}\mathfrak{a}\end{equation*}is necessarily bounded in the Tate ring $A[\varpi^{-1}]$. We call $A$-submodules $\mathfrak{a}$ of $A[\varpi^{-1}]$ as above ``invertible $A$-submodules'', see Definition~\ref{Invertible submodule} below. \end{rmk}  
\begin{example}If $\varpi$ is a unit in $A$, then the unit ideal is the only $\varpi$-fractional ideal of $A$.\end{example}
\begin{example}If $A$ is a valuation ring and $\varpi\in A$ is a non-unit such that $A$ is $\varpi$-adically separated, then $A[\varpi^{-1}]$ is the fraction field of $A$ (\cite{FK}~Ch.~0, Proposition 6.7.2), so every non-zero $A$-submodule $I$ of $A[\varpi^{-1}]$ satisfies $IA[\varpi^{-1}]=A[\varpi^{-1}]$ and is thus open. Moreover, if $I$ is a proper $A$-submodule of $A[\varpi^{-1}]$, we can choose $n\geq1$ with $\frac{1}{\varpi^{n}}\not\in I$, so $\frac{1}{\varpi^{n}f}\not\in A$ for any $f\in I\setminus\{0\}$. Since $A$ is a valuation ring, this entails that $\varpi^{n}f\in A$ for all $f\in I$, i.e., $\varpi^{n}\in A:_{A[\varpi^{-1}]}I$. Then \begin{equation*}\varpi^{n}(A:_{A[\varpi^{-1}]}(A:_{A[\varpi^{-1}]}I))\subseteq A,\end{equation*}so \begin{equation*}\varpi^{n}I\subseteq \varpi^{n}(A:_{A[\varpi^{-1}]}(A:_{A[\varpi^{-1}]}I))\subseteq A,\end{equation*}proving that $I$ is bounded. Thus every non-zero proper $A$-submodule $I$ of $A[\varpi^{-1}]$ is a $\varpi$-fractional ideal and $\varpi$-fractional ideals are the same as fractional ideals of $A$. \end{example}
\begin{mydef}[$\varpi$-residual]\label{Residual}Let $\varpi$ be a non-zero-divisor in a ring $A$. For two $A$-submodules $I, J$ of $A[\varpi^{-1}]$ we define the $\varpi$-residual of $I, J$ to be the $A$-submodule 
\begin{equation*}I:_{A[\varpi^{-1}]}J:=\{\, f\in A[\varpi^{-1}]\mid fJ\subseteq I\,\}.\end{equation*}\end{mydef}
\begin{lemma}\label{Residual is a fractional ideal}For any two $\varpi$-fractional ideals $I, J$ the $\varpi$-residual $I:_{A[\varpi^{-1}]}J$ is a $\varpi$-fractional ideal.\end{lemma}
\begin{proof}Choose integers $m, n, k, l>0$ such that \begin{equation*}\varpi^{n}A\subseteq I\subseteq \varpi^{-m}A\end{equation*}and \begin{equation*}\varpi^{k}A\subseteq J\subseteq \varpi^{-l}A.\end{equation*}Then $\varpi^{n+l}\in I:_{A[\varpi^{-1}]}J$ and $\varpi^{m+k}(I:_{A[\varpi^{-1}]}J)\subseteq A$.\end{proof}
For any $A$, $\varpi$, we denote by $J_{\varpi}(A)$ the set of $A$-submodules of $A[\varpi^{-1}]$ and we denote by $\mathcal{F}_{\varpi}(A)$ the set of $\varpi$-fractional ideals of $A$. We have just observed that $\mathcal{F}_{\varpi}(A)$ is always closed under $\varpi$-residuals. Note that each of $J_{\varpi}(A)$ and $\mathcal{F}_{\varpi}(A)$ is also closed under finite sums, intersections and products and forms a semigroup whose multiplication operation is the product of $A$-submodules of $A[\varpi^{-1}]$.
\begin{mydef}[Cancellative $A$-submodules]An $A$-submodule $I$ of $A[\varpi^{-1}]$ is called cancellative if it is cancellative as an element of the semigroup $J_{\varpi}(A)$. \end{mydef}
\begin{lemma}\label{Products of fractional ideals}For any $I, J\in J_{\varpi}(A)$ we have $I(IJ:_{A[\varpi^{-1}]}I)=IJ$.\end{lemma}
\begin{proof}Clearly, $J\subseteq IJ:_{A[\varpi^{-1}]}I$, so $IJ\subseteq I(IJ:_{A[\varpi^{-1}]}I)$. For the converse, let $f_1,\dots, f_n\in A[\varpi^{-1}]$, $g_1,\dots, g_{n}\in I$ such that $f_{i}I\subseteq IJ$ for all $i$. Then $\sum_{i=1}^{n}f_{i}g_{i}\in IJ$. \end{proof}
The following proposition is an analog of \cite{Fuchs-Salce}~Ch.~1, Proposition~2.1. 
\begin{prop}\label{Properties of cancellative fractional ideals}For $I\in J_{\varpi}(A)$, the following are equivalent: \begin{enumerate}[(1)]\item $I$ is cancellative. \item $IJ:_{A[\varpi^{-1}]}I=J$ for all $J\in J_{\varpi}(A)$. \item For $J, K\in J_{\varpi}(A)$, $IJ\subseteq IK$ implies $J\subseteq K$. \end{enumerate} \end{prop}
\begin{proof}(1)$\Rightarrow$(2): Follows from the above lemma. 

(2)$\Rightarrow$(3): If $IJ\subseteq IK$, then $J=IJ:_{A[\varpi^{-1}]}I\subseteq IK:_{A[\varpi^{-1}]}I=K$. 

(3)$\Rightarrow$ (1): By (anti-)symmetry. \end{proof}
The following is an analog of \cite{Fuchs-Salce}~Ch.~1, Proposition~2.3.
\begin{prop}\label{Cancellative fractional ideals and maximal ideals}If $I$ is a cancellative $A$-submodule of $A[\varpi^{-1}]$ and $\mathfrak{m}_{1},\dots, \mathfrak{m}_{k}$ are distinct maximal ideals of $A$, then \begin{equation*}\bigcup_{i=1}^{k}\mathfrak{m}_{i}I\subsetneq I.\end{equation*}\end{prop}
\begin{proof}We have $\prod_{i\neq j}\mathfrak{m}_{i}\not\subseteq \mathfrak{m}_{j}$ and thus \begin{equation*}\bigcap_{i\neq j}\mathfrak{m}_{i}\not\subseteq\mathfrak{m}_{j}.\end{equation*}Then, since $I$ is cancellative,\begin{equation*}I(\bigcap_{i\neq j}\mathfrak{m}_{i})\not\subseteq I\mathfrak{m}_{j}.\end{equation*}Thus, for each $j$, we can select $f_{j}\in I(\bigcap_{i\neq j}\mathfrak{m}_{i})\setminus I\mathfrak{m}_{j}$. Then the element \begin{equation*}f=\sum_{j=1}^{k}f_{j}\end{equation*}satisfies $f\in I$ but $f\not\in \bigcup_{i=1}^{k}I\mathfrak{m}_{i}$. \end{proof}
Besides the notion of cancellative $A$-submodules of $A[\varpi^{-1}]$, there is the stronger notion of invertible $A$-submodules.
\begin{mydef}[Invertible $A$-submodule]\label{Invertible submodule}An $A$-submodule $I$ of $A[\varpi^{-1}]$ is said to be invertible if there exists an $A$-submodule $J$ of $A[\varpi^{-1}]$ with $IJ=A$.\end{mydef}
\begin{rmk}This is the special case $R=A[\varpi^{-1}]$ (for a non-zero-divisor $\varpi\in A$) of the notion of $R$-invertible $A$-submodule introduced by Knebusch and Zhang in \cite{Knebusch-Zhang}~Ch.~II, \S1, Definition~3. Namely, for an arbitrary ring extension $A\subset R$, Knebusch and Zhang call an $A$-submodule $I$ of $R$ an $R$-invertible submodule if there exists an $A$-submodule $J$ of $R$ such that $IJ=A$.\end{rmk} 
\begin{rmk}\label{Bullet points on invertible ideals}~\begin{itemize}\item Invertible $A$-submodules of $A[\varpi^{-1}]$ are cancellative. \item The $A$-submodule $J$ of $A[\varpi^{-1}]$ with $IJ=A$ is uniquely determined by $I$; this $J$ is called the inverse $I^{-1}$ of $I$. \item An open $A$-submodule $I$ of the Tate ring $A[\varpi^{-1}]$ is invertible in the sense of Definition~\ref{Invertible submodule} if and only if it is invertible as an $A$-module (i.e., as an $A$-submodule of the total ring of fractions $\Frac(A)$ of $A$). Indeed, if $I$ is invertible as an $A$-module and $\varpi^{n}\in I$, then $A:_{\Frac(A)}I\subseteq \varpi^{-n}A$, so $A:_{\Frac(A)}I$ is naturally an $A$-submodule of $A[\varpi^{-1}]$ and thus equal to $A:_{A[\varpi^{-1}]}I$.\end{itemize}\end{rmk}
\begin{example}\label{Regular principal implies invertible}Every $\varpi$-adically open principal $A$-submodule $I$ of $A[\varpi^{-1}]$ is invertible. Indeed, choosing $f\in A[\varpi^{-1}]$ with $I=(f)_{A}$, the openness of $I$ entails that $f$ is a unit in $A[\varpi^{-1}]$ and then $(f^{-1})_{A}$ is an $A$-submodule of $A[\varpi^{-1}]$ with $(f)_{A}(f^{-1})_{A}=A$.\end{example}
\begin{example}[Cancellative submodules which are not invertible]Let $A$, $\varpi$ be arbitrary, with the only condition that $A[\varpi^{-1}]$ is not equal to the total ring of fractions $\Frac(A)$. Let $f\in A[\varpi^{-1}]$ be any non-zero-divisor and non-unit (so, the principal $A$-submodule $(f)_{A}$ is not $\varpi$-adically open). Then $(f)_{A}$ is a cancellative $A$-submodule of $A[\varpi^{-1}]$ (since it is invertible as an $A$-submodule of $\Frac(A)$), but it is not invertible as an $A$-submodule of $A[\varpi^{-1}]$ since $f^{-1}\in\Frac(A)$ does not belong to $A[\varpi^{-1}]$.\end{example} 
\begin{lemma}\label{Invertible implies regular}Every invertible $A$-submodule of $A[\varpi^{-1}]$ is a $\varpi$-fractional ideal.\end{lemma}
\begin{proof}Choose an $A$-submodule $J$ of $A[\varpi^{-1}]$ such that $IJ=A$. Then $IA[\varpi^{-1}]\supseteq IJ=A$. Then, for any $f\in A$ and $n\geq1$, we have $\varpi^{-n}f=\varpi^{-n}(\sum_{i=1}^{m}a_{i}f_{i})$ with $a_1,\dots, a_m\in I$ and $f_1,\dots, f_m\in A[\varpi^{-1}]$, so \begin{equation*}\varpi^{-n}f=\sum_{i=1}^{m}a_{i}\varpi^{-n}f_{i}\in IA[\varpi^{-1}].\end{equation*}It follows that $I$ is open in the Tate ring $A[\varpi^{-1}]$. By symmetry, $J$ is also open and thus $I$ is bounded.\end{proof}
Now we prove an analog of \cite{Fuchs-Salce}~Ch.~1, Proposition 2.5, in the setting of $\varpi$-fractional ideals. Our proof follows the proof of loc.~cit.  
\begin{prop}\label{Properties of invertible fractional ideals}Let $I$ be an invertible $\varpi$-fractional ideal of $A$. Then: \begin{enumerate}[(1)]\item $I^{-1}=A:_{A[\varpi^{-1}]}I$. \item $I$ is finitely generated. \item If $A$ is semilocal, then $I$ is a principal $\varpi$-fractional ideal. Moreover, if $A$ is local, then every generating set of $I$ contains an element generating $I$. \item If $I\subseteq A$ and $\mathfrak{p}$ is a minimal prime ideal of $I$, then every generating set of $I$ contains an element $z$ such that $\mathfrak{p}$ is minimal over $z$. \item If $I\subseteq A$ and there exists an element $f\in I$ contained in only finitely many maximal ideals of $A$, then $I=(f, g)_{A}$ for some $g\in A$.\end{enumerate}\end{prop}
\begin{proof}(1) By Proposition~\ref{Properties of cancellative fractional ideals}(2), $I^{-1}=II^{-1}:_{A[\varpi^{-1}]}I=A:_{A[\varpi^{-1}]}I$. 

(2) We can write $1=\sum_{i=1}^{n}f_{i}g_{i}$ for some $f_{1},\dots, f_{n}\in I$, $g_{1},\dots, g_{n}\in I^{-1}=A:_{A[\varpi^{-1}]}I$. Thus every $x\in I$ can be written as $x=\sum_{i=1}^{n}f_{i}g_{i}x$ with $g_{i}x\in A$ since $g_{i}\in A:_{A[\varpi^{-1}]}I$. It follows that $I$ is generated by $f_1,\dots, f_n$. 

(3) Let $\mathfrak{m}_{1},\dots, \mathfrak{m}_{k}$ be the maximal ideals of the semilocal ring $A$. By Prop.~\ref{Cancellative fractional ideals and maximal ideals}, $I\supsetneq \bigcup_{i=1}^{k}I\mathfrak{m}_{i}$. We claim that every $f\in I\setminus\bigcup_{i=1}^{k}I\mathfrak{m}_{i}$ generates $I$. Indeed, for such $f$ we have $(f)_{A}\not\subseteq \bigcup_{i=1}^{k}I\mathfrak{m}_{i}$ and, consequently, \begin{equation*}fI^{-1}\not\subseteq \bigcup_{i=1}^{k}\mathfrak{m}_{i}.\end{equation*}It follows that $fI^{-1}=A$ (note that $fI^{-1}\subseteq A$ since $f\in I$). In other words, any such $f$ generates $I$, as claimed. 

Assume now that $A$ is local with maximal ideal $\mathfrak{m}$ and let $(f_{\lambda})_{\lambda\in\Lambda}$ be a generating set of $I$. By the previous paragraph, $I$ is generated by a single element $f$. By Nakayama's lemma, $I/\mathfrak{m}I\neq0$, so $I/\mathfrak{m}I$ is a $1$-dimensional $A/\mathfrak{m}$-vector space. By Nakayama's lemma again, $I$ must then be generated by one of the elements $f_{\lambda}$. 

(4) If $I=(f_1,\dots, f_n)_{A}\subseteq A$, then $IA_{\mathfrak{p}}=(\frac{f_1}{1},\dots, \frac{f_{n}}{1})_{A_{\mathfrak{p}}}$. By (3), $IA_{\mathfrak{p}}=(\frac{f_{i}}{1})_{A_{\mathfrak{p}}}$ for some $i$, so $\mathfrak{p}A_{\mathfrak{p}}$ is minimal over $\frac{f_{i}}{1}$ and $\mathfrak{p}$ is minimal over $f_{i}$. 

(5) If $f=0$, then $A$ is semilocal, so the claim follows from (3). Thus we may assume that $f\neq0$. If $I=(f)_{A}$, we are done, so we may assume that $(f)_{A}\subsetneq I$. Then $(f)_{A}I^{-1}\subsetneq A$ and thus $(f)_{A}=fII^{-1}=I(fI^{-1})\subseteq fI^{-1}$. If $fI^{-1}$ is contained in infinitely many maximal ideals of $A$, then so is $f$, since we assumed that $I\subseteq A$. Thus $fI^{-1}$ is contained in only finitely many maximal ideals $\mathfrak{m}_{1},\dots, \mathfrak{m}_{k}$. By Proposition \ref{Cancellative fractional ideals and maximal ideals}, there exists some $g\in I\setminus\bigcup_{i=1}^{k}I\mathfrak{m}_{i}$. Then $gI^{-1}\subseteq A$ and $gI^{-1}\not\subseteq \mathfrak{m}_{i}$ for all $i$. Consequently, $fI^{-1}+gI^{-1}\not\subseteq \mathfrak{m}_{i}$ for all $i$. If $\mathfrak{m}$ is any maximal ideal of $A$ but one of the $\mathfrak{m}_{1},\dots, \mathfrak{m}_{k}$, then $fI^{-1}\not\subseteq \mathfrak{m}$ and thus $fI^{-1}+gI^{-1}\not\subseteq \mathfrak{m}$. Summing up, we have seen that $fI^{-1}+gI^{-1}$ is not contained in any maximal ideal of $A$, i.e., $fI^{-1}+gI^{-1}=A$. Consequently, $I=(f, g)_{A}$.\end{proof}
Recall the following generalization of the notion of a $\varpi$-adically separated valuation ring, due to Fujiwara and Kato.
\begin{mydef}[\cite{FK}~Ch.~0, Definition~8.7.1]\label{Valuative rings}A ring $A$ with a non-zero-divisor $\varpi$ is said to be $\varpi$-valuative if every $\varpi$-adically open finitely generated ideal of $A$ is invertible.\end{mydef}
Recall from \cite{Knebusch-Zhang} the notion of a Prüfer subring $A$ of a ring $R$ (loc.~cit., Ch.~I, \S5, Definition 1). In our case (with $R=A[\varpi^{-1}]$ for some non-zero-divisor $\varpi\in A$), we can relate this notion to the notion of a $\varpi$-valuative ring as follows. 
\begin{prop}\label{Pruefer subrings and valuative rings}Let $\varpi$ be a non-zero-divisor in a ring $A$. Then $A$ is a Prüfer subring of $A[\varpi^{-1}]$ if and only if $A$ is a $\varpi$-valuative ring.\end{prop}
\begin{proof}If $A$ is a Prüfer subring of $A[\varpi^{-1}]$, then every finitely generated open $A$-submodule of $A[\varpi^{-1}]$ is an invertible $\varpi$-fractional ideal, by \cite{Knebusch-Zhang}~Ch.~II, Theorem~1.13. By the last bullet point in Remark~\ref{Bullet points on invertible ideals}, this implies that every finitely generated open $A$-submodule of $A[\varpi^{-1}]$ is invertible as an $A$-module. In particular, every finitely generated ideal $I$ of $A$ containing a power of $\varpi$ is invertible, so $A$ is a $\varpi$-valuative ring. 

Conversely, suppose that $A$ is $\varpi$-valuative. By \cite{Knebusch-Zhang}~Ch.~II, Theorem~2.1(1)$\Leftrightarrow$(3), to prove that $A$ is Prüfer in $A[\varpi^{-1}]$, it suffices to prove that every finitely generated ideal $I$ of $A$ which contains an invertible ideal $J$ of $A$ must itself be invertible. But, by Lemma~\ref{Invertible implies regular}, the invertible ideal $J$ is a $\varpi$-fractional ideal and in particular open. Thus every finitely generated ideal $I$ as above contains some power of $\varpi$. Since $A$ is $\varpi$-valuative, we conclude that every such $I$ is invertible and thus that $A$ is Prüfer in $A[\varpi^{-1}]$. \end{proof}
\begin{cor}\label{Integral closure as intersection of valuative rings}Let $\varpi$ be a non-zero-divisor in a ring $A$. Then the integral closure $\mathcal{A}^{+}$ of $A$ inside $\mathcal{A}=A[\varpi^{-1}]$ is equal to the intersection of all subrings $V\subseteq \mathcal{A}$ with $A\subseteq V$ which are $\varpi$-valuative rings.\end{cor}
\begin{proof}Follows from the above proposition and \cite{Picavet23}~Theorem 4.7. \end{proof}
\begin{cor}\label{Valuative rings via invertible submodules}The ring $A$ is $\varpi$-valuative if and only if every finitely generated open $A$-submodule of $A[\varpi^{-1}]$ is invertible.\end{cor}
\begin{proof}Follows from Proposition \ref{Pruefer subrings and valuative rings} and \cite{Knebusch-Zhang}~Ch.~II, Theorem~1.13.\end{proof}
\begin{prop}\label{Valuative rings and valuation rings}Let $\varpi$ be a non-zero-divisor in a ring $A$ such that $A$ is $\varpi$-adically Zariskian. The following are equivalent: \begin{enumerate}[(1)]\item $A$ is a local $\varpi$-valuative ring. \item $A$ is a local ring and for every $n\geq1$ and $f\in A$ either $\varpi^{n}\in (f)_{A}$ or $f\in (\varpi^{n})_{A}$. \item The quotient $A/\bigcap_{n\geq1}\varpi^{n}A$ is a valuation ring.\end{enumerate}\end{prop}
\begin{proof}(1)$\Rightarrow$ (2): Suppose that $A$ is a local $\varpi$-valuative ring and let $f\in A$, $n\geq1$. Then the open ideal $(f, \varpi^{n})_{A}$ is invertible. By Proposition \ref{Properties of invertible fractional ideals}(3), this implies that $(f, \varpi^{n})_{A}=(f)_{A}$ or $(f, \varpi^{n})_{A}=(\varpi^{n})_{A}$. 

(2)$\Rightarrow$(3): Set $V=A/\bigcap_{n\geq1}\varpi^{n}A$. Since the property (2) holds for $A$, it also holds for $V$. On the other hand, for every $f\in V\setminus\{0\}$, we can choose $n\geq1$ such that $f\not\in(\varpi^{n})_{V}$. Then $\varpi^{n}\in (f)_{V}$, so $f$ is invertible in $V[\varpi^{-1}]$, so every non-zero element of $V$ is invertible in $V[\varpi^{-1}]$. Therefore, $V[\varpi^{-1}]$ is equal to the field of fractions of $V$. Furthermore, the property that every $f\in V$ and $n\geq1$ satisfy either $f\in(\varpi^{n})_{V}$ or $\varpi^{n}\in (f)_{V}$ means precisely that every \begin{equation*}x=\frac{f}{\varpi^{n}}\in V[\varpi^{-1}]\end{equation*}satisfies either $x\in V$ or $x^{-1}\in V$. Thus $V$ is a valuation ring. 

(3)$\Rightarrow$(1): If $V=A/\bigcap_{n\geq1}\varpi^{n}A$ is a valuation ring (and, in particular, a local ring), $A$ must be a local ring, by the assumption that $A$ is also $\varpi$-adically Zariskian. Let $I$ be an open finitely generated ideal of $A$, so $\varpi^{m}\in I$ for some $m$. Since $V$ is a valuation ring, $I=(f)_{A}+\bigcap_{n\geq1}\varpi^{n}A$ for a single element $f\in A$ and then \begin{equation*}I\subseteq (f, \varpi^{m+1})_{A}\subseteq (f)_{A}+\mathfrak{m}I.\end{equation*}By Nakayama's lemma, this implies $I=(f)_{A}$, so $I$ is invertible.\end{proof}
\begin{cor}\label{Valuative rings and completion}Let $A$ be a local ring and let $\varpi$ be a non-zero-divisor in the maximal ideal of $A$. The ring $A$ is $\varpi$-valuative if and only if the $\varpi$-adic completion $\widehat{A}$ of $A$ is a valuation ring (with pseudo-uniformizer $\varpi$).\end{cor}
\begin{proof}If $A$ is $\varpi$-valuative, then by Proposition~\ref{Valuative rings and valuation rings} the Hausdorff quotient $A/\bigcap_{n\geq1}\varpi^{n}A$ is a valuation ring, so its $\varpi$-adic completion, which coincides with the $\varpi$-adic completion of $A$, is a valuation ring. Thus it suffices to prove that $A$ is $\varpi$-valuative whenever $\widehat{A}$ is. Moreover, it suffices to assume that $A$ is $\varpi$-adically separated, so that the map $A\to\widehat{A}$ is injective. Let $f\in A$. If $\widehat{A}$ is $\varpi$-valuative, then (by Proposition~\ref{Valuative rings and valuation rings}), for every $n\geq1$, we either have $f\in(\varpi^{n})_{\widehat{A}}$ or $\varpi^{n}\in (f)_{\widehat{A}}$. In the first case, $f\in (\varpi^{n})_{\widehat{A}}\cap A=(\varpi^{n})_{A}$. In the second case, note that $(f)_{\widehat{A}}\cap A$ is the $\varpi$-adic closure (in $A$) of $(f)_{A}$, so $\varpi^{n}$ is in the $\varpi$-adic closure of $(f)_{A}$. In particular, the $\varpi$-adic closure of $(f)_{A}$ in $A$ is $\varpi$-adically open in $A$. By \cite{FGK}~Lemma~5.1.3(2), this implies that $(f)_{A}$ is itself $\varpi$-adically open and then $(f)_{A}=(f)_{\widehat{A}}\cap A$, so that $\varpi^{n}\in (f)_{A}$. We conclude by applying Proposition~\ref{Valuative rings and valuation rings} that $A$ is $\varpi$-valuative.\end{proof}
\begin{cor}\label{Valuative rings and completion 2}Let $\varpi$ be a non-zero-divisor and a non-unit in a local ring $A$. Then the $\varpi$-adic completion $\widehat{A}$ of $A$ is a valuation ring of rank $1$ if and only if $A$ is $\varpi$-valuative and $\sqrt{(\varpi)_{A}}$ is the maximal ideal of $A$.\end{cor}
\begin{proof}By Corollary \ref{Valuative rings and completion}, $\widehat{A}$ is a valuation ring if and only if $A$ is $\varpi$-valuative. In this case, the valuation ring $\widehat{A}$ is of rank $1$ if and only if $\sqrt{(\varpi)_{\widehat{A}}}$ is the maximal ideal of $\widehat{A}$, by \cite{FK}~Ch.~0, Proposition 6.7.3. But this happens if and only if $\sqrt{(\varpi)_{A}}$ is the maximal ideal of $A$. \end{proof}
We now want to characterize the property of $A$ being a $\varpi$-valuative ring by means of the Berkovich spectrum (respectively, the adic spectrum) of the Tate ring $A[\varpi^{-1}]$ (respectively, of the Tate Huber pair $(A[\varpi^{-1}], A)$ in case $A$ is integrally closed in $A[\varpi^{-1}]$). We first record the following elementary lemma which is presumably well-known but whose proof we include for the reader's convenience.
\begin{lemma}\label{Almost power-bounded}Let $(\mathcal{A}, \mathcal{A}^{+})$ be a Tate Huber pair. Then, for every pseudo-uniformizer $\varpi\in \mathcal{A}^{+}$ of $\mathcal{A}$ we have $\varpi \mathcal{A}^{\circ}\subseteq \mathcal{A}^{+}$.\end{lemma}
\begin{proof}Let $f\in \mathcal{A}^{\circ}$. By the definition  of a power-bounded element, there exists $n\geq1$ such that $\varpi^{n}f^{m}\in \mathcal{A}^{+}$ for all $m\geq0$. In particular, $(\varpi f)^{n}=\varpi^{n}f^{n}\in \mathcal{A}^{+}$. Since $\mathcal{A}^{+}$ is integrally closed in $\mathcal{A}$, we conclude that $\varpi f\in \mathcal{A}^{+}$. \end{proof}
\begin{lemma}\label{Valuative rings and Berkovich spectrum}Let $\varpi$ be a non-zero-divisor and a non-unit in a local ring $A$ and suppose that $A$ is completely integrally closed in $\mathcal{A}=A[\varpi^{-1}]$. Then $A$ is $\varpi$-valuative if and only if the Berkovich spectrum $\mathcal{M}(\mathcal{A})$ consists of a single point.\end{lemma}
\begin{proof}Suppose that $A$ is $\varpi$-valuative. Since $\mathcal{M}(A[\varpi^{-1}])\simeq\mathcal{M}(\widehat{A}[\varpi^{-1}])$, where $\widehat{A}$ is the $\varpi$-adic completion of $A$, we may assume (using Corollary~\ref{Valuative rings and completion}) that $A$ is actually a complete valuation ring of rank $1$ with pseudo-uniformizer $\varpi$. In this case, $\mathcal{A}$ is a nonarchimedean field and the assertion is known (see, for example, \cite{Kedlaya18}, Remark~2.11). 

Conversely, suppose that the Berkovich spectrum $\mathcal{M}(\mathcal{A})$ consists of a single point which we denote by $v$. By \cite{Kedlaya18}~Lemma 2.4, we know that $\mathcal{A}/\ker(v)$ is a field. By \cite{Berkovich}~Theorem 1.3.1, $\vert\cdot\vert_{\spc,\varpi}=v$. Then Proposition~\ref{Arc-closure and the spectral seminorm} tells us that\begin{equation*}\ker(v)=\bigcap_{n\geq1}(\varpi^{n})_{A}^{\arc_{\varpi}}.\end{equation*}By Corollary~\ref{Arc-closure of ideals and rings 2} and the assumption that $A$ is completely integrally closed in $\mathcal{A}$, this entails that \begin{equation*}\ker(v)=\bigcap_{n\geq1}(\varpi^{n})_{A}.\end{equation*}Since $A$ is completely integrally closed in $\mathcal{A}$ and $v=\vert\cdot\vert_{\spc,\varpi}$, the subring $A$ of $\mathcal{A}$ is equal to the closed unit ball $\mathcal{A}_{v\leq1}$ with respect to $v$. Therefore, $A/\bigcap_{n\geq1}\varpi^{n}A=A/\ker(v)$ is the closed unit ball of $\mathcal{A}/\ker(v)$ with respect to the rank $1$ valuation induced by $v$, showing that $A/\bigcap_{n\geq1}\varpi^{n}A$ is completely integrally closed in $\mathcal{A}/\bigcap_{n\geq1}\varpi^{n}A$. Thus, in view of Proposition \ref{Valuative rings and valuation rings}, it suffices to prove that $A$ is $\varpi$-valuative in the case when $\ker(v)=0$. Since, in this case, $\mathcal{A}$ is a field, $\mathcal{A}^{\circ}$ is a valuation ring of rank $1$. But by the assumption that $A$ is completely integrally closed in $\mathcal{A}$ we have $A=\mathcal{A}^{\circ}$. \end{proof}
\begin{lemma}\label{Valuative rings and adic spectrum}Let $\varpi$ be a non-zero-divisor and a non-unit in a local ring $A$. Suppose that $A$ is integrally closed in $A[\varpi^{-1}]$. The adic spectrum $\Spa(A[\varpi^{-1}], A)$ consists of a single point if and only if $A$ is $\varpi$-valuative and $\widehat{A}$ is a valuation ring of rank $1$.\end{lemma}
\begin{proof}If $\widehat{A}$ is a valuation ring of rank $1$, then $\mathcal{M}(A[\varpi^{-1}])\simeq\mathcal{M}(\widehat{A}[\varpi^{-1}])$ consists of a single point by \cite{Kedlaya18}, Remark~2.11. Thus there is a unique rank $1$ point $v$ in $\Spa(A[\varpi^{-1}], A)\simeq\Spa(\widehat{A}[\varpi^{-1}], \widehat{A})$ and every other point is a specialization of this rank $1$ point. However, on one hand, a valuation $w$ on the field $\widehat{A}[\varpi^{-1}]$ defining a specialization of $v$ means that the valuation ring of $w$ is contained in the valuation ring $\widehat{A}$ of $v$ and, on the other hand, a valuation $w$ on $\widehat{A}[\varpi^{-1}]$ defining an element of $\Spa(\widehat{A}[\varpi^{-1}], \widehat{A})$ means that the valuation ring of $w$ contains $\widehat{A}$. It follows that the valuation ring of each $w\in \Spa(\widehat{A}[\varpi^{-1}], \widehat{A})$ is equal to $\widehat{A}$. But this means that every $w$ is of rank $1$. Since we have seen that $v$ is the unique rank $1$ point of $\Spa(\widehat{A}[\varpi^{-1}], \widehat{A})$, we conclude that $\Spa(\widehat{A}[\varpi^{-1}], \widehat{A})$ consists of a single point.

Conversely, if $\Spa(A[\varpi^{-1}], A)$ consists of a single point $v$, so does $\mathcal{M}(A[\varpi^{-1}])$ and thus \begin{equation*}v=\vert\cdot\vert_{\spc,\varpi}.\end{equation*}Since $A$ is integrally closed in $A[\varpi^{-1}]$, the Tate ring $A[\varpi^{-1}]$ is uniform (by Lemma~\ref{Almost power-bounded}), so the topology on $A[\varpi^{-1}]$ is defined by $v=\vert\cdot\vert_{\spc,\varpi}$. By \cite{Kedlaya18}~Lemma~2.4, $\widehat{A}[\varpi^{-1}]$ is a field. Since $A$ is integrally closed in $A[\varpi^{-1}]$, the $\varpi$-adic completion $\widehat{A}$ is integrally closed in the field $\widehat{A}[\varpi^{-1}]$. By \cite{Huber1}~Lemma~3.3(i), this means that \begin{equation*}\widehat{A}=\{\, f\in \widehat{A}[\varpi^{-1}]\mid v(f)\leq1~\textrm{for all}~v\in \Spa(\widehat{A}[\varpi^{-1}], \widehat{A})\,\}.\end{equation*}But by hypothesis, $\Spa(\widehat{A}[\varpi^{-1}], \widehat{A})$ consists of a single point $v$ and thus this unique point is of rank $1$. It follows that $\widehat{A}$ is equal to the valuation ring of $v$ and then $A$ is $\varpi$-valuative, by Corollary \ref{Valuative rings and completion}. \end{proof}
The notion of $\varpi$-fractional ideals also gives rise to a useful characterization of the complete integral closure $\mathcal{A}^{\circ}$ of $A$ inside $\mathcal{A}=A[\varpi^{-1}]$. Note that, by \cite{Swanson-Huneke}~Lemma 2.1.8, an element $f\in \mathcal{A}$ such that $fI\subseteq I$ for some finitely generated, $\varpi$-adically open $A$-submodule $I$ of $\mathcal{A}$ must be integral over $A$ (the assumption that $I$ is $\varpi$-adically open, i.e., contains a power of $\varpi$, ensures that $I$ is a faithful $A[f]$-module, so that the hypotheses of \cite{Swanson-Huneke}~Lemma~2.1.8, are satisfied). The following lemma is an analog of this fact for the complete integral closure. Our formulation of the lemma and its proof were inspired by \cite{Swanson-Huneke}~Proposition~2.4.8. 
\begin{lemma}\label{Complete integral closure and fractional ideals}Let $\varpi$ be a non-zero-divisor and a non-unit in a ring $A$ and consider the Tate ring $\mathcal{A}=A[\varpi^{-1}]$. Then \begin{equation*}\mathcal{A}^{\circ}=\bigcup_{J\in \mathcal{F}_{\varpi}(A)}J:_{\mathcal{A}}J=\bigcup_{J\in\mathcal{F}_{\varpi}(A)}(A:_{\mathcal{A}}J):_{\mathcal{A}}(A:_{\mathcal{A}}J),\end{equation*}where $\mathcal{F}_{\varpi}(A)$ is the set of $\varpi$-fractional ideals of $A$ (see Definition~\ref{Fractional ideals}). In particular, if $f\in \mathcal{A}$ satisfies $fJ\subseteq J$ for some (not necessarily finitely generated) open ideal $J$ of $A$, then $f$ is power-bounded.\end{lemma}
\begin{proof}Let $J\in \mathcal{F}_{\varpi}(A)$ and choose $m, k\geq1$ with $\varpi^{k}\in J\subseteq \varpi^{-m}A$. Suppose that $f\in A[\varpi^{-1}]$ satisfies $fJ\subseteq J$. This implies that $f^{n}J\subseteq J$ for all integers $n\geq0$. Then, in particular, $f^{n}\in \varpi^{-k}J\subseteq \varpi^{-(k+m)}A$ for all $n$, so $f$ is power-bounded. Since for any $J\in\mathcal{F}_{\varpi}(A)$ the $\varpi$-residual $A:_{\mathcal{A}}J$ is again a $\varpi$-fractional ideal (Lemma \ref{Residual is a fractional ideal}), we conclude that \begin{equation*}J:_{\mathcal{A}}J, (A:_{\mathcal{A}}J):_{\mathcal{A}}(A:_{\mathcal{A}}J)\subseteq \mathcal{A}^{\circ}\end{equation*}for every $J\in\mathcal{F}_{\varpi}(A)$. Conversely, let $f\in \mathcal{A}^{\circ}\setminus\{0\}$. Then $A[f]$ is bounded in $\mathcal{A}$, so it is a $\varpi$-fractional ideal of $A$ (note that $A[f]$ is trivially open in $\mathcal{A}$ since $A$ is). Set $I=A:_{\mathcal{A}}A[f]$. Then $I$ is a $\varpi$-fractional ideal of $A$ by Lemma \ref{Residual is a fractional ideal} with $fI\subseteq I$, so \begin{equation*}\mathcal{A}^{\circ}\subseteq \bigcup_{J\in\mathcal{F}_{\varpi}(A)}J:_{\mathcal{A}}J.\end{equation*}On the other hand, since $I$ is the $\varpi$-residual of the $\varpi$-fractional ideal $A[f]$, we also see that \begin{equation*}\mathcal{A}^{\circ}\subseteq \bigcup_{J\in\mathcal{F}_{\varpi}(A)}(A:_{\mathcal{A}}J):_{\mathcal{A}}(A:_{\mathcal{A}}J),\end{equation*}as desired.\end{proof}

\section{Generalities on $\varpi$-$\ast$-operations}\label{sec:star-operations}

In the multiplicative ideal theory of integral domains, a prominent role is played by the so-called star- and semi-star-operations. In this section we study an analog of these notions for a ring $A$ with a fixed non-zero-divisor $\varpi$, where the Tate ring $A[\varpi^{-1}]$ now occupies the role played by the fraction field of an integral domain in the classical theory. Our resulting notion of $\varpi$-$\ast$-operations is actually a special case of a general notion of $\ast$-operations introduced by Knebusch and Kaiser in their book \cite{Knebusch-Zhang2}. 

As in the classical theory, every $\varpi$-$\ast$-operation defines a distinguished class of ideals of $A$, the $\ast$-ideals. Before we introduce $\varpi$-$\ast$-operations, we consider in Definition~\ref{Divisorial} the simplest example of a class of ideals defined by a $\varpi$-$\ast$-operation, which is an analog of the class of divisorial ideals used in the multiplicative ideal theory of integral domains.
\begin{mydef}[$\varpi$-divisorial $A$-submodule]\label{Divisorial}For a non-zero-divisor $\varpi$ in a ring $A$, an $A$-submodule $I$ of $A[\varpi^{-1}]$ is called $\varpi$-divisorial if \begin{equation*}I=A:_{A[\varpi^{-1}]}(A:_{A[\varpi^{-1}]}I).\end{equation*}\end{mydef}
\begin{lemma}\label{Properties of divisorial ideals}Let $\varpi$ be a non-zero-divisor in a ring $A$.\begin{enumerate}[(1)] \item The intersection of two $\varpi$-divisorial $A$-submodules $I, J$ of $A[\varpi^{-1}]$ is $\varpi$-divisorial. \item For any $A$-submodule $I$ of $A[\varpi^{-1}]$ the $\varpi$-residual $A:_{A[\varpi^{-1}]}I$ is a $\varpi$-divisorial $A$-submodule of $A[\varpi^{-1}]$.\end{enumerate}\end{lemma}
\begin{proof}(1) Let $g\in A:_{A[\varpi^{-1}]}(A:_{A[\varpi^{-1}]}(I\cap J))$. Then $gf\in A$ for all $f\in A[\varpi^{-1}]$ with $f(I\cap J)\subseteq A$. A fortiori, $gf\in A$ for all $f\in A[\varpi^{-1}]$ with $fI\subseteq A$ and for all $f\in A[\varpi^{-1}]$ with $fJ\subseteq A$. This proves the inclusion \begin{align*}A:_{A[\varpi^{-1}]}(A:_{A[\varpi^{-1}]}(I\cap J))\\ \subseteq (A:_{A[\varpi^{-1}]}(A:_{A[\varpi^{-1}]}I))\cap (A:_{A[\varpi^{-1}]}(A:_{A[\varpi^{-1}]}J))=I\cap J,\end{align*}where the last equality follows from the assumption that $I$ and $J$ are $\varpi$-divisorial. But the opposite inclusion holds for any $A$-submodule of $A[\varpi^{-1}]$, by definition.

(2) Let $I$ be any $A$-submodule of $A[\varpi^{-1}]$ and let $g\in A:_{A[\varpi^{-1}]}(A:_{A[\varpi^{-1}]}(A:_{A[\varpi^{-1}]}I))$. This means: If $h\in A[\varpi^{-1}]$ such that $fI\subseteq A$ implies $hf\in A$, then $gh\in A$. In particular, $gI\subseteq A$, i.e., $g\in A:_{A[\varpi^{-1}]}I$. \end{proof}
\begin{example}\label{Colon ideals are divisorial}In particular, the above lemma shows that for any $A$-submodule $I$ of $A[\varpi^{-1}]$ the ideal \begin{equation*}A:_{A}I=(A:_{A[\varpi^{-1}]}I)\cap A\end{equation*}of $A$ is $\varpi$-divisorial.\end{example}
\begin{mydef}[Maximal $\varpi$-divisorial ideal]A maximal $\varpi$-divisorial ideal is a maximal element (with respect to inclusion) in the set of all (integral) proper ideals of $A$ which are $\varpi$-divisorial $A$-submodules of $A[\varpi^{-1}]$.\end{mydef}
\begin{lemma}\label{Description of maximal divisorial ideals}Every maximal $\varpi$-divisorial ideal $\mathfrak{p}$ is of the form $(\varpi^{n})_{A}:_{A}f$ for some $f\in A$, $n\geq1$.\end{lemma}
\begin{proof}If $A:_{A[\varpi^{-1}]}\mathfrak{p}=A$, then $A:_{A[\varpi^{-1}]}(A:_{A[\varpi^{-1}]}\mathfrak{p})=A$, contradicting the hypothesis that $\mathfrak{p}$ is $\varpi$-divisorial. Hence there exists an element $g\in (A:_{A[\varpi^{-1}]}\mathfrak{p})\setminus A$. Write $g=\frac{f}{\varpi^{n}}$ for some $f\in A$, $n\geq1$, with $f\not\in(\varpi^{n})_{A}$. Then $\frac{f}{\varpi^{n}}\mathfrak{p}\subseteq A$, i.e., \begin{equation*}\mathfrak{p}\subseteq A:_{A[\varpi^{-1}]}\frac{f}{\varpi^{n}}=(\varpi^{n})_{A}:_{A[\varpi^{-1}]}f.\end{equation*}But the ideal \begin{equation*}(\varpi^{n})_{A}:_{A}f=(A:_{A[\varpi^{-1}]}\frac{f}{\varpi^{n}})\cap A\end{equation*}is $\varpi$-divisorial by Example~\ref{Colon ideals are divisorial}, so we conclude that $\mathfrak{p}=(\varpi^{n})_{A}:_{A}f$, by the maximality of $\mathfrak{p}$. \end{proof}
As promised, we now introduce the analog in our context of the notion of $\ast$-operations on the set of fractional ideals of an integral domain, which is in fact a special case of a general notion of $\ast$-operations on $A$-submodules for a ring extension $A\subseteq R$ developed in the book \cite{Knebusch-Zhang2} by Knebusch and Kaiser (in our case we have $R=A[\varpi^{-1}]$). 
\begin{mydef}[$\varpi$-$\ast$-operations; a special case of \cite{Knebusch-Zhang2}~Ch.~3, Definition~3.1]\label{Star operations}A $\ast$-operation on $J_{\varpi}(A)$, or a $\varpi$-$\ast$-operation, is a map \begin{equation*}\ast: J_{\varpi}(A)\to J_{\varpi}(A),~I\mapsto I^{\ast},\end{equation*}such that for any $I, J\in J_{\varpi}(A)$ the following properties are satisfied: \begin{enumerate}[(1)]\item $I\subseteq I^{\ast}$ \item $I\subseteq J$ $\Rightarrow$ $I^{\ast}\subseteq J^{\ast}$, \item $(I^{\ast})^{\ast}=I^{\ast}$, and \item $JI^{\ast}\subseteq (JI)^{\ast}$.\end{enumerate}A $\varpi$-$\ast$-operation $\ast$ is called strict if $A^{\ast}=A$. Given a $\varpi$-$\ast$-operation $\ast$, we call an $A$-submodule of $A[\varpi^{-1}]$ (respectively, an ideal $I\subseteq A$) a $\ast$-module (respectively, a $\ast$-ideal ) if $I^{\ast}=I$. \end{mydef}
Following \cite{Knebusch-Zhang2}, we also denote $\varpi$-$\ast$-operations by Greek letters $\alpha, \beta, \gamma\dots$ when multiple $\varpi$-$\ast$-operations are involved and for every $I\in J_{\varpi}(A)$ we denote the images of $I$ under $\alpha, \beta, \gamma\dots$ by $I^{\alpha}$, $I^{\beta}$, $I^{\gamma}\dots$. 
\begin{rmk}If in the above definition $J$ is an invertible $\varpi$-fractional ideal, the inclusion in property (4) of the above definition is actually an equality (apply property (4) to the inverse $J^{-1}$). Moreover, by \cite{Knebusch-Zhang2}~Ch.~3, Remark~3.1, property (4) can be replaced by the seemingly weaker requirement that $fI^{\ast}\subseteq (fI)^{\ast}$ for all $f\in A[\varpi^{-1}]$ and $I\in J_{\varpi}(A)$. Since principal fractional ideals of an integral domain are invertible, we see that the above definition indeed specializes to the usual definition of $\ast$-operations in the case when $A$ is a domain and $A[\varpi^{-1}]$ is its field of fractions.\end{rmk}
\begin{example}The identity operation is a trivial example of a $\varpi$-$\ast$-operation.\end{example}
\begin{example}[Special case of \cite{Knebusch-Zhang2}~Ch.~3, Example 3.3]\label{Star operation induced by a map}For an arbitrary ring map $\varphi: A\to B$ the map $\ast: J_{\varpi}(A)\to J_{\varpi}(A)$ given by \begin{equation*}I^{\ast}:=\varphi_{\eta}^{-1}(\varphi_{\eta}(I)B),\end{equation*}where $\varphi_{\eta}$ is the map $A[\varpi^{-1}]\to B[\varpi^{-1}]$ induced by $\varphi$, is a $\varpi$-$\ast$-operation. It is a strict $\varpi$-$\ast$-operation if and only if $\varphi$ is a $\varpi$-adic isometry, i.e., if and only if $\varphi$ induces injective ring maps $A/\varpi^{n}A\to B/\varpi^{n}B$ for all integers $n\geq1$.\end{example}
\begin{example}As a special case of the above example, the map $J_{\varpi}(A)\to J_{\varpi}(A)$, $I\mapsto IA[\varpi^{-1}]$, sending an $A$-submodule of $A[\varpi^{-1}]$ to the ideal of $A[\varpi^{-1}]$ that it generates is a $\varpi$-$\ast$-operation which is not strict unless $\varpi$ is a unit in $A$.\end{example} 
\begin{example}[Special case of \cite{Knebusch-Zhang2}~Ch.~3, Prop.~3.7]\label{Infimum of star operations}If $(\alpha_{\lambda})_{\lambda}$ is a family of $\varpi$-$\ast$-operations, then \begin{equation*}\alpha: J_{\varpi}(A)\to J_{\varpi}(A), I^{\alpha}:=\bigcap_{\lambda}I^{\alpha_{\lambda}},\end{equation*}is again a $\varpi$-$\ast$-operation which is strict if one of the $\alpha_{\lambda}$ is strict. \end{example}
\begin{example}[Special case of \cite{Knebusch-Zhang2}~Ch.~3, discussion preceding Theorem 5.2]\label{Star operation defined by a valuation}If $v$ is a valuation on $A[\varpi^{-1}]$ (not necessarily of rank one), then \begin{equation*}I\mapsto I^{v}\coloneqq\{\, f\in A[\varpi^{-1}]\mid v(f)\leq v(g)~\textrm{for some}~g\in I\,\}\end{equation*}is a $\varpi$-$\ast$-operation.\end{example}
\begin{example}[Special case of \cite{Knebusch-Zhang2}~Ch.~3, \S5, Definition~3]\label{Convex hull}Combining the two preceding examples, we see that for every family $\Phi$ of valuations on $A[\varpi^{-1}]$ the map \begin{equation*}J_{\varpi}(A)\to J_{\varpi}(A), I\mapsto \con_{\Phi}(I)\coloneqq \bigcap_{v\in\Phi}I^{v},\end{equation*}is a $\varpi$-$\ast$-operation.\end{example}
\begin{example}[The $\arc_{\varpi}$-operation]Choosing $\Phi=\mathcal{M}(A[\varpi^{-1}])$ in the above example, we see that the $\arc_{\varpi}$-operation studied in Section~\ref{sec:Rees valuations} of this paper is a $\varpi$-$\ast$-operation. By Example~\ref{Complete integral closure and arc-closure}, this $\varpi$-$\ast$-operation is strict if and only if $A$ is completely integrally closed in $A[\varpi^{-1}]$.\end{example} 
\begin{example}[Integral closure as a $\varpi$-$\ast$-operation]\label{Integral closure as a star-operation}Let $\mathcal{A}^{+}$ be the integral closure of $A$ in $\mathcal{A}=A[\varpi^{-1}]$. Let the family of valuations $\Phi$ in Example~\ref{Convex hull} be the family of all valuations $v$ on $\mathcal{A}$ such that $v$ is the restriction of a valuation on $\Frac(\mathcal{A}^{+}/\mathfrak{p})$, where $\mathfrak{p}$ runs over the minimal prime ideals of $\mathcal{A}^{+}$ (since $\varpi$ is a non-zero-divisor, it does not belong to any minimal prime ideal and thus every such valuation on $\mathcal{A}^{+}$ indeed extends to a valuation on $\mathcal{A}$). By Example~\ref{Convex hull} and \cite{Swanson-Huneke}~Theorem~6.8.3, we see that the map $I\mapsto \overline{I}$ which takes an ideal of $\mathcal{A}^{+}$ to its integral closure in $\mathcal{A}^{+}$ extends to a $\varpi$-$\ast$-operation $J_{\varpi}(A)\to J_{\varpi}(A)$. This $\varpi$-$\ast$-operation is strict if and only if $A$ is integrally closed in $\mathcal{A}$. \end{example}
\begin{example}[Topological closure]\label{Topological closure as a star-operation}It is readily seen that the map \begin{equation*}J_{\varpi}(A)\to J_{\varpi}(A), I\mapsto \overline{I}^{\varpi},\end{equation*}which takes an $A$-submodule $I$ of $A[\varpi^{-1}]$ to its (topological) closure inside the Tate ring $A[\varpi^{-1}]$ is a $\varpi$-$\ast$-operation which is always strict since the ring of definition $A$ of $A[\varpi^{-1}]$ is open and, a fortiori, closed in $A[\varpi^{-1}]$.\end{example} 
For our purposes, the two most important examples of $\varpi$-$\ast$-operations are introduced in the following definition; these are the natural analogs of the v- and t-operations on the set of fractional ideals of an integral domain.
\begin{mydef}[$\varpi$-v-operation and $\varpi$-t-operation]We define a map \begin{equation*}J_{\varpi}(A)\to J_{\varpi}(A),~I\mapsto I_{v},\end{equation*}by setting \begin{equation*}I_{v}:=A:_{A[\varpi^{-1}]}(A:_{A[\varpi^{-1}]}I).\end{equation*}We call $I\mapsto I_{v}$ the $\varpi$-v-operation. We also define a map \begin{equation*}J_{\varpi}(A)\to J_{\varpi}(A),~I\mapsto I_{t},\end{equation*}by setting \begin{equation*}I_{t}=\bigcup_{F\subseteq I}F_{v},\end{equation*}where $F\subseteq I$ ranges over the finitely generated $A$-submodules of $I$. We call $I\mapsto I_{t}$ the $\varpi$-t-operation.\end{mydef}  
\begin{prop}\label{v- and t-operations are star-operations}The $\varpi$-v-operation $I\mapsto I_{v}$ and the $\varpi$-t-operation $I\mapsto I_{t}$ are strict $\varpi$-$\ast$-operations. Moreover, any strict $\varpi$-$\ast$-operation $\ast$ satisfies $I^{\ast}\subseteq I_{v}$ for all $I\in J_{\varpi}(A)$.\end{prop}
\begin{proof}The assertions about the $\varpi$-v-operation are a special case of \cite{Knebusch-Zhang2}~Ch.~3, Proposition~3.6. The assertion about the $\varpi$-t-operation is a special case of op.~cit., Ch.~3, Proposition~6.3.\end{proof}
More generally, the $\varpi$-t-operation is an example of a companion of finite type (in the terminology of \cite{Knebusch-Zhang2}) of a $\varpi$-$\ast$-operation. The following definition is a special case of \cite{Knebusch-Zhang2}~Ch.~3, \S6, Definition~1. 
\begin{mydef}[$\varpi$-$\ast$-operation of finite type]\label{Star operation of finite type}A $\varpi$-$\ast$-operation is called of finite type if for every $I\in J_{\varpi}(A)$ the $A$-module $I^{\ast}$ is the union of the $A$-submodules $F^{\ast}$, where $F$ ranges over the finitely generated $A$-submodules of $I$.\end{mydef}
\begin{prop}[Special case of \cite{Knebusch-Zhang2}~Ch.~3, Proposition 6.3]\label{Companion of finite type}For any $\varpi$-$\ast$-operation $\ast$ define a map $\ast_{f}: J_{\varpi}(A)\to J_{\varpi}(A)$ by \begin{equation*}I^{\ast_{f}}:=\bigcup_{\substack{F\subseteq I\\ F~f.~g.}}F^{\ast}.\end{equation*}Then $\ast_{f}$ is a $\varpi$-$\ast$-operation of finite type.\end{prop}
\begin{mydef}[Special case of \cite{Knebusch-Zhang2}~Ch.~3, \S6, Definition~2]In the situation of the above proposition, we call the $\varpi$-$\ast$-operation $\ast_{f}$ the companion of $\ast$ of finite type or the finite-type companion of $\ast$.\end{mydef}  
Our reason for studying the $\varpi$-v-operation and the $\varpi$-t-operation is their relation to weakly associated prime ideals of $\varpi$.
\begin{lemma}\label{Maximal t-ideals}Let $\ast: J_{\varpi}(A)\to J_{\varpi}(A)$ be a $\varpi$-$\ast$-operation of finite type. Every $\ast$-ideal in the ring $A$ is contained in a maximal $\ast$-ideal. In particular, every $\varpi$-t-ideal of $A$ (and thus also every $\varpi$-divisorial ideal of $A$) is contained in a maximal $\varpi$-t-ideal.\end{lemma}
\begin{proof}By Zorn's lemma, it suffices to prove that every strictly ascending chain of $\ast$-ideals has an upper bound. Let \begin{equation*}I_{0}\subsetneq I_{1}\subsetneq I_{2}\subsetneq\dots\subsetneq I_{i}\subsetneq\dots\end{equation*}be a strictly ascending chain of $\ast$-ideals in $A$. Let $F$ be a finitely generated subideal of the sum $\sum_{i}I_{i}=\bigcup_{i}I_{i}$. Since $F$ is finitely generated, there exists some index $i$ such that $F\subseteq I_{i}$. Since $I_{i}$ is a $\ast$-ideal and $\ast$ is of finite type, this entails that $F^{\ast}\subseteq I_{i}\subseteq \sum_{i}I_{i}$. Consequently, $\sum_{i}I_{i}$ is a $\ast$-ideal and is thus the desired upper bound of our strictly ascending chain. \end{proof}
The following fact has been noticed (in the context of $\ast$-operations on general ring extensions $A\subseteq R$) by Tchamna.
\begin{lemma}[Special case of \cite{Tchamna20}, Remark~2.4]\label{Maximal star-ideals are prime}For a strict $\varpi$-$\ast$-operation of finite type $\ast: J_{\varpi}(A)\to J_{\varpi}(A)$ every maximal $\ast$-ideal of $A$ is a prime ideal.\end{lemma} 
\begin{lemma}[Analog of \cite{Zafrullah78}~Lemma~6]\label{Weakly associated primes and maximal t-ideals}Let $\mathfrak{p}$ be a weakly associated prime ideal of $(\varpi^{n})_{A}$ for some $n\geq1$. Then $\mathfrak{p}A_{\mathfrak{p}}$ is a $\varpi$-t-ideal of $A_{\mathfrak{p}}$.\end{lemma}
\begin{proof}By the definition of weakly associated primes, $\mathfrak{p}$ is minimal over $(\varpi^{n}):_{A}f$ for some $f\in A$. Then $\mathfrak{p}A_{\mathfrak{p}}$ is minimal over $(\varpi^{n})_{A_{\mathfrak{p}}}:_{A_{\mathfrak{p}}}\frac{f}{1}$ and thus, being the maximal ideal of $A_{\mathfrak{p}}$, it is the unique minimal prime ideal over $(\varpi^{n})_{A_{\mathfrak{p}}}:_{A_{\mathfrak{p}}}\frac{f}{1}$. The ideal $(\varpi^{n})_{A_{\mathfrak{p}}}:_{A_{\mathfrak{p}}}\frac{f}{1}$ is $\varpi$-divisorial by Example~\ref{Colon ideals are divisorial} and, in particular, it is a $\varpi$-t-ideal. By Lemma~\ref{Maximal t-ideals}, $(\varpi^{n})_{A_{\mathfrak{p}}}:_{A_{\mathfrak{p}}}\frac{f}{1}$ is contained in a maximal $\varpi$-t-ideal $\mathfrak{m}$ of $A_{\mathfrak{p}}$. By Lemma~\ref{Maximal star-ideals are prime}, $\mathfrak{m}$ is a prime ideal, so $\mathfrak{p}A_{\mathfrak{p}}\subseteq \mathfrak{m}$, by the fact that $\mathfrak{p}A_{\mathfrak{p}}$ is the unique minimal prime over $(\varpi^{n})_{A_{\mathfrak{p}}}:_{A_{\mathfrak{p}}}\frac{f}{1}$, and thus $\mathfrak{p}A_{\mathfrak{p}}=\mathfrak{m}$ since $\mathfrak{p}A_{\mathfrak{p}}$ is the maximal ideal of $A_{\mathfrak{p}}$. \end{proof}
The above lemma is actually a special case of the following more general observation.
\begin{prop}[Analog of \cite{Hedstrom-Houston}~Proposition~1.1(5)]\label{Prime ideals minimal over a star-ideal}If $\ast$ is a strict $\varpi$-$\ast$-operation and $\mathfrak{p}$ is a prime ideal minimal over a $\varpi$-$\ast$-ideal $I$, then $\mathfrak{p}$ is a $\varpi$-$\ast_{f}$-ideal.\end{prop}
\begin{proof}We follow the proof of \cite{Hedstrom-Houston}~Proposition~1.1(5). Let $J$ be a finitely generated subideal of $\mathfrak{p}$. We want to show that $J^{\ast}\subseteq \mathfrak{p}$. Since $\mathfrak{p}$ is minimal over $I$, the maximal ideal $\mathfrak{p}A_{\mathfrak{p}}$ of $A_{\mathfrak{p}}$ is equal to the radical of $IA_{\mathfrak{p}}$. Since $JA_{\mathfrak{p}}\subseteq \mathfrak{p}A_{\mathfrak{p}}$ and $J$ is finitely generated, there is an integer $n>0$ with $J^{n}A_{\mathfrak{p}}\subseteq IA_{\mathfrak{p}}$. Thus, using that $J^{n}$ is finitely generated, we see that there exists an element $s\in A\setminus\mathfrak{p}$ with $sJ^{n}\subseteq I$. Then \begin{equation*}s(J^{\ast})^{n}\subseteq s((J^{\ast})^{n})^{\ast}=s(J^{n})^{\ast}\subseteq (sJ^{n})^{\ast}\subseteq I^{\ast}\subseteq\mathfrak{p}.\end{equation*}For the first inclusion we used property (1) in the definition of a $\varpi$-$\ast$-operation. For the inclusion $s(J^{n})^{\ast}\subseteq (sJ^{n})^{\ast}$ we used property (4) in that definition. For the middle equality we used \cite{Knebusch-Zhang2}~Ch.~3, Proposition~4.1a). Since $\ast$ is strict, $J^{\ast}\subseteq A^{\ast}=A$. Since $s\not\in\mathfrak{p}$ and $\mathfrak{p}$ is a prime ideal, we have $J^{\ast}\subseteq\mathfrak{p}$, as desired.\end{proof}
For any $A$ and $\varpi$, the family of weakly associated prime ideals of $\varpi^{n}$ for $n\geq1$ (which actually coincides with the family $\wAss(\varpi)$ of weakly associated prime ideals of $\varpi$, by \cite{Brewer-Heinzer}~Theorem~3) is readily seen to have the property that $A$ is equal to the set of elements of $A[\varpi^{-1}]$ whose images in $A_{\mathfrak{p}}[\varpi^{-1}]$ belong to $A_{\mathfrak{p}}$ for all $\mathfrak{p}\in \wAss(\varpi)$ (see the proof of Proposition~\ref{Weakly associated prime ideals and boundaries} for the argument). Since the ideal $(\varpi)_{A}:_{A}f$ is $\varpi$-divisorial for all $f\in A$ (see Example~\ref{Colon ideals are divisorial}), Proposition~\ref{Prime ideals minimal over a star-ideal} implies that the same is also true for the family of all prime $\varpi$-t-ideals. The following proposition, which is an analog of \cite{Griffin67} Prop.~4, puts this observation into a broader context.
\begin{prop}[Analog of \cite{Griffin67}~Proposition~4]\label{Griffin's intersection property}Let $\ast: J_{\varpi}(A)\to J_{\varpi}(A)$ be a strict $\varpi$-$\ast$-operation of finite type. Then, for every $I\in J_{\varpi}(A)$, we have \begin{equation*}I^{\ast}=\bigcap_{\mathfrak{p}\in\Max_{\ast}(A)}\varphi_{\mathfrak{p}}^{-1}(\varphi_{\mathfrak{p}}(I^{\ast})A_{\mathfrak{p}}),\end{equation*}where we denote by $\Max_{\ast}(A)$ the set of maximal $\ast$-ideals of $A$ and where, for every prime ideal $\mathfrak{p}$, we denote by $\varphi_{\mathfrak{p}}$ the map $A[\varpi^{-1}]\to A_{\mathfrak{p}}[\varpi^{-1}]$ induced by the localization map. \end{prop}
\begin{proof}By abuse of notation, we denote the $A_{\mathfrak{p}}$-submodule $\varphi_{\mathfrak{p}}(J)A_{\mathfrak{p}}$ of $A_{\mathfrak{p}}[\varpi^{-1}]$, for any $A$-submodule $J$ of $A[\varpi^{-1}]$, by $JA_{\mathfrak{p}}$. Similarly, for $J\in J_{\varpi}(A_{\mathfrak{p}})$ we sometimes denote the pre-image $\varphi_{\mathfrak{p}}^{-1}(J)$ in $A[\varpi^{-1}]$ by $J\cap A[\varpi^{-1}]$.

If a non-zero element $\frac{f}{\varpi^{n}}$ of $A[\varpi^{-1}]$ (with $f\in A$, $n\geq1$) maps into $I^{\ast}A_{\mathfrak{p}}$ for some $\mathfrak{p}\in\Max_{\ast}(A)$, write $\frac{f}{\varpi^{n}}=\frac{g}{h}$ with $g\in I^{\ast}$ and $h\in A\setminus\mathfrak{p}$. Then there exists $h^{\prime}\in A\setminus\mathfrak{p}$ with $h^{\prime}h\frac{f}{\varpi^{n}}=h^{\prime}g\in I^{\ast}$, so \begin{equation*}I^{\ast}:_{A}\frac{f}{\varpi^{n}}\not\subseteq\mathfrak{p}.\end{equation*}Therefore, if $\frac{f}{\varpi^{n}}$ maps to an element of $I^{\ast}A_{\mathfrak{p}}$ for every $\mathfrak{p}\in\Max_{\ast}(A)$, we see that the ideal $I^{\ast}:_{A}\frac{f}{\varpi^{n}}$ is not contained in any maximal $\ast$-ideal $\mathfrak{p}$ and thus \begin{equation*}(I^{\ast}:_{A}\frac{f}{\varpi^{n}})^{\ast}=A,\end{equation*}by Lemma~\ref{Maximal t-ideals}. But then \begin{equation*}\frac{f}{\varpi^{n}}\in \frac{f}{\varpi^{n}}(I^{\ast}:_{A}\frac{f}{\varpi^{n}})^{\ast} \subseteq (\frac{f}{\varpi^{n}}(I^{\ast}:_{A}\frac{f}{\varpi^{n}}))^{\ast}\subseteq I^{\ast}.\end{equation*}This shows that \begin{equation*}\bigcap_{\mathfrak{p}\in\Max_{\ast}(A)}(I^{\ast}A_{\mathfrak{p}}\cap A[\varpi^{-1}])\subseteq I^{\ast}\end{equation*}for every $I\in J_{\varpi}(A)$. The opposite inclusion is obvious.
\end{proof}
Following the exposition in \cite{Fuchs-Salce}~Chapter~1, \S3, we introduce an analog of the notion of divisors of an integral domain.
\begin{mydef}[$\varpi$-divisors]\label{Divisors}For a non-zero-divisor $\varpi$ in a ring $A$, define a congruence relation $\sim_{\varpi}$ on $\mathcal{F}_{\varpi}(A)$ by $I\sim_{\varpi} J$ if and only if $A:_{A[\varpi^{-1}]}I=A:_{A[\varpi^{-1}]}J$. Recall that $A:_{A[\varpi^{-1}]}I, I_{v}\in \mathcal{F}_{\varpi}(A)$ for all $I\in\mathcal{F}_{\varpi}(A)$ by Lemma \ref{Residual is a fractional ideal}. Note that $I\sim_{\varpi} I_{v}$ for all $I\in \mathcal{F}_{\varpi}(A)$ and that distinct $\varpi$-divisorial $\varpi$-fractional ideals are incongruent. We call the equivalence classes under $\sim_{\varpi}$ the $\varpi$-divisors of $A$.\end{mydef}
We denote the $\varpi$-divisor associated with $I\in \mathcal{F}_{\varpi}(A)$ by $\ddiv I$ and denote the set $\mathcal{F}_{\varpi}(A)/\sim_{\varpi}$ of $\varpi$-divisors by $D_{\varpi}(A)$. For $0\neq f\in A[\varpi^{-1}]$ with $\varpi^{n}\in (f)_{A}$ for some $n\geq1$, denote the $\varpi$-divisor $\ddiv((f)_{A})$ by $\ddiv(f)$. The set $D_{\varpi}(A)$ becomes an additive monoid under the operation $\ddiv I+\ddiv J\coloneqq\ddiv IJ$. This operation is well-defined. Indeed, using Lemma~\ref{Properties of divisorial ideals}(2) and \cite{Knebusch-Zhang2}~Ch.~3, Proposition~4.4,  
\begin{align*}A:_{A[\varpi^{-1}]}IJ=(A:_{A[\varpi^{-1}]}I):_{A[\varpi^{-1}]}J=(A:_{A[\varpi^{-1}]}I_{v}):_{A[\varpi^{-1}]}J \\=(A:_{A[\varpi^{-1}]}I_{v}):_{A[\varpi^{-1}]}J_{v}=A:_{A[\varpi^{-1}]}I_{v}J_{v}\end{align*}and thus $\ddiv I+\ddiv J=\ddiv I_{v}+\ddiv J_{v}$. Moreover, $D_{\varpi}(A)=\mathcal{F}_{\varpi}(A)/\sim_{\varpi}$ is a partially ordered monoid by defining $\ddiv I\leq \ddiv J$ if $A:_{A[\varpi^{-1}]}I\subseteq A:_{A[\varpi^{-1}]}J$. In fact, for $J\in J_{\varpi}(A)$, the inclusion \begin{equation*}A:_{A[\varpi^{-1}]}I\subseteq A:_{A[\varpi^{-1}]}I'\end{equation*}implies $A:_{A[\varpi^{-1}]}IJ\subseteq A:_{A[\varpi^{-1}]}I'J$, so the above partial order is compatible with the addition operation of $D_{\varpi}(A)$.
\begin{prop}[Analog of \cite{Fuchs-Salce}~Proposition~3.11]The monoid $D_{\varpi}(A)$ is a lattice-ordered monoid. For $I, I'\in \mathcal{F}_{\varpi}(A)$ the supremum of $\ddiv I$ and $\ddiv I'$ in $D_{\varpi}(A)$ is \begin{equation*}\ddiv(I_{v}\cap I'_{v})\end{equation*}and the infimum is $\ddiv(I+I')$.\end{prop}
\begin{proof}Let $J\in \mathcal{F}_{\varpi}(A)$ with $\ddiv I\leq \ddiv J$ and $\ddiv I'\leq \ddiv J$. Then $J\subseteq J_{v}\subseteq I_{v}$ and $J\subseteq J_{v}\subseteq I'_{v}$, showing that the upper bound $\ddiv(I_{v}\cap I'_{v})$ of $\ddiv I$ and $\ddiv I'$ is a supremum. 

Now let $J\in \mathcal{F}_{\varpi}(A)$ with $\ddiv J\leq \ddiv I$ and $\ddiv J\leq \ddiv I'$. We have $I\subseteq I_{v}\subseteq J_{v}$ and $I'\subseteq I'_{v}\subseteq J_{v}$, so that $I+I'\subseteq J_{v}$ and, consequently, $\ddiv J=\ddiv J_{v}\leq \ddiv(I+I')$. This shows that the lower bound $\ddiv(I+I')$ is indeed an infimum in $D_{\varpi}(A)$. \end{proof}
\begin{mydef}[Special case of \cite{Knebusch-Zhang2}~Ch.~3, \S4, Def.~1]Given a $\varpi$-$\ast$-operation \begin{equation*}\ast: J_{\varpi}(A)\to J_{\varpi}(A)\end{equation*}we define the $\ast$-product of any two $A$-submodules $I, J\in J_{\varpi}(A)$ by \begin{equation*}I\circ_{\ast}J=(IJ)^{\ast}.\end{equation*}An $A$-submodule $I$ of $A[\varpi^{-1}]$ is said to be $\ast$-invertible if there exists $J\in J_{\varpi}(A)$ with $I\circ_{\ast}J=A^{\ast}$. When $\ast$ is the $\varpi$-v-operation (respectively, the $\varpi$-t-operation) we call $\circ_{\ast}=\circ_{v}$ (respectively, $\circ_{\ast}=\circ_{t}$) the $\varpi$-v-multiplication (respectively, the $\varpi$-t-multiplication). \end{mydef}
\begin{prop}\label{Monoid of divisors}The map $I\mapsto \ddiv I$ is an order-reversing bijection between the monoid of $\varpi$-divisorial $\varpi$-fractional ideals of $A$ with the $\varpi$-v-multiplication \begin{equation*}(I, J)\mapsto (IJ)_{v},\end{equation*}partially ordered by inclusion, and the partially ordered monoid $D_{\varpi}(A)$.\end{prop}
\begin{proof}Follows from the definitions and the fact that $I\sim_{\varpi}I_{v}$ for every $I\in\mathcal{F}_{\varpi}(A)$. \end{proof}
\begin{lemma}\label{Star inverse of a fractional ideal}Let $\ast: J_{\varpi}(A)\to J_{\varpi}(A)$ be a strict $\varpi$-$\ast$-operation. If $I$ is a $\ast$-invertible $\varpi$-fractional ideal and $J\in J_{\varpi}(A)$ is such that $I\circ_{\ast}J=A$, then $J^{\ast}$ is again a $\varpi$-fractional ideal.

In particular, if $I$ is a $\varpi$-fractional ideal and $J$ is a $\varpi$-divisorial $A$-submodule of $A[\varpi^{-1}]$ such that $(IJ)_{v}=A$, then $J$ is also a $\varpi$-fractional ideal. \end{lemma}
\begin{proof}By \cite{Knebusch-Zhang2}~Ch.~3, Proposition~4.1a), $I^{\ast}\circ_{\ast}J=I\circ_{\ast}J=A$. By op.~cit., Ch.~3, Proposition~4.5a), $J^{\ast}=A:_{A[\varpi^{-1}]}I$. Hence $J^{\ast}$ is a $\varpi$-fractional ideal by Lemma \ref{Residual is a fractional ideal}. \end{proof}
We can now prove a partial analog of a classical result (\cite{Fuchs-Salce}~Ch.~1, Theorem~3.12) characterizing completely integrally closed domains by means of their monoids of divisors.
\begin{thm}\label{Rings of power-bounded elements and divisors}Suppose that $A$ is completely integrally closed in $A[\varpi^{-1}]$. The partially ordered monoid $D_{\varpi}(A)$ of $\varpi$-divisors of $A$ is a lattice-ordered group.\end{thm}
\begin{proof}We follow the first paragraph of the proof of \cite{Fuchs-Salce}~Ch.~1, Theorem~3.12. Suppose that $A$ is completely integrally closed in $A[\varpi^{-1}]$. To prove that the monoid $D_{\varpi}(A)$ is a group, we show that $I(A:_{A[\varpi^{-1}]}I)\sim_{\varpi} A$ for every $I\in \mathcal{F}_{\varpi}(A)$. That is, for every $I\in \mathcal{F}_{\varpi}(A)$, we want to prove that $A:_{A[\varpi^{-1}]}(I(A:_{A[\varpi^{-1}]}I))=A$. The inclusion $A\subseteq A:_{A[\varpi^{-1}]}I(A:_{A[\varpi^{-1}]}I)$ is trivial. For the opposite inclusion, observe that $f\in A[\varpi^{-1}]$ satisfies $fI(A:_{A[\varpi^{-1}]}I)\subseteq A$ if and only if \begin{equation*}f(A:_{A[\varpi^{-1}]}I)\subseteq A:_{A[\varpi^{-1}]}I.\end{equation*}By Lemma~\ref{Residual is a fractional ideal}, $A:_{A[\varpi^{-1}]}I$ is a $\varpi$-fractional ideal, so we can then apply Lemma~\ref{Complete integral closure and fractional ideals} to see that \begin{equation*}A:_{A[\varpi^{-1}]}(I(A:_{A[\varpi^{-1}]}I))\subseteq A[\varpi^{-1}]^{\circ}.\end{equation*}But by Lemma~\ref{Completely integrally closed}, the right hand side is equal to $A$ since $A$ was assumed to be completely integrally closed in $A[\varpi^{-1}]$.\end{proof}   
                 
\section{Characterizing the Shilov boundary}\label{sec:heart of the paper}

This section is the core of the paper. In it, we introduce a certain ring-theoretic property of a ring $A$ with a non-zero-divisor and non-unit $\varpi\in A$ (the property of being strongly $\varpi$-Shilov) and show that, under some mild assumptions, this property is actually equivalent to Berkovich's description of the Shilov boundary (Definition~\ref{Berkovich's description of the Shilov boundary}) for the Tate ring $\mathcal{A}=A[\varpi^{-1}]$. We also find some sufficient conditions for a ring $A$ to be strongly $\varpi$-Shilov with respect to some non-zero-divisor and non-unit $\varpi$. 
\begin{mydef}[Weakly $\varpi$-Shilov rings]\label{Weakly and strongly Shilov}Let $A$ be a ring with a non-zero-divisor and non-unit $\varpi\in A$. We say that $A$ is weakly $\varpi$-Shilov if for every minimal prime ideal $\mathfrak{p}$ of $\varpi$ in $A$ the $\varpi$-adic completion $\widehat{A_{\mathfrak{p}}}$ of the local ring $A_{\mathfrak{p}}$ is a valuation ring of rank $1$.

In other words, $A$ is weakly $\varpi$-Shilov if for every minimal prime $\mathfrak{p}$ of $\varpi$ the quotient \begin{equation*}A_{\mathfrak{p}}/\bigcap_{n\geq1}\varpi^{n}A_{\mathfrak{p}}\end{equation*}is a valuation ring of rank $1$ (see Proposition~\ref{Valuative rings and valuation rings} and Corollary~\ref{Valuative rings and completion 2} for the equivalence of the two conditions). \end{mydef}
We combine the above definition with a variant of Serre's condition S1 for the quotient ring $A/(\varpi)_{A}$. The latter condition is usually defined for Noetherian rings using the notion of depth. Since we are also interested in non-Noetherian examples, we use the following definition which coincides with the usual definition of condition S1 in the Noetherian context.
\begin{mydef}[Generalized condition S1]A ring $A$ is said to satisfy (generalized) condition S1 if every weakly associated prime ideal of $A$ is a minimal prime ideal.\end{mydef}
We usually omit the adjective ``generalized'' and just say ``condition S1'' since this is the only variant of the condition we consider in this paper. The following class of rings will be the main focus of our study. 
\begin{mydef}[Strongly $\varpi$-Shilov rings]\label{Weakly and strongly Shilov 2}Let $A$ be a ring with a non-zero-divisor and non-unit $\varpi\in A$. We say that $A$ is strongly $\varpi$-Shilov if $A$ is weakly $\varpi$-Shilov (see Definition~\ref{Weakly and strongly Shilov}) and the quotient ring $A/(\varpi)_{A}$ satisfies condition S1. 

In other words, $A$ is strongly $\varpi$-Shilov if it is weakly $\varpi$-Shilov and every weakly associated prime ideal of $\varpi$ in $A$ is a minimal prime ideal of $\varpi$.\end{mydef}
We also introduce somewhat stronger variants of the above notions.
\begin{mydef}[Weakly and strongly $\varpi$-Krull rings]\label{Weakly and strongly Krull}Let $A$ be a ring and let $\varpi\in A$ be a non-zero-divisor and a non-unit. We call $A$ weakly $\varpi$-Krull if for every minimal prime ideal $\mathfrak{p}$ of $\varpi$ in $A$, the local ring $A_{\mathfrak{p}}$ is a valuation ring of rank $1$, and we call $A$ strongly $\varpi$-Krull if $A$ is weakly $\varpi$-Krull and, moreover, the quotient $A/(\varpi)_{A}$ satisfies condition S1.

In other words, $A$ is weakly (respectively, strongly) $\varpi$-Krull if $A$ is weakly (respectively, strongly) $\varpi$-Shilov and for every minimal prime $\mathfrak{p}$ of $\varpi$ in $A$, the local ring $A_{\mathfrak{p}}$ is $\varpi$-adically separated. \end{mydef}
We will prove in Theorem~\ref{Reduced Noetherian rings are strongly Shilov} that every integrally closed ring $A$ admitting an integral extension which is a finite product of Krull domains is strongly $\varpi$-Krull for every non-zero-divisor and non-unit $\varpi\in A$. In particular, all Krull domains are strongly $\varpi$-Krull rings for every non-zero non-unit $\varpi$, justifying our terminology.
\begin{rmk}\label{Serre's criterion}Let us add a few words about the meanings of Definition~\ref{Weakly and strongly Shilov}, Definition~\ref{Weakly and strongly Shilov 2} and Definition~\ref{Weakly and strongly Krull} in the Noetherian case. First of all, we note that for $A$ Noetherian, the weakly (respectively, strongly) $\varpi$-Shilov condition and the weakly (respectively, strongly) $\varpi$-Krull condition are equivalent by Krull's intersection theorem. Moreover, if $A$ is Noetherian and weakly $\varpi$-Shilov, then the local rings of $A$ at prime ideals minimal over $\varpi$ are discrete valuation rings - that is, the weak $\varpi$-Shilov condition is a ``$\varpi$-local" analog of the condition known as being ``regular in codimension $1$" in algebraic geometry. In particular, a reduced Noetherian ring $A$ is weakly $\varpi$-Krull (or, equivalently, weakly $\varpi$-Shilov) for all non-unit non-zero-divisors $\varpi\in A$ if and only if the affine scheme $\Spec(A)$ is regular in codimension $1$. Similarly, the requirement that $A/(\varpi)_{A}$ satisfies condition S1, which separates the weak and strong $\varpi$-Shilov conditions from each other, can be thought of as a ``$\varpi$-local" analog of Serre's condition S2. In particular, by Serre's criterion for normality (see, for example, \cite{Eisenbud}~Theorem~11.5), a reduced Noetherian ring $A$ is a finite product of normal domains if and only if it is strongly $\varpi$-Shilov for all non-unit non-zero-divisors $\varpi\in A$ if and only if it is strongly $\varpi$-Krull for all non-unit non-zero-divisors $\varpi\in A$. By studying the strong $\varpi$-Shilov condition in more detail, we will obtain a ``$\varpi$-local" generalization of Serre's criterion (for reduced Noetherian rings), see Remark~\ref{Serre's criterion 2} below.\end{rmk}   
On a related note, we have the following ``$\varpi$-local" result on Noetherian rings.
\begin{lemma}\label{Noetherian implies strongly Krull}Let $\varpi$ be a non-zero-divisor and non-unit in a Noetherian ring $A$. If $A$ is integrally closed in $A[\varpi^{-1}]$, then $A/(\varpi)_{A}$ satisfies condition S1. In particular, $A$ is weakly $\varpi$-Shilov if and only if it is strongly $\varpi$-Krull.\end{lemma}
\begin{proof}The second assertion follows from the first. Indeed, if $A$ is weakly $\varpi$-Shilov, then the first assertion immediately implies that $A$ is strongly $\varpi$-Shilov and then the Noetherian assumption implies that $A$ is strongly $\varpi$-Krull by Krull's intersection theorem. It remains to prove the first assertion of the lemma. Since $A$ is Noetherian, there is no distinction between weakly associated and associated prime ideals of $\varpi$ in $A$. Let $\mathfrak{p}$ be an associated prime ideal of $\varpi$ in $A$ and choose $f\in A$ with $\mathfrak{p}=(\varpi)_{A}:_{A}f$. Then $\mathfrak{p}A_{\mathfrak{p}}=(\varpi)_{A_{\mathfrak{p}}}:_{A_{\mathfrak{p}}}f$. We prove that $\mathfrak{p}A_{\mathfrak{p}}$ is a principal ideal of $A_{\mathfrak{p}}$, in which case $\mathfrak{p}$ must be of height $1$ by Krull's principal ideal theorem. By definition $\frac{f}{\varpi}\mathfrak{p}A_{\mathfrak{p}}\subseteq A_{\mathfrak{p}}$. If $\frac{f}{\varpi}\mathfrak{p}A_{\mathfrak{p}}\subseteq \mathfrak{p}A_{\mathfrak{p}}$, then, since $A_{\mathfrak{p}}$ is integrally closed in $A_{\mathfrak{p}}[\varpi^{-1}]$ and Noetherian, \cite{Swanson-Huneke}~Lemma~2.1.8 implies that $\frac{f}{\varpi}\in A_{\mathfrak{p}}$, in contradiction to $\mathfrak{p}A_{\mathfrak{p}}$ being a proper ideal. Thus $\frac{f}{\varpi}\mathfrak{p}A_{\mathfrak{p}}=A_{\mathfrak{p}}$, so there exists $z\in\mathfrak{p}A_{\mathfrak{p}}$ with $\frac{f}{\varpi}z=1$ in $A_{\mathfrak{p}}$. Let $g, h\in A_{\mathfrak{p}}$ with $fg=h\varpi$. Since $\varpi=fz$ is a non-zero-divisor in $A_{\mathfrak{p}}$, so are $z$ and $f$. It follows that $(z)_{A_{\mathfrak{p}}}=(\varpi)_{A_{\mathfrak{p}}}:_{A_{\mathfrak{p}}}f=\mathfrak{p}A_{\mathfrak{p}}$ and $\mathfrak{p}$ has height $1$.  \end{proof}
We also introduce an intermediate condition between the weak and strong $\varpi$-Shilov properties.
\begin{mydef}[$\varpi$-Shilov rings and $\varpi$-Krull rings]Let $A$ be a ring and let $\varpi\in A$ be a non-zero-divisor and a non-unit. We say that $A$ is $\varpi$-Shilov if for every weakly associated prime ideal $\mathfrak{p}$ of $\varpi$ in $A$, the $\varpi$-adic completion of the local ring $A_{\mathfrak{p}}$ is a valuation ring. In other words, $A$ is $\varpi$-Shilov if for every weakly associated prime $\mathfrak{p}$ of $\varpi$ in $A$ the local ring $A_{\mathfrak{p}}$ is $\varpi$-valuative, i.e., the $\varpi$-adically separated quotient $A_{\mathfrak{p}}/\bigcap_{n\geq1}\varpi^{n}A_{\mathfrak{p}}$ is a valuation ring (cf.~Proposition~\ref{Valuative rings and valuation rings} and Corollary~\ref{Valuative rings and completion}). 

We say that $A$ is $\varpi$-Krull if, for every weakly associated prime ideal $\mathfrak{p}$ of $\varpi$ in $A$, the local ring $A_{\mathfrak{p}}$ itself is a $\varpi$-adically separated valuation ring.\end{mydef}
\begin{rmk}Note that, by Corollary~\ref{Valuative rings and completion 2}, the $\varpi$-adic completion of the local ring $A_{\mathfrak{p}}$ at a prime ideal $\mathfrak{p}$ containing $\varpi$ is a valuation ring of rank $1$ if and only if the $\varpi$-adic completion is a valuation ring and the prime ideal $\mathfrak{p}$ is minimal over $\varpi$. In particular, every $\varpi$-Shilov (respectively, $\varpi$-Krull) ring is also weakly $\varpi$-Shilov (respectively, weakly $\varpi$-Krull). It is also clear that every strongly $\varpi$-Shilov ring (respectively, strongly $\varpi$-Krull ring) is $\varpi$-Shilov (respectively, $\varpi$-Krull).\end{rmk} 
\begin{lemma}\label{Completion and local rings}Let $\varpi$ be a non-zero-divisor and a non-unit in a ring $A$ and let $\widehat{A}$ be the $\varpi$-adic completion of $A$. For every prime ideal $\mathfrak{p}$ of $A$ containing $\varpi$, the canonical map $A\to \widehat{A}$ induces an isomorphism between the $\varpi$-adic completion $\widehat{A_{\mathfrak{p}}}$ of $A_{\mathfrak{p}}$ and the $\varpi$-adic completion of $\widehat{A}_{\mathfrak{p}\widehat{A}}$.\end{lemma}
\begin{proof}It suffices to prove that \begin{equation*}\widehat{A}_{\mathfrak{p}\widehat{A}}/\varpi^{n}\widehat{A}_{\mathfrak{p}\widehat{A}}\cong A_{\mathfrak{p}}/\varpi^{n}A_{\mathfrak{p}}\end{equation*}for all $n\geq1$ and all prime ideals $\mathfrak{p}$ of $A$ containing $\varpi$. We calculate: \begin{align*}\widehat{A}_{\mathfrak{p}\widehat{A}}/\varpi^{n}\widehat{A}_{\mathfrak{p}\widehat{A}}\cong (\widehat{A}/\varpi^{n}\widehat{A})_{(\mathfrak{p}\widehat{A}/\varpi^{n}\widehat{A})} \\ \cong (A/\varpi^{n}A)_{(\mathfrak{p}/\varpi^{n}A)}\cong A_{\mathfrak{p}}/\varpi^{n}A_{\mathfrak{p}}.\end{align*}\end{proof} 
\begin{lemma}\label{Weakly Krull rings and completion}Let $\varpi\in A$ be a non-zero-divisor and non-unit in a ring $A$. Then $A$ is weakly $\varpi$-Shilov if and only if the $\varpi$-adic completion $\widehat{A}$ of $A$ is weakly $\varpi$-Shilov.\end{lemma}
\begin{proof}The map $A\to \widehat{A}$ induces a homeomorphism $\Spf(\widehat{A})\simeq\Spf(A)$. In particular, every prime ideal $\mathfrak{p}$ of $\widehat{A}$ minimal over $\varpi$ is of the form $\mathfrak{p}=(\mathfrak{p}\cap A)\widehat{A}$ and $\mathfrak{p}\cap A$ is a prime ideal of $A$ minimal over $\varpi$. Hence the assertion follows from Lemma~\ref{Completion and local rings}.\end{proof}
\begin{lemma}\label{Strongly Krull rings and completion}Let $\varpi\in A$ be a non-zero-divisor and non-unit in a ring $A$. Then $A$ is $\varpi$-Shilov (respectively, strongly $\varpi$-Shilov) if and only if the $\varpi$-adic completion $\widehat{A}$ of $A$ is $\varpi$-Shilov (respectively, strongly $\varpi$-Shilov).\end{lemma}
\begin{proof}Let $\mathfrak{p}$ be a weakly associated prime ideal of $\varpi$ in $\widehat{A}$. Let $f\in\widehat{A}$ such that $\mathfrak{p}$ is minimal over $J=(\varpi)_{\widehat{A}}:_{\widehat{A}}f$ in $\widehat{A}$. Choose $g\in A$ with $f-g\in (\varpi)_{\widehat{A}}$. Then $J=(\varpi)_{\widehat{A}}:_{\widehat{A}}g$. Moreover, since $J$ is open with respect to the $\varpi$-adic topology on $\widehat{A}$, it is equal to the (topological) closure in $\widehat{A}$ of \begin{equation*}J\cap A=((\varpi)_{\widehat{A}}\cap A):_{A}g=(\varpi)_{A}:_{A}g.\end{equation*}Since the restriction map $\Spf(\widehat{A})\to\Spf(A)$ is a homeomorphism, $\mathfrak{p}$ being minimal over $J$ implies that the prime ideal $\mathfrak{p}\cap A$ of $A$ is minimal over $J\cap A=(\varpi)_{A}:_{A}g$. Thus we have shown that the restriction $\mathfrak{p}\cap A$ to $A$ of any weakly associated prime of $\varpi$ in $\widehat{A}$ is a weakly associated prime of $\varpi$ in $A$. Since for any open prime ideal $\mathfrak{p}\subsetneq \widehat{A}$ we have $\mathfrak{p}=(\mathfrak{p}\cap A)\widehat{A}$, the assertion of the lemma follows from Lemma~\ref{Completion and local rings}.\end{proof}
On the other hand, the weak and strong $\varpi$-Krull properties are not preserved under $\varpi$-adic completion, as is demonstrated by the following example due to Heitmann and Ma. 
\begin{example}[cf.~Heitmann-Ma, \cite{Heitmann-Ma26}~Proposition~ 4.10]\label{Example of Heitmann-Ma}Let $R$ be a complete Noetherian local domain of mixed characteristic $(0, p)$ of dimension $\dim(R)\geq2$ and let $\widehat{R^{+}}$ be the $p$-adic completion of the absolute integral closure $R^{+}$ of $R$. Heitmann proved in \cite{Heitmann22} that $\widehat{R^{+}}$ is an integral domain. Moreover, by Corollary~\ref{Absolute integral closure} below, $R^{+}$ is strongly $p$-Krull and $\widehat{R^{+}}$ is strongly $p$-Shilov. In particular, it is weakly $p$-Shilov, i.e., for every minimal prime ideal $\mathfrak{p}$ of $p$ in $\widehat{R^{+}}$, the $p$-adic completion of the local ring at $\mathfrak{p}$ is a valuation ring of rank one (note that the weak $p$-Shilov property in this example can also be proved by elementary means, without reference to Corollary~\ref{Absolute integral closure}). However, for the $p$-normalized rank one valuation $v$ corresponding to any minimal prime $\mathfrak{p}$ of $p$, Heitmann and Ma constructed a non-zero element $g\in \widehat{R^{+}}$ such that $v(g)=0$, see \cite{Heitmann-Ma26}~Proposition~4.10. This means that (the image of) $g$ is a non-zero element of $(\widehat{R^{+}})_{\mathfrak{p}}$ satisfying\begin{equation*}g\in \bigcap_{n\geq1}p^{n}(\widehat{R^{+}})_{\mathfrak{p}}.\end{equation*}In particular, the ring $\widehat{R^{+}}$ is not weakly $p$-Krull and, a fortiori, not strongly $p$-Krull.\end{example}
Our interest in (strongly) $\varpi$-Shilov rings is motivated by the following property of weakly associated prime ideals.
\begin{prop}\label{Weakly associated prime ideals and boundaries}Let $\varpi$ be a non-zero-divisor and a non-unit in a ring $A$. Then, for every $n\geq0$, 
\begin{equation*}(\varpi^{n})_{A}=\bigcap_{\mathfrak{p}\in\wAss(\varpi)}\varphi_{\mathfrak{p}}^{-1}(\varpi^{n}A_{\mathfrak{p}})=\bigcap_{\mathfrak{p}\in\wAss(\varpi)}\widehat{\varphi_{\mathfrak{p}}}^{-1}(\varpi^{n}\widehat{A_{\mathfrak{p}}}),\end{equation*}where $\wAss(\varpi)$ is the set of weakly associated prime ideals of $\varpi$ and where, for every prime $\mathfrak{p}$, we denote by $\varphi_{\mathfrak{p}}$ and $\widehat{\varphi_{\mathfrak{p}}}$ the canonical maps $A[\varpi^{-1}]\to A_{\mathfrak{p}}[\varpi^{-1}]$ and $A[\varpi^{-1}]\to \widehat{A_{\mathfrak{p}}}[\varpi^{-1}]$, with $\widehat{A_{\mathfrak{p}}}$ being the $\varpi$-adic completion of $A_{\mathfrak{p}}$.\end{prop}
\begin{proof}The case $n=0$ follows from the case $n\geq1$ since for any ring $B$ with a non-zero-divisor $\varpi\in B$ an element $\frac{f}{\varpi^{n}}\in B[\varpi^{-1}]$ (with $f\in B$, $n\geq1$) belongs to $B$ if and only if $f$ belongs to $(\varpi^{n})_{B}$. Thus, given $f\in A$, $n\geq1$ with $f\not\in(\varpi^{n})_{A}$, we want to prove that there exists some $\mathfrak{p}\in\wAss(\varpi)$ such that the image of $f$ in $A_{\mathfrak{p}}$ does not belong to $(\varpi^{n})_{A_{\mathfrak{p}}}$ (note that \begin{equation*}(\varpi^{n})_{\widehat{A_{\mathfrak{p}}}}\cap A_{\mathfrak{p}}=(\varpi^{n})_{A_{\mathfrak{p}}}\end{equation*}for all $n\geq1$, so it suffices to consider $A_{\mathfrak{p}}$ instead of $\widehat{A_{\mathfrak{p}}}$). If $f\not\in(\varpi^{n})_{A}$, then $(\varpi^{n})_{A}:_{A}f$ is a proper ideal of $A$. Let $\mathfrak{p}$ be a prime ideal of $A$ minimal over $(\varpi^{n})_{A}:_{A}f$. Then \begin{equation*}(\varpi^{n})_{A_{\mathfrak{p}}}:_{A_{\mathfrak{p}}}f\subseteq \mathfrak{p}A_{\mathfrak{p}},\end{equation*}so $(\varpi^{n})_{A_{\mathfrak{p}}}:_{A_{\mathfrak{p}}}f$ is a proper ideal of $A_{\mathfrak{p}}$. It follows that $f\not\in \varpi^{n}A_{\mathfrak{p}}$. By \cite{Brewer-Heinzer}~Theorem~3, the weakly associated prime ideal $\mathfrak{p}$ of $\varpi^{n}$ is also a weakly associated prime ideal of $\varpi$. The assertion follows.\end{proof}
The above proposition also shows that the strong $\varpi$-Shilov property imposes rather strong algebraic conditions on the ring $A$ in question (or, more precisely, on the ring extension $A\subseteq A[\varpi^{-1}]$), as is exemplified by the following corollary.  
\begin{cor}\label{Strongly Shilov implies completely integrally closed}Let $A$ be a ring which is strongly $\varpi$-Shilov for some non-zero-divisor and non-unit $\varpi$. Then $A$ is completely integrally closed in $\mathcal{A}=A[\varpi^{-1}]$.\end{cor}
\begin{proof}By the proposition, \begin{equation*}A=\bigcap_{\mathfrak{p}\in\wAss(\varpi)}\varphi_{\mathfrak{p}}^{-1}(\widehat{A_{\mathfrak{p}}}),\end{equation*}where $\varphi_{\mathfrak{p}}: A[\varpi^{-1}]\to \widehat{A_{\mathfrak{p}}}[\varpi^{-1}]$ are the canonical maps. By the hypothesis that $A$ is strongly $\varpi$-Shilov, the rings $\widehat{A_{\mathfrak{p}}}$ for $\mathfrak{p}\in\wAss(\varpi)$ are valuation rings of rank $1$, so $A$ is equal to the subring of $A[\varpi^{-1}]$ where the corresponding continuous rank $1$ valuations are $\leq1$. We conclude by applying \cite{Berkovich}~Theorem 1.3.1, that $A$ is equal to the closed unit ball for the spectral seminorm $\vert\cdot\vert_{\spc,\varpi}$ on $A[\varpi^{-1}]$ derived from the canonical extension of the $\varpi$-adic seminorm on $A$. \end{proof}
Proposition~\ref{Weakly associated prime ideals and boundaries} should also be compared to the following fact about minimal prime ideals of $\varpi$. We denote by $\Min_{A}(\varpi)$, or simply $\Min(\varpi)$, the set of minimal prime ideals of a non-zero-divisor and non-unit $\varpi$ in a ring $A$. The author is indebted to Linquan Ma for the idea of using prime avoidance to establish condition S1 for $A/(\varpi)_{A}$ when $A/(\varpi)_{A}$ has finitely many minimal primes.
\begin{prop}\label{Minimal prime ideals and boundaries}Let $\varpi$ be a non-zero-divisor and a non-unit in a ring $A$. Suppose that $A/(\varpi)_{A}$ has only finitely many minimal prime ideals (in other words, there are only finitely many minimal prime ideals of $\varpi$ in $A$) and that \begin{equation*}A=\bigcap_{\mathfrak{p}\in\Min(\varpi)}\varphi_{\mathfrak{p}}^{-1}(A_{\mathfrak{p}})\end{equation*}or \begin{equation*}A=\bigcap_{\mathfrak{p}\in\Min(\varpi)}\widehat{\varphi_{\mathfrak{p}}}^{-1}(\widehat{A_{\mathfrak{p}}}),\end{equation*}where $\varphi_{\mathfrak{p}}$, $\widehat{\varphi_{\mathfrak{p}}}$ are as in Proposition~\ref{Weakly associated prime ideals and boundaries}. Then $A/(\varpi)_{A}$ satisfies condition S1, i.e., every weakly associated prime ideal of $\varpi$ in $A$ is minimal over $\varpi$.\end{prop}
\begin{proof}Note that the two displayed equalities are actually equivalent, since $\widehat{A_{\mathfrak{p}}}\cap A_{\mathfrak{p}}[\varpi^{-1}]=A_{\mathfrak{p}}$ for every $\mathfrak{p}$. By contradiction, suppose that $A/(\varpi)_{A}$ does not satisfy S1. Let $\mathfrak{q}$ be a weakly associated prime ideal of $\varpi$ which is not minimal over $\varpi$. Then $\mathfrak{q}$ is not contained in any minimal prime ideal of $\varpi$ in $A$. By assumption, there exist only finitely many minimal prime ideals of $\varpi$, so, by prime avoidance, \begin{equation*}\mathfrak{q}\not\subseteq \bigcup_{\mathfrak{p}\in \Min_{A}(\varpi)}\mathfrak{p}.\end{equation*}By \cite{Stacks}, Tag~05C3, the union of all weakly associated prime ideals of $\varpi$ is precisely the set of elements $g\in A$ which are zero-divisors on $A/(\varpi)_{A}$. Hence there exists an element $g\in A$ which is a zero-divisor on $A/(\varpi)_{A}$, but which does not belong to any minimal prime ideal of $\varpi$ in $A$. Let $f\in A\setminus (\varpi)_{A}$ such that $gf\in (\varpi)_{A}$. Then \begin{equation*}(\varpi)_{A}:_{A}f\not\subseteq \bigcup_{\mathfrak{p}\in\Min_{A}(\varpi)}\mathfrak{p}\end{equation*}and thus \begin{equation*}((\varpi)_{A}:_{A}f)A_{\mathfrak{p}}=A_{\mathfrak{p}}\end{equation*}for every $\mathfrak{p}\in\Min_{A}(\varpi)$. It follows that the image of $\frac{f}{\varpi}$ in $A_{\mathfrak{p}}[\varpi^{-1}]$ belongs to $A_{\mathfrak{p}}$ for all $\mathfrak{p}\in\Min_{A}(\varpi)$. By our assumption on $A$, this entails $\frac{f}{\varpi}\in A$. This provides the desired contradiction since we had $f\not\in (\varpi)_{A}$. \end{proof}
The above results are already reminiscent of the description of the notion of boundary for a Tate ring $\mathcal{A}$ or a Tate Huber pair $(\mathcal{A}, \mathcal{A}^{+})$ which we gave in Section~\ref{sec:Rees valuations}. However, before we describe the Shilov boundaries of Tate rings in this fashion, let us verify that strongly $\varpi$-Shilov rings are actually very common in commutative algebra. This rests on the following theorem, which is the first main result of this section. Our proof of this theorem was inspired by a comment of the anonymous referee as well as by our discussions with Linquan Ma. 
\begin{thm}\label{Spectral seminorm and S1}Let $\varpi\in A$ be a non-zero-divisor and a non-unit in a ring $A$. Assume that $A$ is completely integrally closed in $\mathcal{A}$, that $A/(\varpi)_{A}$ has only finitely many minimal prime ideals and that for every element $f\in \mathcal{A}$ with $\vert f\vert_{\spc,\varpi}\neq0$ there exist an integer $n\geq1$ and a unit element $c\in \mathcal{A}^{\times}$ such that $c$ is multiplicative with respect to the spectral seminorm $\vert\cdot\vert_{\spc,\varpi}$ and \begin{equation*}\vert cf^{n}\vert_{\spc,\varpi}=1.\end{equation*}Then $A/(\varpi)_{A}$ satisfies condition S1.\end{thm}
\begin{proof}By contradiction, suppose that $A/(\varpi)_{A}$ does not satisfy S1. Let $\mathfrak{q}$ be a weakly associated prime ideal of $\varpi$ which is not minimal over $\varpi$. By the same prime avoidance argument as in the proof of Proposition~\ref{Minimal prime ideals and boundaries}, we find an element $g\in A$ which is a zero-divisor on $A/(\varpi)_{A}$, but which does not belong to any minimal prime ideal of $\varpi$ in $A$. Let $f\in A\setminus (\varpi)_{A}$ such that $gf\in (\varpi)_{A}$. Note that $g^{n}f^{n}\in (\varpi^{n})_{A}$ for all $n\geq1$. If $\vert f\vert_{\spc,\varpi}\leq \frac{1}{2}=\vert \varpi\vert_{\spc,\varpi}$, then $\vert \varpi^{-1}f\vert_{\spc,\varpi}\leq1$ and thus $\varpi^{-1}f\in A$ by the assumption that $A$ is completely integrally closed in $\mathcal{A}$. This contradicts the assumption that $f\not\in (\varpi)_{A}$. Hence $\vert f\vert_{\spc,\varpi}>\frac{1}{2}$, so, in particular, $\vert f\vert_{\spc,\varpi}\neq0$. 

Choose an integer $n\geq1$ and a unit element $c\in \mathcal{A}^{\times}$ such that $c$ is multiplicative with respect to $\vert\cdot\vert_{\spc,\varpi}$ and $\vert cf^{n}\vert_{\spc,\varpi}=1$. Since $\vert f\vert_{\spc,\varpi}>\frac{1}{2}$, we have $\vert f^{n}\vert_{\spc,\varpi}=\vert f\vert_{\spc,\varpi}^{n}>2^{-n}$ and thus $\vert c\vert_{\spc,\varpi}<2^{n}$, the element $c$ being multiplicative with respect to $\vert\cdot\vert_{\spc,\varpi}$. Since $A$ is completely integrally closed in $\mathcal{A}$, we know that $cf^{n}$ is an element of $A$. We estimate: \begin{equation*}\vert g^{n}(cf^{n})\vert_{\spc,\varpi}=\vert c\vert_{\spc,\varpi}\vert g^{n}f^{n}\vert_{\spc,\varpi}\leq \vert c\vert_{\spc,\varpi}2^{-n}<1.\end{equation*}This implies \begin{equation*}g^{n}(cf^{n})\in \sqrt{(\varpi)_{A}}.\end{equation*}In other words, \begin{equation*}g^{n}(cf^{n})\in \mathfrak{p}\end{equation*}for all minimal prime ideals $\mathfrak{p}$ of $\varpi$ in $A$. On the other hand, $\vert cf^{n}\vert_{\spc,\varpi}=1$ means that $cf^{n}\not\in \sqrt{(\varpi)_{A}}$, so there exists a minimal prime ideal $\mathfrak{p}$ of $\varpi$ such that $cf^{n}\not\in \mathfrak{p}$. But then $g\in \mathfrak{p}$, which contradicts our assumption that $g$ does not belong to any minimal prime ideal of $\varpi$.\end{proof}
This applies, for example, to rings of power-bounded elements in classical affinoid algebras.
\begin{cor}\label{Affinoid algebras and S1}Let $K$ be a nonarchimedean field with pseudo-uniformizer $\varpi$ and let $\mathcal{A}$ be a reduced affinoid $K$-algebra in the sense of Tate. Then $\mathcal{A}^{\circ}/(\varpi)_{\mathcal{A}^{\circ}}$ satisfies condition S1.\end{cor}
\begin{proof}By \cite{BGR}~Theorem~6.2.4/1, every reduced $K$-affinoid algebra $\mathcal{A}$ is uniform (as a Tate ring), so its subring of power-bounded elements $\mathcal{A}^{\circ}$ is completely integrally closed by \cite{Dine22}~Lemma~2.24. By \cite{BGR}~Proposition~6.2.1/4(ii), the spectral seminorm $\vert\cdot\vert_{\spc,\varpi}$ on $\mathcal{A}$ has the rescaling property stated in Theorem~\ref{Spectral seminorm and S1}. Moreover, the nilreduction of $\mathcal{A}^{\circ}/(\varpi)_{\mathcal{A}^{\circ}}$ is a finite-type algebra over the residue field of $K$, so $\mathcal{A}^{\circ}/(\varpi)_{\mathcal{A}^{\circ}}$ has only finitely many minimal prime ideals. Therefore, $\mathcal{A}^{\circ}$ satisfies all hypotheses of Theorem~\ref{Spectral seminorm and S1}.\end{proof} 
\begin{cor}\label{Spectral seminorm and S1, rational case}Let $\varpi\in A$ be a non-zero-divisor and non-unit in a ring $A$ such that $A$ is completely integrally closed in $\mathcal{A}=A[\varpi^{-1}]$ and such that $A/(\varpi)_{A}$ has only finitely many minimal prime ideals. If the spectral seminorm $\vert\cdot\vert_{\spc,\varpi}$ on $\mathcal{A}$ takes values in $2^{\mathbb{Q}}\cup\{0\}$, then $A/(\varpi)_{A}$ satisfies condition S1.\end{cor}
\begin{proof}If $\vert\cdot\vert_{\spc,\varpi}$ takes values in $2^{\mathbb{Q}}\cup\{0\}$, then for every $f\in \mathcal{A}$ with $\vert f\vert_{\spc,\varpi}\neq0$ there exist integers $n, m$ with $n\geq1$ such that \begin{equation*}\vert f\vert_{\spc,\varpi}^{n}=2^{-m}=\vert \varpi^{m}\vert_{\spc,\varpi}.\end{equation*}Then we can set $c=\varpi^{-m}$ in Theorem~\ref{Spectral seminorm and S1} and see that the assumptions of that theorem are satisfied.  \end{proof}
Applying the above corollary to rings which are ``close enough'' to being Noetherian, we can prove the following result. 
\begin{thm}\label{Reduced Noetherian rings are strongly Shilov}Let $\varpi\in A$ be a non-unit non-zero-divisor in a ring $A$ such that $A$ is integrally closed in $\mathcal{A}=A[\varpi^{-1}]$. If there exists an integral extension $A\hookrightarrow B$ such that $\varpi$ is a non-zero-divisor on $B$ and $B$ is a finite product of Krull domains, then $A$ is strongly $\varpi$-Krull. More precisely, $A/(\varpi)_{A}$ satisfies S1 and all local rings of $A$ at minimal prime ideals of $\varpi$ are DVRs. 

In particular, every ring $A$ which is an integral extension of a reduced Noetherian ring inside its total ring of fractions is strongly $\varpi$-Krull for every non-zero-divisor and non-unit $\varpi\in A$ such that $A$ is integrally closed in $\mathcal{A}=A[\varpi^{-1}]$. \end{thm}
In order to apply Theorem~\ref{Spectral seminorm and S1} to the proof of Theorem~\ref{Reduced Noetherian rings are strongly Shilov}, we first verify a few elementary lemmas about spectral seminorms. 
\begin{lemma}\label{Finite products and spectral seminorm}Let $A=A_{1}\times\dots\times A_{k}$ be a finite product of rings $A_{i}$ and let $\varpi\in A$ be a non-zero-divisor and non-unit in $A$. Let $\varpi_{i}$ denote the image of $\varpi$ in $A_{i}$ and put $\mathcal{A}_{i}=A_{i}[\varpi^{-1}]$ as well as $\mathcal{A}=A[\varpi^{-1}]=\mathcal{A}_{1}\times\dots \mathcal{A}_{k}$. Then the spectral seminorm $\vert\cdot\vert_{\spc,\varpi}$ on $\mathcal{A}$ is given by \begin{equation*}\vert (f_{1},\dots, f_{k})\vert_{\spc,\varpi}=\max_{1\leq i\leq k}\vert f_{i}\vert_{\spc,\varpi_{i}}.\end{equation*}In particular, if each of $\vert\cdot\vert_{\spc,\varpi_{i}}$ takes values in $2^{\mathbb{Q}}\cup\{0\}$, so does $\vert\cdot\vert_{\spc,\varpi}$. \end{lemma}
\begin{proof}It is readily seen that the canonical extension $\lVert\cdot\rVert_{A,\varpi}$ of the $\varpi$-adic seminorm on $A$ is given by \begin{equation*}\lVert (f_1,\dots, f_k)\rVert_{A,\varpi}=\max_{1\leq i\leq k}\lVert f_{i}\rVert_{A_{i},\varpi_{i}},\end{equation*}hence the continuous maps \begin{equation*}\mathcal{M}(\mathcal{A}_{i})\to \mathcal{M}(\mathcal{A})\end{equation*}induced by the projections $\mathcal{A}\to \mathcal{A}_{i}$ define a homeomorphism \begin{equation*}\coprod_{i=1}^{k}\mathcal{M}(\mathcal{A}_{i})\tilde{\to}\mathcal{M}(\mathcal{A}).\end{equation*}In particular, for any $f=(f_1,\dots, f_k)\in \mathcal{A}$, \begin{equation*}\vert f\vert_{\spc,\varpi}=\max_{v\in \mathcal{M}(\mathcal{A})}v(f)=\max_{1\leq i\leq n}\max_{v\in \mathcal{M}(\mathcal{A}_{i})}v(f_{i})=\max_{1\leq i\leq n}\vert f_{i}\vert_{\spc,\varpi}.\end{equation*}The assertion follows.\end{proof}
\begin{lemma}\label{Krull domains and the spectral seminorm}Let $A$ be a Krull domain and let $\varpi\in A$ be a non-zero non-unit. Let $\mathfrak{p}_{1},\dots, \mathfrak{p}_{n}$ be all the height one prime ideals of $A$ containing $\varpi$ and let $v_{1},\dots, v_{n}$ be the corresponding discrete valuations on $A_{\mathfrak{p}_{i}}[\varpi^{-1}]$, written multiplicatively, not additively, and normalized by $v_{i}(\varpi_{i})=\frac{1}{2}$ for a uniformizer $\varpi_{i}$ of $A_{\mathfrak{p}_{i}}$. For every $i=1,\dots, n$, let $N_{i}\geq1$ be an integer such that $\varpi$ is a unit multiple of $\varpi_{i}^{N_{i}}$ in $A_{\mathfrak{p}_{i}}$. Then \begin{equation*}\vert f\vert_{\spc,\varpi}=\max_{1\leq i\leq n}v_{i}(f)^{1/N_{i}}\end{equation*}for all $f\in \mathcal{A}=A[\varpi^{-1}]$. In particular, the spectral seminorm $\vert\cdot\vert_{\spc,\varpi}$ takes values in $2^{\mathbb{Q}}\cup\{0\}$.\end{lemma}
\begin{proof}By the choice of $N_{i}$, we have $v_{i}(\varpi)=v_{i}(\varpi_{i}^{N_{i}})=(\frac{1}{2})^{N_{i}}$. Hence the rank-$1$ valuations $v_{i}^{1/N_{i}}$ on $A_{\mathfrak{p}_{i}}[\varpi^{-1}]$ satisfy \begin{equation*}v_{i}^{1/N_{i}}(\varpi)=\frac{1}{2}=\lVert\varpi\rVert_{A_{\mathfrak{p}_{i}}, \varpi}.\end{equation*}It follows by Lemma~\ref{Boundedness vs. continuity 2} that $v_{i}^{1/N_{i}}$ is equal to the spectral seminorm derived from the canonical extension $\lVert \cdot\rVert_{A_{\mathfrak{p}_{i}}, \varpi}$ of the $\varpi$-adic seminorm on $A_{\mathfrak{p}_{i}}$. Therefore, it remains to prove that the maximum of the spectral seminorms on $A_{\mathfrak{p}_{i}}[\varpi^{-1}]$ (for $i=1,\dots, n$) restricts to the spectral seminorm $\vert\cdot\vert_{\spc,\varpi}$ on $\mathcal{A}$. By Lemma~\ref{Finite products and spectral seminorm}, this is equivalent to proving that the map \begin{equation*}\mathcal{A}=A[\varpi^{-1}]\hookrightarrow \mathcal{B}=\prod_{i=1}^{n}A_{\mathfrak{p}_{i}}[\varpi^{-1}]\end{equation*}is an isometry when the source and target are endowed with their respective spectral seminorms $\vert\cdot\vert_{\spc,\varpi}$. By definition of the spectral seminorm, it suffices to prove that the canonical extension $\lVert\cdot\rVert_{B,\varpi}$ of the $\varpi$-adic seminorm on $B=\prod_{i=1}^{n}A_{\mathfrak{p}_{i}}$ restricts to the canonical extension $\lVert\cdot\rVert_{A,\varpi}$ of the $\varpi$-adic seminorm on $A$. Since \begin{equation*}\lVert\cdot\rVert_{B,\varpi}=\max_{1\leq i\leq n}\lVert\cdot\rVert_{A_{\mathfrak{p}_{i}},\varpi},\end{equation*}it suffices to prove that the map \begin{equation*}A/(\varpi^{k})_{A}\to \prod_{i=1}^{n}A_{\mathfrak{p}_{i}}/(\varpi^{k})_{A_{\mathfrak{p}_{i}}}\end{equation*}is injective for every $k\geq1$. To this end, let $f\in A$ be an element such that $f\in \varpi^{k}A_{\mathfrak{p}_{i}}$ for all $i=1,\dots, n$. Then $\frac{f}{\varpi^{k}}\in A_{\mathfrak{p}_{i}}$ for all $i=1,\dots, n$. But then $\frac{f}{\varpi^{k}}\in A_{\mathfrak{p}}$ for all height one prime ideals of $A$, since for every height one prime ideal $\mathfrak{p}\not\in \{\mathfrak{p}_{1},\dots, \mathfrak{p}_{n}\}$ the element $\varpi$ is a unit in $A_{\mathfrak{p}}$. By the definition of a Krull domain, \begin{equation*}A=\bigcap_{\height(\mathfrak{p})=1}A_{\mathfrak{p}}.\end{equation*}Hence $\frac{f}{\varpi^{k}}\in A$ or equivalently $f\in \varpi^{k}A$. This concludes the proof of the lemma.
\end{proof}
\begin{lemma}\label{Spectral isometries}Let $A\hookrightarrow B$ be an integral extension of rings, let $\varpi\in A$ be a non-unit non-zero-divisor in $A$ and in $B$ and consider the Tate rings $\mathcal{A}=A[\varpi^{-1}]$, $\mathcal{B}=B[\varpi^{-1}]$. Then the map $\mathcal{A}\to \mathcal{B}$ is an isometry when the source and target are endowed with their respective spectral seminorms $\vert\cdot\vert_{\spc,\varpi}$.\end{lemma}
\begin{proof}Let $\mathcal{A}^{+}$, $\mathcal{B}^{+}$ denote the integral closures of $A$, $B$ in $\mathcal{A}$, $\mathcal{B}$, respectively. The Tate ring $\mathcal{A}^{+}[\varpi^{-1}]$, with pair of definition $(\mathcal{A}^{+}, \varpi)$, has the same Berkovich spectrum as $\mathcal{A}$ (this follows, for example, from Lemma~\ref{Continuous means normalized}), and similarly for $\mathcal{B}^{+}[\varpi^{-1}]$ and $\mathcal{B}$. In particular, $\mathcal{A}^{+}[\varpi^{-1}]$ and $\mathcal{A}$ (respectively, $\mathcal{B}^{+}[\varpi^{-1}]$ and $\mathcal{B}$) have the same spectral seminorm. This reduces the lemma to the special case when $A$, $B$ are integrally closed in $\mathcal{A}$, $\mathcal{B}$. In this case, Proposition~\ref{Rings of integral elements} and \cite{Swanson-Huneke}~Proposition~1.6.1, entail that $A\to B$ is an isometry for the respective $\varpi$-adic seminorms. This formally implies that $\mathcal{A}\to \mathcal{B}$ is an isometry for the canonical extensions $\lVert\cdot\rVert_{A,\varpi}$ and $\rVert\cdot\rVert_{B,\varpi}$ of the $\varpi$-adic seminorms of $A$ and $B$. We conclude by Lemma~\ref{Bounded implies submetric}. \end{proof}
\begin{proof}[Proof of Theorem~\ref{Reduced Noetherian rings are strongly Shilov}]The last assertion of the theorem (about reduced Noetherian rings) follows from the rest by the Nagata-Mori Theorem (see, for example, \cite{Swanson-Huneke}~Theorem~4.10.5).

Let $A\hookrightarrow B$ be an integral extension where $B=B_{1}\times\dots \times B_{k}$ is a finite product of Krull domains $B_{i}$ and let $\varpi\in A$ be a non-zero-divisor and non-unit on $A$ and $B$. We first prove that the local rings of $A$ at minimal prime ideals of $\varpi$ are DVRs (so, in particular, $A$ is weakly $\varpi$-Krull). Let $\mathfrak{p}$ be a minimal prime ideal of $\varpi$ in $A$ and consider the integral extension $A_{\mathfrak{p}}\hookrightarrow B_{\mathfrak{p}}$. Every maximal ideal of $B_{\mathfrak{p}}$ lies over $\mathfrak{p}A_{\mathfrak{p}}$, so every maximal ideal of $B_{\mathfrak{p}}$ is minimal over $\varpi$. Since $B$ is a finite product of Krull domains, it has only finitely many minimal prime ideals over $\varpi$ and the local rings at these prime ideals are DVRs. Therefore, $B_{\mathfrak{p}}$ is a semi-local ring and the local ring of $B_{\mathfrak{p}}$ at every maximal ideal is a DVR. It follows that $B_{\mathfrak{p}}$ is of Krull dimension~$1$ and thus $A_{\mathfrak{p}}$ is also of Krull dimension~$1$. In particular, $A_{\mathfrak{p}}[\varpi^{-1}]$ is zero-dimensional, so it must be equal to the total ring of fractions of $A_{\mathfrak{p}}$. By assumption, $A$ is integrally closed in $A[\varpi^{-1}]$, so $A_{\mathfrak{p}}$ is also integrally closed in $A_{\mathfrak{p}}[\varpi^{-1}]$. We conclude that $A_{\mathfrak{p}}$ is integrally closed in its total ring of fractions. Since $B$ is a finite product of domains, $B_{\mathfrak{p}}$ has only finitely many minimal prime ideals. By lying over and going up, this implies that $A_{\mathfrak{p}}$ has only finitely many minimal prime ideals. This shows that $A_{\mathfrak{p}}$ is a domain since any integrally closed reduced local ring with only finitely many minimal prime ideals is a domain. Now let $\mathfrak{q}_{1},\dots, \mathfrak{q}_{n}$ be the prime ideals of $B$ lying over $\mathfrak{p}$. Then $B_{\mathfrak{p}}=\bigcap_{i=1}^{n}B_{\mathfrak{q}_{i}}$ (inside the total ring of fractions of $B_{\mathfrak{p}}$), so every element of $\Frac(A_{\mathfrak{p}})$ which lies in the intersection of the DVRs $B_{\mathfrak{q}_{i}}\cap \Frac(A_{\mathfrak{p}})$, $i=1,\dots, n$, lies in $B$ and is thus integral over $A_{\mathfrak{p}}$. Since we have seen that $A_{\mathfrak{p}}$ is integrally closed, \begin{equation*}A_{\mathfrak{p}}=\bigcap_{i=1}^{n}(B_{\mathfrak{q}_{i}}\cap \Frac(A_{\mathfrak{p}})).\end{equation*}On the other hand, each of the valuation rings $B_{\mathfrak{q}_{i}}\cap \Frac(A_{\mathfrak{p}})$ of the fraction field $\Frac(A_{\mathfrak{p}})$ is centered at the maximal ideal $\mathfrak{p}A_{\mathfrak{p}}$ of $A_{\mathfrak{p}}$ (since each $\mathfrak{q}_{i}$ lies over $\mathfrak{p}$). It follows by \cite{BourbakiAlg}~Ch.~VI, \S7, Proposition~1, that $A_{\mathfrak{p}}=B_{\mathfrak{q}_{i}}\cap \Frac(A_{\mathfrak{p}})$ for all $i=1,\dots, n$. In particular, $A_{\mathfrak{p}}$ is a DVR.  

It remains to prove that $A/(\varpi)_{A}$ satisfies condition S1. To this end, we verify that $A$ and $\varpi$ satisfy the assumptions of Corollary~\ref{Spectral seminorm and S1, rational case}. Firstly, we check that $A$ is completely integrally closed in $\mathcal{A}$. Every element $f$ of $\mathcal{A}$ which is almost integral over $A$ is also almost integral over $B$ and hence belongs to $B$ (every Krull domain is completely integrally closed, so the product of Krull domains $B$ is also completely integrally closed in its total ring of fractions). Since $B$ is integral over $A$, this entails that $f$ is integral over $A$ and then $f\in A$ by the assumption that $A$ is integrally closed in $\mathcal{A}$. Thus $A$ is completely integrally closed in $\mathcal{A}$, verifying the first assumption of Corollary~\ref{Spectral seminorm and S1, rational case}. Secondly, $B$ has only finitely many prime ideals of height $1$ containing $\varpi$ and hence only finitely many minimal prime ideals of $\varpi$. Since every minimal prime ideal of $\varpi$ in $A$ lifts to a minimal prime ideal of $\varpi$ in the integral extension $B$ of $A$, we see that there are only finitely many minimal prime ideals of $\varpi$ in $A$. Finally, it remains to verify that the spectral seminorm $\vert\cdot\vert_{\spc,\varpi}$ takes values in $2^{\mathbb{Q}}\cup\{0\}$. By Lemma~\ref{Spectral isometries}, this reduces to the same assertion for the spectral seminorm on $\mathcal{B}=B[\varpi^{-1}]$. By Lemma~\ref{Finite products and spectral seminorm}, it then suffices to check that the spectral seminorm $\vert\cdot\vert_{\spc,\varpi}$ on $\mathcal{B}_{i}=B_{i}[\varpi^{-1}]$ takes values in $2^{\mathbb{Q}}\cup\{0\}$ for every $i=1,\dots, k$. Since each $B_{i}$ is a Krull domain, this follows from Lemma~\ref{Krull domains and the spectral seminorm}. \end{proof} 
\begin{rmk}\label{Serre's criterion 2}By varying the non-unit non-zero-divisor $\varpi$ in Corollary~\ref{Strongly Shilov implies completely integrally closed} and Theorem~\ref{Reduced Noetherian rings are strongly Shilov}, we recover Serre's criterion for normality (\cite{Eisenbud}~Theorem~11.5) for reduced Noetherian rings (cf.~Remark~\ref{Serre's criterion}).\end{rmk}
We now relate the weakly and strongly $\varpi$-Shilov conditions to adic spectra (and Berkovich spectra) of the relevant Tate rings. By a generic point of a topological space we mean the generic point of some irreducible component of the space.
\begin{prop}\label{Weakly Shilov}Let $\varpi$ be a non-zero-divisor and non-unit in a ring $A$ and suppose that $A$ is integrally closed in $\mathcal{A}=A[\varpi^{-1}]$. Consider the specialization map \begin{equation*}\spc: \Spa(\mathcal{A}, A)\to \Spf(A), v\mapsto \spc(v):=\{\, f\in A\mid v(f)<1\,\}.\end{equation*}Then $A$ is weakly $\varpi$-Shilov if and only if for every generic point $\mathfrak{p}$ of $\Spf(A)$ the pre-image $\spc^{-1}(\{\mathfrak{p}\})$ consists of a single point. In this case, for every generic point $\mathfrak{p}\in\Spf(A)$, the map $\mathcal{A}=A[\varpi^{-1}]\to \widehat{A_{\mathfrak{p}}}[\varpi^{-1}]$ induces a bijection \begin{equation*}\Spa(\widehat{A_{\mathfrak{p}}}[\varpi^{-1}], \widehat{A_{\mathfrak{p}}})\tilde{\to}\spc^{-1}(\{\mathfrak{p}\}).\end{equation*}\end{prop}
We first establish the following general proposition.
\begin{prop}\label{Local rings and fibres of the specialization map}Let $A$ be a ring and let $\varpi\in A$ be a non-zero-divisor and non-unit such that $A$ is integrally closed in $A[\varpi^{-1}]$. Then the pre-image $\spc^{-1}(\{\mathfrak{p}\})$ of any $\mathfrak{p}\in\Spf(A)$ is contained in the image of the map \begin{equation*}\Spa(\widehat{A_{\mathfrak{p}}}[\varpi^{-1}], \widehat{A_{\mathfrak{p}}})\to \Spa(A[\varpi^{-1}], A)\end{equation*}induced by \begin{equation*}A[\varpi^{-1}]\to \widehat{A_{\mathfrak{p}}}[\varpi^{-1}].\end{equation*} \end{prop}
\begin{proof}Let $v$ be a continuous valuation on $A[\varpi^{-1}]$ with $v(f)\leq1$ for all $f\in A$ and with $\spc(v)=\{\, f\in A\mid v(f)<1\,\}=\mathfrak{p}$. Then $v$ extends to a valuation $v_{\mathfrak{p}}$ on $A_{\mathfrak{p}}[\varpi^{-1}]$ by setting \begin{equation*}v_{\mathfrak{p}}(\frac{f}{s})=v(f)\end{equation*}for all $f\in A$, $s\in A\setminus\mathfrak{p}$ (this is well-defined since $v(s)=1$ for all $s\in A\setminus\mathfrak{p}$). If $\frac{f}{s}$ ($f\in A[\varpi^{-1}]$, $s\in A\setminus\mathfrak{p}$) belongs to $\varpi^{n}A_{\mathfrak{p}}$, then \begin{equation*}v_{\mathfrak{p}}(\frac{f}{s})=v(f)\leq v(\varpi^{n}),\end{equation*}showing that $v_{\mathfrak{p}}$ is a continuous valuation on the Tate ring $A_{\mathfrak{p}}[\varpi^{-1}]$. Thus $v_{\mathfrak{p}}$ extends to a continuous valuation on $\widehat{A_{\mathfrak{p}}}[\varpi^{-1}]$ which is bounded above by $1$ on $\widehat{A_{\mathfrak{p}}}$ and whose restriction along the map $A[\varpi^{-1}]\to \widehat{A_{\mathfrak{p}}}[\varpi^{-1}]$ is $v$.\end{proof}
\begin{proof}[Proof of Proposition~\ref{Weakly Shilov}]By Lemma~\ref{Valuative rings and adic spectrum}, it suffices to prove that for any $A$ and $\varpi$ such that $A$ is integrally closed in $A[\varpi^{-1}]$, there exists a canonical bijection \begin{equation*}\Spa(\widehat{A_{\mathfrak{p}}}[\varpi^{-1}], \widehat{A_{\mathfrak{p}}})\tilde{\to}\spc^{-1}(\{\mathfrak{p}\})\end{equation*}for every prime ideal $\mathfrak{p}$ of $A$ minimal over $\varpi$. 

Let $\mathfrak{p}$ be a prime ideal of $A$ minimal over $\varpi$ and consider the map \begin{equation*}\spc^{-1}(\{\mathfrak{p}\})\to \Spa(\widehat{A_{\mathfrak{p}}}[\varpi^{-1}], \widehat{A_{\mathfrak{p}}}), v\mapsto v_{\mathfrak{p}},\end{equation*}constructed in the proof of Proposition~\ref{Local rings and fibres of the specialization map}. We claim that the image of the map\begin{equation*}\Spa(\widehat{A_{\mathfrak{p}}}[\varpi^{-1}], \widehat{A_{\mathfrak{p}}})\to \Spa(A[\varpi^{-1}], A), w\mapsto w\vert_{A[\varpi^{-1}]},\end{equation*}where $w\vert_{A[\varpi^{-1}]}$ denotes the restriction along $A[\varpi^{-1}]\to \widehat{A_{\mathfrak{p}}}[\varpi^{-1}]$ of a continuous valuation $w$ on $\widehat{A_{\mathfrak{p}}}[\varpi^{-1}]$, is $\spc^{-1}(\{\mathfrak{p}\,\})$ and that $w\mapsto w\vert_{A[\varpi^{-1}]}$ is inverse to the map $v\mapsto v_{\mathfrak{p}}$. To prove this, let \begin{equation*}w\in \Spa(A_{\mathfrak{p}}[\varpi^{-1}], A_{\mathfrak{p}})\simeq \Spa(\widehat{A_{\mathfrak{p}}}[\varpi^{-1}], \widehat{A_{\mathfrak{p}}})\end{equation*}and let $s\in A\setminus\mathfrak{p}$. Then $w(\frac{s}{1})\leq1$ since $\frac{s}{1}\in A_{\mathfrak{p}}$ and\begin{equation*}w(\frac{s}{1})^{-1}=w(\frac{1}{s})\leq1\end{equation*}since $\frac{1}{s}\in A_{\mathfrak{p}}$, so $w(\frac{s}{1})=1$. It follows that $w=v_{\mathfrak{p}}$ with $v=w\vert_{A[\varpi^{-1}]}$. Moreover, the center of $w$ on $A_{\mathfrak{p}}$ is contained in $\mathfrak{p}A_{\mathfrak{p}}$ and thus the center $\spc(v)$ of $v$ on $A$ is contained in $\mathfrak{p}$. However, the center of a $\varpi$-adically continuous valuation on $A$ is always a prime ideal containing $\varpi$ and thus $\spc(v)=\mathfrak{p}$, by the minimality of $\mathfrak{p}$ over $\varpi$. This concludes the proof.\end{proof}
We can now similarly characterize strongly $\varpi$-Shilov rings. In what follows, for any Tate Huber pair $(\mathcal{A}, \mathcal{A}^{+})$ with topologically nilpotent unit $\varpi\in\mathcal{A}$, we view the Berkovich spectrum $\mathcal{M}(\mathcal{A})$ of $\mathcal{A}$ relative to $\varpi$ (see Definition~\ref{Berkovich spectrum of a Tate ring}) as the set of rank $1$ points of the adic spectrum $\Spa(\mathcal{A}, \mathcal{A}^{+})$. Recall the definition of Berkovich's description of the Shilov boundary and of Berkovich's description of the Shilov boundary in the strong sense, Definition~\ref{Berkovich's description of the Shilov boundary}.  
\begin{thm}\label{Strongly Shilov}Let $\varpi$ be a non-zero-divisor and a non-unit in a ring $A$ and suppose that $A$ is integrally closed in $\mathcal{A}=A[\varpi^{-1}]$. If the ring $A$ is $\varpi$-Shilov, the set $\spc^{-1}(\wAss(\varpi))$ of pre-images of weakly associated prime ideals of $\varpi$ is a bounding set for $(\mathcal{A}, A)$ in the sense of Definition~\ref{Bounding set}. If the ring $A$ is strongly $\varpi$-Shilov, then $\mathcal{A}$ satisfies Berkovich's description of the Shilov boundary.

Conversely, if $A$ is completely integrally closed in $\mathcal{A}$, the quotient $A/(\varpi)_{A}$ has only finitely many minimal prime ideals and $\mathcal{A}$ satisfies Berkovich's description of the Shilov boundary, then $A$ is strongly $\varpi$-Shilov. \end{thm}
We first prove a general lemma.
\begin{lemma}\label{Boundaries and Shilov boundaries}Let $\varpi\in A$ be a non-zero-divisor and non-unit in a ring $A$ such that $A$ is completely integrally closed in $\mathcal{A}=A[\varpi^{-1}]$. If the generic point of every irreducible component of $\Spf(A)$ has a unique pre-image in $\Spa(\mathcal{A}, A)$ and the closure of $\spc^{-1}(\Spf(A)_{\gen})$ in $\mathcal{M}(\mathcal{A})$ is a boundary for $\mathcal{A}$, then $\mathcal{A}$ satisfies Berkovich's description of the Shilov boundary. In other words, the closure $\mathcal{S}$ of $\spc^{-1}(\Spf(A)_{\gen})$ inside $\mathcal{M}(\mathcal{A})$ is the Shilov boundary for $\mathcal{A}$.\end{lemma}
\begin{proof}By \cite{Guennebaud}~Ch.~1, Proposition~4, it suffices to prove that $\mathcal{S}$ is a minimal closed boundary, i.e., that it does not properly contain any smaller closed boundary for $\mathcal{A}$. By way of contradiction, assume that $\mathcal{S}^{\prime}$ is a closed boundary for $\mathcal{A}$ which is properly contained in $\mathcal{S}$. If $\mathcal{S}^{\prime}$ contains $\spc^{-1}(\Spf(A)_{\gen})$, then the fact that $\mathcal{S}^{\prime}$ is closed implies $\mathcal{S}^{\prime}=\mathcal{S}$, a contradiction. Hence there exists some $v\in\spc^{-1}(\Spf(A)_{\gen})$ which does not belong to $\mathcal{S}^{\prime}$. Let \begin{equation*}U=\mathcal{M}(\mathcal{A})\setminus\mathcal{S}^{\prime},\end{equation*}an open neighbourhood of $v$ in $\mathcal{M}(\mathcal{A})$. We prove that $\mathcal{S}^{\prime}$ is not a boundary by proving the existence of an element $f\in A$ with $\vert f\vert_{\spc,\varpi}=1$ and \begin{equation*}\{\, w\in\mathcal{M}(\mathcal{A})\mid w(f)=1\,\}\subseteq U.\end{equation*}By hypothesis, we know that \begin{equation*}\{v\}=\spc^{-1}(\{\spc(v)\}),\end{equation*}so \begin{align*}U\supseteq \{v\}=\spc^{-1}(\bigcap_{f\in A\setminus\spc(v)}D(f))=\bigcap_{f\in A\setminus\spc(v)}\spc^{-1}(D(f))\\=\bigcap_{f\in A\setminus\spc(v)}\{\, w\in \mathcal{M}(\mathcal{A})\mid w(f)=1\,\}.\end{align*}In other words, the intersection of the family of closed subsets of $\mathcal{M}(\mathcal{A})$ given by \begin{equation*}\{\, w\in \mathcal{M}(\mathcal{A})\mid w(f)=1\,\}\setminus U, f\in A\setminus\spc(v),\end{equation*}is empty. By compactness of the Berkovich spectrum (\cite{Berkovich}~Theorem~1.2.1), there must then exist a finite subfamily (of this family of closed sets) whose intersection is also empty. But for any finite family of elements $f_1,\dots, f_{n}\in A$, we have \begin{equation*}\bigcap_{i=1}^{n}\{\, w\in\mathcal{M}(\mathcal{A})\mid w(f_{i})=1\,\}=\{\, w\in\mathcal{M}(\mathcal{A})\mid w(f_1\cdot~\dots~\cdot f_{n})=1\,\},\end{equation*}and if $f_1,\dots, f_n\in A\setminus\spc(v)$, then $f=f_{1}\cdot~\dots~\cdot f_{n}\in A\setminus\spc(v)$ due to $\spc(v)$ being a prime ideal. It follows that we can find a single element $f\in A\setminus\spc(v)$ such that \begin{equation*}\{\, w\in \mathcal{M}(\mathcal{A})\mid w(f)=1\,\}\setminus U=\emptyset,\end{equation*}i.e., $U\supseteq \{\, w\in\mathcal{M}(\mathcal{A})\mid w(f)=1\,\}$. But, since $\spc(v)$ is a prime ideal containing $\varpi$, the property $f\in A\setminus\spc(v)$ implies $f\in A\setminus\sqrt{(\varpi)_{A}}$, so $\vert f\vert_{\spc,\varpi}=1$. It follows that the element $f\in A$ attains its maximum only at points in $U$ and the closed subset $\mathcal{S}^{\prime}=\mathcal{S}\setminus U$ is not a boundary for $\mathcal{A}$. \end{proof}
\begin{proof}[Proof of Theorem~\ref{Strongly Shilov}]The assertion that $\spc^{-1}(\wAss(\varpi))$ is a bounding set for the Tate Huber pair $(\mathcal{A}, A)$ if $A$ is $\varpi$-Shilov is a direct consequence of Theorem~\ref{Boundaries and integral closure} and Proposition~\ref{Weakly associated prime ideals and boundaries}.

Assume that $A$ is strongly $\varpi$-Shilov. By Proposition~\ref{Weakly Shilov}, we know that every generic point of $\Spf(A)$ has a unique pre-image in $\Spa(\mathcal{A}, A)$. Moreover, Proposition~\ref{Weakly Shilov}, together with Lemma \ref{Valuative rings and adic spectrum}, implies that the pre-image of any $\mathfrak{p}\in\Spf(A)_{\gen}$ is a rank $1$ point. Therefore, $\spc^{-1}(\Spf(A)_{\gen})$ is a subset of $\mathcal{M}(\mathcal{A})$ and the question whether its closure in $\mathcal{M}(\mathcal{A})$ is equal to the Shilov boundary for $\mathcal{A}$ is well-posed. By Corollary~\ref{Strongly Shilov implies completely integrally closed}, $A$ is completely integrally closed in $\mathcal{A}$, i.e., $A=\mathcal{A}_{\vert\cdot\vert_{\spc,\varpi}\leq1}$. By the definition of strongly $\varpi$-Shilov rings, we know that every weakly associated prime of $\varpi$ is minimal over $\varpi$, so \begin{equation*}\mathcal{A}_{\vert\cdot\vert_{\spc,\varpi}\leq1}=A=\bigcap_{\mathfrak{p}\in \Min_{A}(\varpi)}\widehat{A_{\mathfrak{p}}}\end{equation*}by Proposition~\ref{Weakly associated prime ideals and boundaries}. By Theorem~\ref{Boundaries and arc-closure}, this implies that the closure $\mathcal{S}$ of $\spc^{-1}(\Spf(A)_{\gen})$ in $\mathcal{M}(\mathcal{A})$ is a boundary for $\mathcal{A}$. We conclude by Lemma~\ref{Boundaries and Shilov boundaries} that $\mathcal{S}$ is the Shilov boundary for $\mathcal{A}$.  

Conversely, suppose that $A$ is completely integrally closed in $\mathcal{A}$, there exist only finitely many minimal primes of $\varpi$ in $A$ and the Tate ring $\mathcal{A}$ satisfies Berkovich's description of the Shilov boundary. We want to prove that $A$ is strongly $\varpi$-Shilov. By Proposition~\ref{Weakly Shilov}, the assumption that every generic point of $\Spf(A)$ has a unique pre-image (which is part of the definition of Berkovich's description of the Shilov boundary) is equivalent to $A$ being weakly $\varpi$-Shilov. It remains to prove that $A/(\varpi)_{A}$ satisfies condition S1, i.e., that every weakly associated prime of $\varpi$ is minimal over $\varpi$. By Proposition~\ref{Minimal prime ideals and boundaries} and the assumption that $A/(\varpi)_{A}$ has only finitely many minimal prime ideals, it suffices to show that \begin{equation*}A=\bigcap_{\mathfrak{p}\in\Min_{A}(\varpi)}\widehat{A_{\mathfrak{p}}},\end{equation*}where, as usual, $\Min_{A}(\varpi)$ denotes the (in this case, finite) set of prime ideals of $A$ minimal over $\varpi$. Since $A$ is completely integrally closed, \begin{equation*}A=\mathcal{A}_{\vert\cdot\vert_{\spc,\varpi}\leq1}.\end{equation*}Hence the desired assertion follows from the assumption that the closure of $\spc^{-1}(\Spf(A)_{\gen})$ is a boundary and Theorem~\ref{Boundaries and arc-closure}.
\end{proof}
In the same vein, we record the following result on uniform Tate rings $\mathcal{A}$ whose reduction $\mathcal{A}^{\circ}/\mathcal{A}^{\circ\circ}$ is an integral domain.
\begin{prop}\label{Multiplicative norms}Let $\mathcal{A}$ be a uniform Tate ring with topologically nilpotent unit $\varpi$. Suppose that \begin{equation*}\vert\mathcal{A}\vert_{\spc,\varpi}\subseteq \sqrt{\vert\mathcal{A}^{\times,m}\vert_{\spc,\varpi}}\cup\{0\}.\end{equation*}If $\mathcal{A}^{\circ}$ has a unique minimal prime ideal over $\varpi$ (or, equivalently, if $\mathcal{A}^{\circ}/\mathcal{A}^{\circ\circ}$ is an integral domain), then $\vert\cdot\vert_{\spc,\varpi}$ is multiplicative and is the unique pre-image under $\spc$ of the unique minimal prime ideal of $\varpi$ in $\mathcal{A}^{\circ}$. In other words, the Shilov boundary for $\mathcal{A}$ consists of a single point and $\mathcal{A}$ satisfies Berkovich's description of the Shilov boundary in the strong sense.\end{prop}
\begin{proof}Set $A=\mathcal{A}^{\circ}$. We follow the first part of the proof of \cite{Berkovich}~Proposition~2.4.4(ii). Let $\mathfrak{p}$ be the unique minimal prime ideal of $\varpi$ in $A$. Then $\mathfrak{p}=\sqrt{(\varpi)_{A}}=\mathcal{A}^{\circ\circ}$, so $\mathfrak{p}$ is the open unit ball for the spectral seminorm $\vert\cdot\vert_{\spc,\varpi}$. Recall that the specialization map $\spc$ is always surjective (as follows, for example, from \cite{BhattNotes}~Theorem~8.1.2), so the fiber $\spc^{-1}(\mathfrak{p})$ is non-empty. Recall that every point in an analytic adic space has a unique rank-$1$ generization (this is a consequence of \cite{Huber3}~Lemma~1.1.10i) and ii)). Since the center $\spc(v)$ on $A$ of a continuous valuation $v\in \Spa(\mathcal{A}, A)$ can only shrink under generization, the fiber $\spc^{-1}(\mathfrak{p})$ must contain a rank-$1$ point. Let $v\in \spc^{-1}(\mathfrak{p})$ be such a rank-$1$ point. By the definition of $\spc$, we know that $v(f)=\vert f\vert_{\spc,\varpi}=1$ for all $f\in\mathcal{A}^{\circ}$ with $f\in\mathcal{A}^{\circ}\setminus\mathfrak{p}$, i.e., $v(f)=\vert f\vert_{\spc,\varpi}$ for all $f\in\mathcal{A}$ with $\vert f\vert_{\spc,\varpi}=1$. By the assumption on the values of $\vert\cdot\vert_{\spc,\varpi}$, this implies $v(f)=\vert f\vert_{\spc,\varpi}$ for all $f\in\mathcal{A}$; in particular, $\vert\cdot\vert_{\spc,\varpi}=v$ is multiplicative and is the only rank-$1$ point in $\spc^{-1}(\mathfrak{p})$. But by the proof of Proposition~\ref{Weakly Shilov}, the continuous map \begin{equation*}\Spa(\widehat{A_{\mathfrak{p}}}[\varpi^{-1}], \widehat{A_{\mathfrak{p}}})\to \Spa(\mathcal{A}, A), w\mapsto w\vert_{\mathcal{A}},\end{equation*}restricts to a homeomorphism \begin{equation*}\Spa(\widehat{A_{\mathfrak{p}}}[\varpi^{-1}], \widehat{A_{\mathfrak{p}}})\tilde{\to}\spc^{-1}(\mathfrak{p}).\end{equation*}In particular, $\Spa(\widehat{A_{\mathfrak{p}}}[\varpi^{-1}], \widehat{A_{\mathfrak{p}}})$ has a unique rank-$1$ point or, equivalently, the Berkovich spectrum $\mathcal{M}(\widehat{A_{\mathfrak{p}}}[\varpi^{-1}])$ is a single point. By Lemma~\ref{Valuative rings and Berkovich spectrum}, this implies that $\widehat{A_{\mathfrak{p}}}$ is a valuation ring. Since $\mathfrak{p}$ is a minimal prime ideal of $\varpi$ in $A$, we conclude that $\widehat{A_{\mathfrak{p}}}$ is a valuation ring of rank $1$. It follows that $\spc^{-1}(\mathfrak{p})$ is a singleton set $\{v\}$ and is equal to the Shilov boundary for $\mathcal{A}$.  \end{proof}

\section{The $\varpi$-Shilov property and coherence}\label{sec:coherence}
   
To give additional sufficient conditions for a ring to be $\varpi$-Shilov, we leverage the formalism of $\varpi$-$\ast$-operations introduced in Section~\ref{sec:star-operations}. The following definition is an analog in this context of the notion of a Prüfer v-multiplication domain (PvMD).
\begin{mydef}[$\ast$-multiplication ring]Let $\varpi$ be a non-zero-divisor and a non-unit in a ring $A$. Let $\ast: J_{\varpi}(A)\to J_{\varpi}(A)$ be a strict $\varpi$-$\ast$-operation and let $\ast_{f}$ be the finite-type companion of $\ast$. The ring $A$ is called a $\ast$-multiplication ring if for every $f\in A$ and every integer $n\geq1$ there exists a finitely generated $A$-submodule $J$ of $A[\varpi^{-1}]$ with $J\circ_{\ast}(f, \varpi^{n})_{A}=A$ (equivalently, with $J\circ_{\ast_{f}}(f, \varpi^{n})_{A}=A$).\end{mydef}
By analogy with the work of Houston, Malik and Mott \cite{Houston-Malik-Mott} on $\ast$-multiplication domains (for a usual $\ast$-operation on the domain), we characterize the notion of $\ast$-multiplication rings for a $\varpi$-$\ast$-operation $\ast$ in a number of ways. Hereby, the notion of $\varpi$-valuative rings introduced by Fujiwara-Kato in \cite{FK} and studied in Section~\ref{sec:valuative rings} of the present paper takes on the role played in the classical case by valuation rings.

For a ring $A$ and an ideal $P$ of the polynomial ring $A[X]$ we write $c(P)\subseteq A$ for the content of $P$, the ideal of coefficients of elements of $P$. If $P$ is generated by a single polynomial $F\in A[X]$, then $c(P)$ is the ideal of $A$ generated by the coefficients of $F$ and we write $c(F)$ instead of $c(P)$.   
\begin{thm}[Analog of \cite{Houston-Malik-Mott}~Theorem~1.1]\label{Houston-Malik-Mott}Let $\varpi\in A$ be a non-zero-divisor and non-unit in a ring $A$. Let $\ast$ be a strict $\varpi$-$\ast$-operation on $A$ and let $\ast_{f}$ be its companion of finite type. The following are equivalent: 
\begin{enumerate}[(1)]\item $A$ is a $\ast$-multiplication ring. \item $A$ is integrally closed in $A[\varpi^{-1}]$ and, for every $f\in A$, $n\geq1$, the ideal \begin{equation*}(\varpi^{n}X-f)A[\varpi^{-1}][X]\cap A[X]\end{equation*}of the polynomial ring $A[X]$ contains an element $F$ with $c(F)^{\ast}=A$. \item $A_{\mathfrak{p}}$ is a $\varpi$-valuative ring for every maximal $\ast_{f}$-ideal $\mathfrak{p}$. \item $A_{\mathfrak{p}}$ is a $\varpi$-valuative ring for every prime $\ast_{f}$-ideal $\mathfrak{p}$. \end{enumerate}\end{thm}
\begin{proof}(1)$\Rightarrow$(4): Let $\mathfrak{p}$ be a prime $\ast_{f}$-ideal of $A$. By virtue of Proposition~\ref{Properties of invertible fractional ideals}(3) and Proposition~\ref{Valuative rings and valuation rings}, to prove that $A_{\mathfrak{p}}$ is a $\varpi$-valuative ring, it suffices to let $f\in A_{\mathfrak{p}}$, $f\neq0$, $n\geq1$, and prove that the ideal $(f, \varpi^{n})_{A_{\mathfrak{p}}}$ of $A_{\mathfrak{p}}$ is invertible. For this it suffices to assume that $f\in A$. Since $A$ is a $\ast$-multiplication ring, there exists a finitely generated $A$-submodule $J$ of $A[\varpi^{-1}]$ with $((f, \varpi^{n})_{A}J)^{\ast}=A$. Since $J$ is finitely generated, this is equivalent to $((f, \varpi^{n})_{A}J)^{\ast_{f}}=A$. Since $\mathfrak{p}$ is a $\ast_{f}$-ideal, this implies \begin{equation*}(f, \varpi^{n})_{A}J\not\subseteq\mathfrak{p}\end{equation*}and thus $(f, \varpi^{n})_{A_{\mathfrak{p}}}J_{\mathfrak{p}}=A_{\mathfrak{p}}$. 

(4)$\Rightarrow$(3): Trivial in view of Lemma~\ref{Maximal star-ideals are prime}.

(3)$\Rightarrow$(2): By Proposition~\ref{Pruefer subrings and valuative rings}, (3) implies that $A_{\mathfrak{p}}$ is a Prüfer subring of $A_{\mathfrak{p}}[\varpi^{-1}]$ for every $\mathfrak{p}\in\Max_{\ast_{f}}(A)$. By \cite{Knebusch-Zhang}~Ch.~1, Theorem~5.2(1)$\Leftrightarrow$(4), this means in particular that $A_{\mathfrak{p}}$ is integrally closed in $A_{\mathfrak{p}}[\varpi^{-1}]$ for all $\mathfrak{p}\in\Max_{\ast_{f}}(A)$, so, by Proposition~\ref{Griffin's intersection property}, $A$ is integrally closed in $A[\varpi^{-1}]$. Given $f\in A$, $f\neq0$, $n\geq1$, set \begin{equation*}P=(\varpi^{n}X-f)A[\varpi^{-1}][X]\cap A[X].\end{equation*}By \cite{Knebusch-Zhang}~Theorem~5.2(1)$\Leftrightarrow$(6), for every $\mathfrak{p}\in\Max_{\ast_{f}}(A)$ there exists a polynomial $F_{\mathfrak{p}}\in A_{\mathfrak{p}}[X]$ with at least one unit coefficient such that $F_{\mathfrak{p}}(\frac{f}{\varpi^{n}})=0$. If $\mathfrak{p}\in\Max_{\ast_{f}}(A)$ does not contain $\varpi$, then $\varpi$ is a unit in $A_{\mathfrak{p}}$ and $A_{\mathfrak{p}}$ is trivially $\varpi$-valuative. For every $\mathfrak{p}\in\Max_{\ast_{f}}(A)$ with $\varpi\in\mathfrak{p}$, set \begin{equation*}B_{\mathfrak{p}}=A_{\mathfrak{p}}/\bigcap_{k\geq1}\varpi^{k}A_{\mathfrak{p}}.\end{equation*}The image of $F_{\mathfrak{p}}$ in $B_{\mathfrak{p}}[X]$ is a polynomial with coefficients in $B_{\mathfrak{p}}$ which has a zero at (the image of) $\frac{f}{\varpi^{n}}$ and has at least one unit coefficient; in particular, this polynomial does not belong to $(\mathfrak{p}B_{\mathfrak{p}})[X]$. 

On the other hand, by Proposition~\ref{Valuative rings and valuation rings}, for every $\mathfrak{p}\in \Max_{\ast_{f}}(A)$, the $\varpi$-adically separated local ring $B_{\mathfrak{p}}$ is a valuation ring and then, by \cite{FK}~Ch.~0, Proposition~6.7.2, $B_{\mathfrak{p}}[\varpi^{-1}]$ is a field. Thus, for every $\mathfrak{p}\in\Max_{\ast_{f}}(A)$, the ideal of $B_{\mathfrak{p}}[\varpi^{-1}][X]$ consisting of polynomials with a zero at $\frac{f}{\varpi^{n}}$ is equal to $(\varpi^{n}X-f)B_{\mathfrak{p}}[\varpi^{-1}][X]$. It follows that \begin{equation*}(\varpi^{n}X-f)B_{\mathfrak{p}}[\varpi^{-1}][X]\cap B_{\mathfrak{p}}[X]\not\subseteq (\mathfrak{p}B_{\mathfrak{p}})[X],\end{equation*}whence\begin{equation*}(\varpi^{n}X-f)A_{\mathfrak{p}}[\varpi^{-1}][X]\cap A_{\mathfrak{p}}[X]\not\subseteq (\mathfrak{p}A_{\mathfrak{p}})[X]\end{equation*}for every $\mathfrak{p}\in\Max_{\ast_{f}}(A)$. Therefore, for each $\mathfrak{p}\in\Max_{\ast_{f}}(A)$ we have $c(P)\not\subseteq\mathfrak{p}$. Consequently, $c(P)^{\ast_{f}}=A$. Since $\ast_{f}$ is the finite-type companion of $\ast$, there exists a finitely generated subideal $J$ of $c(P)$ with $1\in J^{\ast}$, so that $J^{\ast}=A$. Let $G_1,\dots, G_n\in P$ be polynomials whose coefficients generate $J$. One can then construct a polynomial $F\in (G_1,\dots, G_n)_{A[X]}$ with $c(F)=\sum_{i=1}^{n}c(G_{i})$, so that $c(F)^{\ast}=A$.

(2)$\Rightarrow$(1): Let $f\in A$, $n\geq1$, and choose $F\in (\varpi^{n}X-f)A[\varpi^{-1}][X]\cap A[X]$ with $c(F)^{\ast}=A$. Write \begin{equation*}F=(\varpi^{n}X-f)H\end{equation*}for some $H\in A[\varpi^{-1}][X]$. Then $A=c(F)^{\ast}\subseteq (c(\varpi^{n}X-f)c(H))^{\ast}=((f, \varpi^{n})_{A}c(H))^{\ast}$. Hence it remains to show that $(f, \varpi^{n})_{A}c(H)\subseteq A$. By the Dedekind-Mertens Lemma, \begin{equation*}(f, \varpi^{n})_{A}c(H)(f, \varpi^{n})_{A}^{m}=c((\varpi^{n}X-f)H)(f, \varpi^{n})_{A}^{m}\subseteq (f, \varpi^{n})_{A}^{m},\end{equation*}where $m$ is the degree of the polynomial $H\in A[\varpi^{-1}][X]$. By the assumption that $A$ is integrally closed in $A[\varpi^{-1}]$ and \cite{Swanson-Huneke}~Lemma~2.1.8, this implies that $(f, \varpi^{n})_{A}c(H)\subseteq A$, as desired.\end{proof}
The importance of the notion of $\ast$-multiplication rings for our purposes arises from the following theorem.
\begin{thm}\label{Multiplication rings are strongly Shilov}Let $\varpi$ be a non-zero-divisor and non-unit in a ring $A$ and let $\ast$ be a strict $\varpi$-$\ast$-operation on $J_{\varpi}(A)$. If $A$ is a $\ast$-multiplication ring, then $A$ is a $\varpi$-Shilov ring. \end{thm}
\begin{proof}Let $\mathfrak{p}$ be a weakly associated prime ideal of $\varpi$ in $A$. Choose $f\in A$ with $\mathfrak{p}$ minimal over $(\varpi)_{A}:_{A}f=A:_{A}\frac{f}{\varpi}$. By Lemma~\ref{Properties of divisorial ideals}, $(\varpi)_{A}:_{A}f$ is a $\varpi$-v-ideal ($\varpi$-divisorial ideal). Since by Proposition~\ref{v- and t-operations are star-operations} the $\varpi$-v-operation is the coarsest strict $\varpi$-$\ast$-operation, $(\varpi)_{A}:_{A}f$ is also a $\ast$-ideal. By Proposition~\ref{Prime ideals minimal over a star-ideal}, it follows that $\mathfrak{p}$ is a prime $\ast_{f}$-ideal. We conclude by Theorem~\ref{Houston-Malik-Mott}(1)$\Rightarrow$(4) (in conjunction with Corollary~\ref{Valuative rings and completion}) that $\widehat{A_{\mathfrak{p}}}$ is a valuation ring. \end{proof}    
To find more practical sufficient criteria for being a $\varpi$-Shilov ring, we use the following analogs of the notions of a finite conductor ring (FC ring) and of a v-coherent ring.  
\begin{mydef}[$\varpi$-FC ring]Let $\varpi$ be a non-zero-divisor and non-unit in a ring $A$. We say that $A$ is a $\varpi$-(finite conductor) ring, or $\varpi$-FC ring, if for every $f\in A$ and every integer $n\geq1$ the $\varpi$-residual $A:_{A[\varpi^{-1}]}(f, \varpi^{n})_{A}$ is finitely generated as an $A$-submodule of $A[\varpi^{-1}]$.\end{mydef}
\begin{mydef}[$\ast$-coherent ring]Let $\varpi$ be a non-zero-divisor and a non-unit in a ring $A$ and let $\ast: J_{\varpi}(A)\to J_{\varpi}(A)$ be a strict $\varpi$-$\ast$-operation. We say that the ring $A$ is $\ast$-coherent if for every $f\in A$ and every integer $n\geq1$ the $\varpi$-residual $A:_{A[\varpi^{-1}]}(f, \varpi^{n})_{A}$ is $\ast$-finite, i.e., there exists a finitely generated $A$-submodule $J$ of $A:_{A[\varpi^{-1}]}(f, \varpi^{n})_{A}$ such that $J^{\ast}=A:_{A[\varpi^{-1}]}(f, \varpi^{n})_{A}$. 

In particular, $A$ is called $\varpi$-v-coherent if for every $f\in A$ and every integer $n\geq1$ the $A$-submodule $A:_{A[\varpi^{-1}]}(f, \varpi^{n})_{A}$ of $A[\varpi^{-1}]$ is equal to $J_{v}=A:_{A[\varpi^{-1}]}(A:_{A[\varpi^{-1}]}J)$ for some finitely generated $A$-submodule $J$ of $A:_{A[\varpi^{-1}]}(f, \varpi^{n})_{A}$.\end{mydef}
It is clear that a $\varpi$-FC ring is $\ast$-coherent for every $\varpi$-$\ast$-operation $\ast$. On the other hand, the name '$\varpi$-(finite conductor) ring' is justified by the following lemma.
\begin{lemma}\label{Finite conductor property}Let $A$ be a ring such that $\Ann(f)$ is a finitely generated ideal of $A$ for every $f\in A$ and let $\varpi\in A$ be a non-zero-divisor and a non-unit in $A$. Then the following are equivalent: \begin{enumerate}[(1)]\item $A$ is a $\varpi$-FC ring. \item For every $f\in A$ and $n\geq1$ the conductor $A:_{A}\frac{f}{\varpi^{n}}=(\varpi^{n})_{A}:_{A}f$ of $\frac{f}{\varpi^{n}}\in A[\varpi^{-1}]$ in $A$ is a finitely generated ideal of $A$. \item For every $f\in A$ and $n\geq1$ the intersection $(f)_{A}\cap (\varpi^{n})_{A}$ is finitely generated.\end{enumerate}\end{lemma}
\begin{proof}Since $\Ann(f)$ is finitely generated, it follows from the two short exact sequences \begin{center}\begin{tikzcd}0\arrow{r} & \varpi^{-n}\Ann(f)=\Ann(\varpi^{n}f)\arrow{r} & A:_{A[\varpi^{-1}]}(f, \varpi^{n})_{A}\arrow{r} & (f)_{A}\cap (\varpi^{n})_{A}\arrow{r} & 0\end{tikzcd}\end{center}and \begin{center}\begin{tikzcd}0\arrow{r} & \Ann(f)\arrow{r} & A:_{A}\frac{f}{\varpi^{n}}\arrow{r} & (f)_{A}\cap (\varpi^{n})_{A}\arrow{r} & 0,\end{tikzcd}\end{center}where the map $A:_{A[\varpi^{-1}]}(f, \varpi^{n})_{A}\to (f)_{A}\cap (\varpi^{n})_{A}$ is multiplication by $\varpi^{n}f$ and the map $A:_{A}\frac{f}{\varpi^{n}}\to (f)_{A}\cap (\varpi^{n})_{A}$ is multiplication by $f$ (to show that the two maps are surjective, choose $g\in (f)_{A}\cap (\varpi^{n})_{A}$ and write it as $g=hf=h'\varpi^{n}$ for some $h, h'\in A$; then $g=\frac{h}{\varpi^{n}}\varpi^{n}f$ and \begin{equation*}\frac{h}{\varpi^{n}}\in A:_{A[\varpi^{-1}]}(f, \varpi^{n})_{A}\end{equation*}as well as $h\in A:_{A}\frac{f}{\varpi^{n}}$).\end{proof}
\begin{cor}\label{Coherent}Every finite conductor ring $A$ (in the sense of \cite{Glaz00}) and, in particular, every coherent ring $A$, is a $\varpi$-FC ring for every non-zero-divisor and non-unit $\varpi\in A$.\end{cor}
\begin{lemma}\label{Finite conductor rings and localization}Let $A$ be a $\varpi$-FC ring, for some non-zero-divisor and non-unit $\varpi\in A$. If $S\subseteq A$ is a multiplicative subset containing $1$ and not containing any power of $\varpi$, then the localization $S^{-1}A$ is again a $\varpi$-FC ring.\end{lemma}
\begin{proof}By abuse of notation, we identify elements $f\in A$ with their images $\frac{f}{1}$ in $S^{-1}A$. For $f\in S^{-1}A$ and $n\geq1$, we have to prove that the $S^{-1}A$-submodule \begin{equation*}S^{-1}A:_{S^{-1}A[\varpi^{-1}]}(f, \varpi^{n})_{S^{-1}A}\end{equation*}is finitely generated. It suffices to assume that $f\in A$. In this case, let $f_1,\dots, f_n$ be generators of the $A$-submodule $A:_{A[\varpi^{-1}]}(f, \varpi^{n})_{A}$ of $A[\varpi^{-1}]$. We prove that (the images of) $f_1,\dots, f_n$ also generate $S^{-1}A:_{S^{-1}A[\varpi^{-1}]}(f, \varpi^{n})_{S^{-1}A}$ as an $S^{-1}A$-module. Let $g\in A[\varpi^{-1}]$, $t\in S$ be such that $\frac{g}{t}f, \frac{g}{t}\varpi^{n}\in S^{-1}A$. Then there exists some $t'\in S$ with $t'gf, t'g\varpi^{n}\in A$. Thus $t'g\in (f_1,\dots, f_n)_{A}$ and, consequently, $\frac{g}{t}\in (f_1,\dots, f_n)_{S^{-1}A}$. \end{proof}
We obtain the following sufficient criterion for being a $\varpi$-Shilov ring. 
\begin{thm}\label{Finite conductor implies strongly Shilov}Let $\varpi$ be a non-zero-divisor and a non-unit in a ring $A$ such that $A$ is integrally closed in the Tate ring $\mathcal{A}=A[\varpi^{-1}]$. If $A$ is a $\varpi$-FC ring, then $A$ is $\varpi$-Shilov.\end{thm}
\begin{cor}\label{Coherent implies strongly Shilov}If $A$ is a coherent ring which is integrally closed in $\mathcal{A}=A[\varpi^{-1}]$ for some non-unit non-zero-divisor $\varpi\in A$, then $A$ is $\varpi$-Shilov.\end{cor}
Our proof of Theorem~\ref{Finite conductor implies strongly Shilov} is inspired by Zafrullah's proof of Theorem~2 in \cite{Zafrullah78} and, in particular, by the proof of Lemma 5 in that paper. We split most of the proof into a sequence of lemmas.
\begin{lemma}\label{Finite residuals and integral closure}Let $A$ be a ring with a non-zero-divisor and non-unit $\varpi\in A$ and let $\mathcal{A}^{+}$ be the integral closure of $A$ in $\mathcal{A}=A[\varpi^{-1}]$. If $I\subseteq A$ is a finitely generated $A$-submodule of $\mathcal{A}=A[\varpi^{-1}]$ with $A:_{A[\varpi^{-1}]}I$ finitely generated, then $A:_{A[\varpi^{-1}]}I(A:_{A[\varpi^{-1}]}I)\subseteq\mathcal{A}^{+}$.\end{lemma}
\begin{proof}Note that $f\in A:_{A[\varpi^{-1}]}I(A:_{A[\varpi^{-1}]}I)$ if and only if \begin{equation*}f(A:_{A[\varpi^{-1}]}I)\subseteq A:_{A[\varpi^{-1}]}I.\end{equation*}Since $I$ is finitely generated, it is bounded, so $A:_{A[\varpi^{-1}]}I$ contains some power of $\varpi$. In particular, the finitely generated $A$-module $A:_{A[\varpi^{-1}]}I$ is also a faithful $A[f]$-module, $\varpi$ being a non-zero-divisor on $A[f]$. Hence the assertion follows from \cite{Swanson-Huneke}~Lemma~2.1.8. \end{proof}
\begin{lemma}\label{Residual, t-ideals and invertible ideals}Let $A$ be a local ring with a non-unit non-zero-divisor $\varpi$. Suppose that $A$ is integrally closed in $A[\varpi^{-1}]$ and that the maximal ideal $\mathfrak{m}$ of $A$ is a $\varpi$-t-ideal. Then a finitely generated ideal $I$ of $A$ is invertible if and only if $A:_{A[\varpi^{-1}]}I$ is finitely generated.\end{lemma}
\begin{proof}If $I$ is invertible, then, by Proposition~\ref{Properties of invertible fractional ideals}(1), $A:_{A[\varpi^{-1}]}I$ is its inverse. In particular, $A:_{A[\varpi^{-1}]}I$ is an invertible $\varpi$-fractional ideal and then $A:_{A[\varpi^{-1}]}I$ is finitely generated by Proposition~\ref{Properties of invertible fractional ideals}(2). Conversely, suppose that $A:_{A[\varpi^{-1}]}I$ is finitely generated. By Lemma~\ref{Finite residuals and integral closure}, this implies that $(I(A:_{A[\varpi^{-1}]}I))_{v}=A$. But this already means that $I(A:_{A[\varpi^{-1}]}I)=A$, since otherwise $I(A:_{A[\varpi^{-1}]}I)$ is a proper finitely generated ideal of $A$ and then $(I(A:_{A[\varpi^{-1}]}I))_{v}\subseteq \mathfrak{m}$ by the assumption that $\mathfrak{m}$ is a $\varpi$-t-ideal.\end{proof}
\begin{proof}[Proof of Theorem~\ref{Finite conductor implies strongly Shilov}]It suffices to prove that for every weakly associated prime ideal $\mathfrak{p}$ of $\varpi$ in $A$ the local ring $A_{\mathfrak{p}}$ is $\varpi$-valuative. Let $\mathfrak{p}$ be a weakly associated prime of $\varpi$. By Proposition~\ref{Properties of invertible fractional ideals}(3) and Proposition~\ref{Valuative rings and valuation rings}, to prove that $A_{\mathfrak{p}}$ is $\varpi$-valuative, it suffices to prove that the ideal $(f, \varpi^{n})_{A_{\mathfrak{p}}}$ of $A_{\mathfrak{p}}$ is invertible for every $f\in A_{\mathfrak{p}}$ and every integer $n\geq1$. Fix some $f\in A_{\mathfrak{p}}$ and some integer $n\geq1$. Since $A$ is integrally closed in $A[\varpi^{-1}]$, the local ring $A_{\mathfrak{p}}$ is integrally closed in $A_{\mathfrak{p}}[\varpi^{-1}]$. Since $A$ is a $\varpi$-FC ring, $A_{\mathfrak{p}}$ is a $\varpi$-FC ring, by Lemma~\ref{Finite conductor rings and localization}. By Lemma~\ref{Weakly associated primes and maximal t-ideals}, the maximal ideal $\mathfrak{p}A_{\mathfrak{p}}$ of $A_{\mathfrak{p}}$ is a $\varpi$-t-ideal. We can then apply Lemma~\ref{Residual, t-ideals and invertible ideals} to the finitely generated ideal $(f, \varpi^{n})_{A_{\mathfrak{p}}}$ to conclude.\end{proof}
We supplement Theorem~\ref{Finite conductor implies strongly Shilov} and Theorem~\ref{Multiplication rings are strongly Shilov} with the following criterion for being a $\varpi$-v-multiplication ring.
\begin{thm}\label{Multiplication rings and v-coherent rings}Let $\varpi$ be a non-zero-divisor and a non-unit in a ring $A$. Suppose that $A$ is completely integrally closed in $A[\varpi^{-1}]$. Then $A$ is a $\varpi$-v-multiplication ring if and only if $A$ is $\varpi$-v-coherent.\end{thm}
\begin{proof}Suppose that $A$ is a $\varpi$-v-multiplication ring. Let $f\in A\setminus\{0\}$ and $n\geq1$ and let $J$ be a finitely generated $A$-submodule of $A[\varpi^{-1}]$ with $((f, \varpi^{n})_{A}J)_{v}=A$. By \cite{Knebusch-Zhang2}~Ch.~3, Proposition~4.1a), this means that $((f, \varpi^{n})_{A}J_{v})_{v}=A$. By loc.~cit., Proposition~4.5a), this means that \begin{equation*}J_{v}=A:_{A[\varpi^{-1}]}(f, \varpi^{n})_{A},\end{equation*}showing that $A$ is $\varpi$-v-coherent.

Conversely, suppose that $A$ is $\varpi$-v-coherent. Let $f\in A\setminus\{0\}$ and let $n\geq1$ be an integer. By Theorem~\ref{Rings of power-bounded elements and divisors}, the hypothesis that $A$ be completely integrally closed in $A[\varpi^{-1}]$ implies that \begin{equation*}((f, \varpi^{n})_{A}(A:_{A[\varpi^{-1}]}(f, \varpi^{n})_{A}))_{v}=A.\end{equation*}By the $\varpi$-v-coherence of $A$, there is a finitely generated submodule $J$ of $A:_{A[\varpi^{-1}]}(f, \varpi^{n})_{A}$ with $J_{v}=A:_{A[\varpi^{-1}]}(f, \varpi^{n})_{A}$. We conclude by means of \cite{Knebusch-Zhang2}~Ch.~3, Prop.~4.1a), that $((f, \varpi^{n})_{A}J)_{v}=A$.\end{proof}
We summarize some of our findings in the following theorem.
\begin{thm}\label{Big theorem}Let $\mathcal{A}$ be a reduced uniform Tate ring with topologically nilpotent unit $\varpi$. Suppose that one of the following is true. \begin{enumerate}[(1)]\item The subring of power-bounded elements $\mathcal{A}^{\circ}$ is integral over a Noetherian subring with the same total ring of fractions. \item The subring of power-bounded elements $\mathcal{A}^{\circ}$ is $\varpi$-v-coherent, $\mathcal{A}^{\circ}/(\varpi)_{\mathcal{A}^{\circ}}$ has only finitely many minimal prime ideals and \begin{equation*}\vert\mathcal{A}\vert_{\spc,\varpi}\subseteq \sqrt{\vert \mathcal{A}^{\times,m}\vert_{\spc,\varpi}}\cup\{0\}.\end{equation*}\end{enumerate}Then $\mathcal{A}^{\circ}$ is strongly $\varpi$-Shilov and $\mathcal{A}$ satisfies Berkovich's description of the Shilov boundary in the strong sense. \end{thm}
\begin{proof}(1) By Theorem~\ref{Reduced Noetherian rings are strongly Shilov}, $\mathcal{A}^{\circ}$ is strongly $\varpi$-Shilov (and even strongly $\varpi$-Krull), so $\mathcal{A}$ satisfies Berkovich's description of the Shilov boundary, by Theorem~\ref{Strongly Shilov}. Let $A$ be a Noetherian subring of $\mathcal{A}^{\circ}$ such that $\mathcal{A}^{\circ}$ is integral over $A$ with the same total ring of fractions and let $B$ be the integral closure of $A$ in its total ring of fractions. Note that $B$ necessarily contains $\mathcal{A}^{\circ}$. By the Nagata-Mori Theorem, $B$ is a finite product of Krull domains, so it has only finitely many prime ideals minimal over $\varpi$. By lying over, every prime ideal of $\mathcal{A}^{\circ}$ minimal over $\varpi$ is the restriction of a minimal prime ideal of $\varpi$ in $B$. Hence there are only finitely many minimal prime ideals of $\varpi$ in $\mathcal{A}^{\circ}$. It follows that the pre-image of $\Min_{\mathcal{A}^{\circ}}(\varpi)$ in $\mathcal{M}(\mathcal{A})$ is finite, hence closed in $\mathcal{M}(\mathcal{A})$. Therefore, $\mathcal{A}$ must satisfy Berkovich's description of the Shilov boundary in the strong sense.

(2) It follows from Theorem~\ref{Multiplication rings are strongly Shilov} and Theorem~\ref{Multiplication rings and v-coherent rings} that $\mathcal{A}^{\circ}$ is $\varpi$-Shilov. It then follows from Theorem~\ref{Spectral seminorm and S1} (together with the assumption on the value set of the spectral seminorm) that $\mathcal{A}^{\circ}$ is also strongly $\varpi$-Shilov. By Theorem~\ref{Strongly Shilov}, this implies that $\mathcal{A}$ satisfies Berkovich's description of the Shilov boundary. Moreover, the pre-image $\spc^{-1}(\Min_{\mathcal{A}^{\circ}}(\varpi))$ of $\Min_{\mathcal{A}^{\circ}}(\varpi)$ is again finite and, in particular, closed. This shows that $\mathcal{A}$ satisfies Berkovich's description of the Shilov boundary in the strong sense.\end{proof}
\begin{rmk}\label{Comparison with Berkovich's result}By the Reduced Fiber Theorem for affinoid algebras (see, for example, \cite{Luetkebohmert}~Theorem~3.4.2), the subring of power-bounded elements $\mathcal{A}^{\circ}$ of any reduced affinoid algebra $\mathcal{A}$ (in the sense of Tate) over an algebraically closed nonarchimedean field $K$ is topologically of finite type over the valuation ring $K^{\circ}$, while, by a result of Fujiwara-Gabber-Kato (\cite{FGK}~Corollary~7.3.6), every $K^{\circ}$-algebra topologically of finite type is a coherent ring. Moreover, every affinoid algebra in the sense of Tate over any nonarchimedean field satisfies the assumption on the value set of the spectral seminorm in Theorem~\ref{Big theorem}(2) by virtue of \cite{BGR}~Proposition~6.2.1/4(ii), and it is also clear that for any affinoid algebra $\mathcal{A}$ in the sense of Tate $\mathcal{A}^{\circ}/(\varpi)_{\mathcal{A}^{\circ}}$ has only finitely many minimal prime ideals, the nilreduction of this ring being a finite-type algebra over the residue field of $K$. Thus, Theorem~\ref{Big theorem}(2) immediately recovers Berkovich's original result (Proposition~\ref{Berkovich's result}) for Tate-affinoid algebras over an algebraically closed nonarchimedean field $K$. \end{rmk}

\section{Stability under integral extensions}\label{sec:integral extensions}

In this section we consider stability of Berkovich's description of the Shilov boundary under completed integral extensions. By Theorem~\ref{Strongly Shilov}, this often amounts to checking stability under integral extensions for the strong $\varpi$-Shilov property. The first step is studying the analogous question for the weak $\varpi$-Shilov property.  
\begin{prop}\label{Integral extensions and valuation rings of rank one, first result}Let $A\hookrightarrow B$ be an integral ring extension and let $\varpi\in A$ be a non-zero-divisor and a non-unit in $A$ and in $B$ such that $B$ is integrally closed in $B[\varpi^{-1}]$.\begin{enumerate}[(1)]\item If $A$ is $\varpi$-valuative, then so is $B$. \item If $A$ is a local ring and the $\varpi$-adic completion of $A$ is a valuation ring of rank $1$, then for every prime ideal $\mathfrak{q}$ of $B$ containing $\varpi$, the $\varpi$-adic completion of the local ring $B_{\mathfrak{q}}$ is a valuation ring of rank $1$. \item If $A$ is a valuation ring of rank $1$ and $B$ is reduced, then, for every prime ideal $\mathfrak{q}$ of $B$ containing $\varpi$, the local ring $B_{\mathfrak{q}}$ is a valuation ring of rank $1$. \end{enumerate}\end{prop}
We need the following elementary lemma. 
\begin{lemma}\label{Finite extensions}Let $A$ be an integral domain with fraction field $K$ and let $\Lambda$ be a $K$-algebra which is a finite product $\Lambda=\prod_{i=1}^{r}K_{i}$ of finite field extensions $K_{i}$ of $K$. Then the integral closure of $A$ in $\Lambda$ is equal to the product of the integral closures of $A$ inside $K_{i}$. \end{lemma}
\begin{proof}Let $B$ be the integral closure of $A$ inside $\Lambda$. Choose idempotent elements $e_1,\dots, e_r\in \Lambda$ such that $\Lambda/(e_{i})_{\Lambda}=K_{i}$. Explicitly, let \begin{equation*}e_{1}=(0, 1,\dots, 1), \dots, e_{r}=(1, \dots, 1, 0).\end{equation*}Since $B$ is integrally closed in $\Lambda$, we have $e_{i}\in B$ for all $i$. Note that \begin{equation*}(e_{i})_{\Lambda}\cap B\subseteq \ker(\times (1-e_{i}): B\to B),\end{equation*}so \begin{equation*}(e_{i})_{\Lambda}\cap B=(e_{i})_{B}\end{equation*}for all $i$. It is readily seen that, for every $i\neq j$, the element $e_{i}-1$ belongs to $(e_{j})_{\Lambda}$, and thus also to $(e_{j})_{B}$. In particular, the ideals $(e_{i})_{B}$, $i=1,\dots, r$, of $B$ are pairwise coprime. By the Chinese Remainder Theorem, $B$ decomposes as a product of domains \begin{equation*}B=\prod_{i=1}^{r}B/(e_{i})_{B}.\end{equation*}Since $(e_{i})_{\Lambda}\cap K=0$ for all $i$, we have $(e_{i})_{B}\cap A=0$ for all $i$ and thus each $B/(e_{i})_{B}$ is an integral extension of $A$.  

It remains to prove that each of the domains $B/(e_{i})_{B}$ is integrally closed in its field of fractions $K_{i}$. Suppose that $f\in \Lambda$ is an element whose image $\overline{f}$ in $K_{i}=\Lambda/(e_{i})_{\Lambda}$ is integral over $B/(e_{i})_{B}$. Then there exists an equation of integral dependence \begin{equation*}f^{n}+a_{1}f^{n-1}+\dots+a_{n-1}f+a_{n}=0\end{equation*}with $a_{1},\dots, a_{n-1}\in B$ and $a_{n}\in (e_{i})_{\Lambda}$. Multiplying this equation by $(1-e_{i})^{n}=1-e_{i}$, we see that $(1-e_{i})f=f-e_{i}f\in \Lambda$ is integral over $B$ and hence belongs to $B$. Reducing modulo $(e_{i})_{\Lambda}$, we find that $\overline{f}\in \Lambda/(e_{i})_{\Lambda}$ belongs to $B/(e_{i})_{B}$.  \end{proof}
\begin{proof}[Proof of Proposition~\ref{Integral extensions and valuation rings of rank one, first result}](1) Note that $B$ is the integral closure of $A$ inside $B[\varpi^{-1}]$. If $A$ is $\varpi$-valuative, then so is $B$, by Proposition~\ref{Pruefer subrings and valuative rings} and \cite{Knebusch-Zhang}~Ch.~I, Theorem~5.9. 

(2) By Corollary~\ref{Valuative rings and completion 2}, $A$ is a $\varpi$-valuative local ring and its maximal ideal is minimal over $\varpi$. By (1), $B$ is $\varpi$-valuative. By \cite{FK}~Ch.~0, Proposition~8.7.2, this implies that the local ring $B_{\mathfrak{q}}$ is $\varpi$-valuative for every prime ideal $\mathfrak{q}$ of $B$. Every maximal ideal of $B$ restricts to the unique maximal ideal of $A$, so all maximal ideals of $B$ are minimal prime ideals over $\varpi$. It follows that every prime ideal $\mathfrak{q}$ of $B$ containing $\varpi$ is minimal over $\varpi$. That is, the maximal ideal $\mathfrak{q}B_{\mathfrak{q}}$ of $B_{\mathfrak{q}}$ is minimal over $\varpi$. We conclude by Corollary~\ref{Valuative rings and completion 2} that the $\varpi$-adic completion of $B_{\mathfrak{q}}$ is a valuation ring of rank $1$ for every prime ideal $\mathfrak{q}\subsetneq B$ with $\varpi\in\mathfrak{q}$.
 
(3) Now assume that $A$ is a valuation ring of rank $1$ and $B$ is reduced. Then $\mathcal{A}=A[\varpi^{-1}]$ is a field, so the integral extension $\mathcal{B}=B[\varpi^{-1}]$ is zero-dimensional. Write $B$ as $B=\varinjlim_{i}B_{i}$ for a directed system of finite subextensions $B_{i}$ of $A$ and let $\mathcal{B}_{i}^{+}$ be the integral closure of $B_{i}$ in $\mathcal{B}_{i}=B_{i}[\varpi^{-1}]$. Since $B$ is integrally closed in $\mathcal{B}$, we can write \begin{equation*}B=\varinjlim_{i}\mathcal{B}_{i}^{+}.\end{equation*}Let $\mathfrak{q}$ be a prime ideal of $B$ with $\varpi\in\mathfrak{q}$. To prove that $B_{\mathfrak{q}}$ is a valuation ring, it suffices to prove that $(\mathcal{B}_{i}^{+})_{\mathfrak{q}\cap \mathcal{B}_{i}^{+}}$ is a valuation ring for all $i$, so we may assume that $\mathcal{B}$ is finite over $\mathcal{A}$. Since $\mathcal{A}$ is a field and $\mathcal{B}$ is reduced, $\mathcal{B}$ must then be a finite product of finite field extensions of $\mathcal{A}$. By Lemma~\ref{Finite extensions}, the integral closure $B$ of $A$ in $\mathcal{B}$ is then a finite product of integral closures of $A$ in finite field extensions of its field of fractions. But it is classically known that the integral closure of a Prüfer domain in an algebraic extension of its field of fractions is a Prüfer domain. Hence $B$ is a finite product of one-dimensional Prüfer domains, and it follows that the localization of $B$ at any prime ideal containing $\varpi$ is a rank-$1$ valuation ring, as claimed.   \end{proof}
\begin{prop}\label{Integral extensions, first result}Let $A\hookrightarrow B$ be a torsion-free integral ring extension, where $A$ is a normal domain, and let $\varpi\in A$ be a non-zero non-unit in $A$ such that $B$ is integrally closed in $B[\varpi^{-1}]$. If $A$ is weakly $\varpi$-Shilov, then so is $B$. If $A$ is weakly $\varpi$-Krull and $B$ is reduced, then $B$ is weakly $\varpi$-Krull. \end{prop}
\begin{proof}Suppose that $A$ is weakly $\varpi$-Shilov. Let $\mathfrak{q}$ be a prime ideal of $B$ minimal over $\varpi$. By ``going down" (\cite{Cohen-Seidenberg}~Theorem~5), $\mathfrak{p}=\mathfrak{q}\cap A$ is a minimal prime ideal over $\varpi$ in $A$. By Proposition~\ref{Integral extensions and valuation rings of rank one, first result}(2) applied to the integral extension \begin{equation*}A_{\mathfrak{p}}\hookrightarrow B_{\mathfrak{p}}=(A\setminus \mathfrak{p})^{-1}B,\end{equation*}we see that the $\varpi$-adic completion of $B_{\mathfrak{q}}$ is a valuation ring of rank $1$. Since $\mathfrak{q}$ was an arbitrary minimal prime over $\varpi$, we conclude that $B$ is weakly $\varpi$-Shilov. 

Finally, if $A$ is weakly $\varpi$-Krull and $B$ is reduced, then for every minimal prime $\mathfrak{q}$ over $\varpi$ in $B$ and $\mathfrak{p}=\mathfrak{q}\cap A$, the local ring $A_{\mathfrak{p}}$ is a valuation ring of rank $1$ (we again use ``going down" to make sure that the prime ideal $\mathfrak{p}$ of $A$ is minimal over $\varpi$). Applying Proposition~\ref{Integral extensions and valuation rings of rank one, first result}(3) to the integral extension $B_{\mathfrak{p}}=(A\setminus \mathfrak{p})^{-1}B$ of $A_{\mathfrak{p}}$, we conclude that $B_{\mathfrak{q}}$ is a valuation ring of rank $1$, so $B$ is weakly $\varpi$-Krull.\end{proof}
We deduce from Proposition~\ref{Integral extensions, first result} another sufficient condition for a Tate ring $\mathcal{A}$ to satisfy Berkovich's description of the Shilov boundary. Note that this settles case (5) of Theorem~\ref{Main theorem 1} in the introduction.
\begin{prop}\label{Large value groups and integral extensions}Let $\mathcal{A}$ be a uniform Tate ring, with topologically nilpotent unit $\varpi$, and let $\mathcal{B}$ be a reduced uniform Tate ring which is an integral extension of $\mathcal{A}$ such that $\mathcal{B}^{\circ}$ is integral over $\mathcal{A}^{\circ}$. Suppose that the spectral seminorm $\vert\cdot\vert_{\spc,\varpi}$ on $\mathcal{B}$ satisfies \begin{equation*}\vert\mathcal{B}\vert_{\spc,\varpi}\subseteq \sqrt{\vert\mathcal{B}^{\times,m}\vert_{\spc,\varpi}}\cup\{0\}\end{equation*}and that $\mathcal{B}^{\circ}/(\varpi)_{\mathcal{B}^{\circ}}$ has only finitely many minimal prime ideals. Suppose further that $\mathcal{A}$ is a normal domain and $\mathcal{B}$ is torsion-free over $\mathcal{A}$. If $\mathcal{A}$ satisfies Berkovich's description of the Shilov boundary, then so does $\mathcal{B}$ (and hence so does its completion $\widehat{\mathcal{B}}$).\end{prop}
\begin{proof}By Theorem~\ref{Strongly Shilov} (applied to $\mathcal{B}$), to prove that $\mathcal{B}$ satisfies Berkovich's description of the Shilov boundary, it suffices to prove that $\mathcal{B}^{\circ}$ is strongly $\varpi$-Shilov. By assumption, $\mathcal{B}^{\circ}$ satisfies the hypotheses of Theorem~\ref{Spectral seminorm and S1}, so $\mathcal{B}^{\circ}/(\varpi)_{\mathcal{B}^{\circ}}$ satisfies S1 by that theorem. Hence it remains to prove that $\mathcal{B}^{\circ}$ is weakly $\varpi$-Shilov. Since $\mathcal{A}$ satisfies Berkovich's description of the Shilov boundary, Proposition~\ref{Weakly Shilov} ensures that $\mathcal{A}^{\circ}$ is weakly $\varpi$-Shilov. But $\mathcal{B}^{\circ}$ is integral and torsion-free over $\mathcal{A}^{\circ}$ and $\mathcal{A}^{\circ}$ is a normal domain. We conclude by applying Proposition~\ref{Integral extensions, first result}.\end{proof}
For flat integral extensions, we can prove a stronger assertion.
\begin{thm}\label{Flat integral extensions}Let $\varpi\in A$ be a non-zero-divisor and non-unit in a ring $A$ such that $A/(\varpi)_{A}$ has finitely many minimal prime ideals and satisfies condition S1. Let $A\hookrightarrow B$ be a flat integral extension of $A$. Then $B/(\varpi)_{B}$ satisfies condition S1.\end{thm}
\begin{proof}Let $\mathfrak{q}$ be a prime ideal of $B$ which is not minimal over $\varpi$. We show that $\mathfrak{q}$ is not a weakly associated prime ideal of $\varpi$ in $B$. Consider the prime ideal $\mathfrak{p}=\mathfrak{q}\cap A$ of $A$. By lying over and going up, $\mathfrak{p}$ is not minimal over $\varpi$. By prime avoidance and the assumption that $A/(\varpi)_{A}$ has only finitely many minimal prime ideals, there exists an element $f\in \mathfrak{p}$ which does not belong to any minimal prime ideal of $\varpi$ in $A$. Since $A/(\varpi)_{A}$ satisfies condition S1, this implies that $f$ is not a zero-divisor on $A/(\varpi)$. But the map $A/(\varpi)_{A}\to B/(\varpi)_{B}$ is flat (the map $A\hookrightarrow B$ being flat), so $f\in\mathfrak{q}$ is not a zero-divisor on $B/(\varpi)_{B}$. We conclude by \cite{Stacks}, Tag~05C3, that $\mathfrak{q}$ is not a weakly associated prime ideal of $\varpi$ in $B$. \end{proof}
\begin{cor}\label{Flat integral extensions 2}Let $A\hookrightarrow B$ be a flat integral extension, where $A$ is a normal domain, and let $\varpi\in A$ be a non-zero non-unit such that $A/(\varpi)_{A}$ has only finitely many minimal prime ideals and such that $B$ is integrally closed in $B[\varpi^{-1}]$. If $A$ is strongly $\varpi$-Shilov, then so is $B$. If $A$ is strongly $\varpi$-Krull and $B$ is reduced, then $B$ is strongly $\varpi$-Krull. \end{cor}
\begin{proof}By Proposition~\ref{Integral extensions, first result}, $B$ is weakly $\varpi$-Shilov (respectively, weakly $\varpi$-Krull). By Theorem~\ref{Flat integral extensions}, $B/(\varpi)_{B}$ satisfies condition S1.\end{proof}
We can now prove case (6) of Theorem~\ref{Main theorem 1}.
\begin{cor}\label{Flat integral extensions 3}Let $\mathcal{A}\subseteq \mathcal{B}$ be an integral extension of reduced uniform Tate rings, with topologically nilpotent unit $\varpi\in\mathcal{A}$, such that the map $\mathcal{A}^{\circ}\hookrightarrow \mathcal{B}^{\circ}$ is integral and flat. Suppose that $\mathcal{A}$ is a normal domain and $\mathcal{A}^{\circ}/(\varpi)_{\mathcal{A}^{\circ}}$ has only finitely many minimal prime ideals. If $\mathcal{A}$ satisfies Berkovich's description of the Shilov boundary, so does $\mathcal{B}$.\end{cor}
\begin{proof}Since $\mathcal{A}$ satisfies Berkovich's description of the Shilov boundary and $\mathcal{A}^{\circ}/(\varpi)_{\mathcal{A}^{\circ}}$ has only finitely many minimal prime ideals, $\mathcal{A}^{\circ}$ is strongly $\varpi$-Shilov by Theorem~\ref{Strongly Shilov}. By Corollary~\ref{Flat integral extensions 2} applied to the flat integral extension $\mathcal{A}^{\circ}\hookrightarrow \mathcal{B}^{\circ}$, we see that $\mathcal{B}^{\circ}$ is also strongly $\varpi$-Shilov. By Theorem~\ref{Strongly Shilov}, this implies that $\mathcal{B}$ satisfies Berkovich's description of the Shilov boundary. \end{proof}  
Next, we prove a very general result on preservation of the strong $\varpi$-Shilov property for Noetherian domains. To this end, we first establish the following general fact about condition S1. The proof we give below is based on an argument the author learned from Linquan Ma.
\begin{prop}\label{Filtered colimits}Let $(A_{i})_{i\geq0}$ be a directed system of rings with torsion-free, integral, injective transition maps and let $\varpi\in A=A_{0}$ be an element which is a non-unit and non-zero-divisor in each $A_{i}$. Set \begin{equation*}A_{\infty}=\varinjlim_{i\geq0}A_{i}.\end{equation*}Suppose that $A$ is a normal domain, that $A/(\varpi)_{A}$ has only finitely many minimal prime ideals and that each $A_{i}$ is integrally closed in $\mathcal{A}_{i}=A_{i}[\varpi^{-1}]$. If $A_{i}/(\varpi)_{A_{i}}$ satisfies condition S1 for all $i\geq0$, then $A_{\infty}/(\varpi)_{A_{\infty}}$ also satisfies condition S1.\end{prop}
\begin{proof}By contradiction, suppose that $A_{\infty}/(\varpi)_{A_{\infty}}$ does not satisfy S1, so there exists a weakly associated prime ideal $\mathfrak{q}$ of $\varpi$ in $A_{\infty}$ which is not minimal over $\varpi$. By lying over and going up, $\mathfrak{q}_{i}=\mathfrak{q}\cap A_{i}$ is not minimal over $\varpi$ for any $i\geq0$. In particular, $\mathfrak{q}_{0}$ is a prime ideal of $A$ containing $\varpi$ which is not minimal over $\varpi$. By prime avoidance (and the assumption that $A/(\varpi)_{A}$ has only finitely many minimal prime ideals), there exists an element $f\in \mathfrak{q}_{0}$ which does not belong to any minimal prime ideal of $\varpi$ in $A$. Since $f$ belongs to the weakly associated prime ideal $\mathfrak{q}$ of $\varpi$ in $A_{\infty}$, it is a zero-divisor on $A_{\infty}/(\varpi)_{A_{\infty}}$. Since each $A_{i}$ is integrally closed in $\mathcal{A}_{i}$, the ideal $(\varpi)_{A_{i}}$ of $A_{i}$ is integrally closed by Proposition~\ref{Rings of integral elements} for every $i\geq0$. Then the maps $A_{i}/(\varpi)_{A_{i}}\to A_{j}/(\varpi)_{A_{j}}$ induced by the integral and injective transition maps $A_{i}\to A_{j}$ are injective, by \cite{Swanson-Huneke}~Proposition~1.6.1. In particular, we can write \begin{equation*}A_{\infty}/(\varpi)_{A_{\infty}}=\varinjlim_{i\geq0}A_{i}/(\varpi)_{A_{i}}\end{equation*}and see that there exists an index $j\geq0$ such that $f$ is already a zero-divisor on $A_{j}/(\varpi)_{A_{j}}$. By assumption, $A_{j}/(\varpi)_{A_{j}}$ satisfies S1, so $f$ belongs to some minimal prime ideal of $\varpi$ in $A_{j}$. But $A$ is a normal domain and the integral extension $A\hookrightarrow A_{j}$ is torsion-free, so, by going down, $f$ must then belong to a minimal prime ideal of $\varpi$ in $A$. This provides the desired contradiction. \end{proof}
\begin{cor}\label{Filtered colimits 2}Let $\varpi\in A$ be a non-zero non-unit in a normal domain $A$ such that $A/(\varpi)_{A}$ has only finitely many minimal prime ideals and let $(A_{i})_{i\geq0}$ be a directed system of torsion-free integral extensions of $A$ with $A_{0}=A$. If each $A_{i}$ is strongly $\varpi$-Shilov (respectively, strongly $\varpi$-Krull and reduced), then the direct limit $A_{\infty}=\varinjlim_{i}A_{i}$ is strongly $\varpi$-Shilov (respectively, strongly $\varpi$-Krull).\end{cor}
\begin{proof}By Proposition~\ref{Integral extensions, first result}, we already know that $A_{\infty}$ is weakly $\varpi$-Shilov (respectively, weakly $\varpi$-Krull), so we only need to check that $A_{\infty}/(\varpi)_{A_{\infty}}$ satisfies condition S1. By Corollary~\ref{Strongly Shilov implies completely integrally closed}, $A_{i}$ being strongly $\varpi$-Shilov implies that $A_{i}$ is integrally closed (in fact, even completely integrally closed) in $\mathcal{A}_{i}=A_{i}[\varpi^{-1}]$. Hence we can apply Proposition~\ref{Filtered colimits} to conclude that $A_{\infty}/(\varpi)_{A_{\infty}}$ satisfies S1.\end{proof}
For a ring $A$ which is not necessarily an integral domain, we say that $A$ is normal if it is reduced and integrally closed in its total ring of fractions.   
\begin{thm}\label{Integral extensions of Noetherian rings}Let $A$ be an integral domain whose integral closure $\overline{A}$ in its field of fractions is a Krull domain and let $\varpi\in A$ be a non-zero non-unit. Let $A\hookrightarrow B$ be a torsion-free integral ring extension with $B$ reduced and integrally closed in $\mathcal{B}=B[\varpi^{-1}]$. Suppose that $A\hookrightarrow B$ factors through an injective map $\overline{A}\hookrightarrow B$ (for example, this is true if either $A$ or $B$ is normal). Then $B$ is strongly $\varpi$-Krull. In particular, the $\varpi$-adic completion $\widehat{B}$ of $B$ is strongly $\varpi$-Shilov. \end{thm}
We need the following elementary lemma. 
\begin{lemma}\label{Krull domains and finite field extensions}Let $A$ be a Krull domain and let $\Lambda$ be a reduced $\Frac(A)$-algebra which is finite-dimensional as a $\Frac(A)$-vector space. Then the integral closure of $A$ inside $\Lambda$ is a finite product of Krull domains.\end{lemma}
\begin{proof}Let $B$ be the integral closure of $A$ inside $\Lambda$. By assumption, $\Lambda$ is a reduced Artinian ring, so it is a finite product of fields. Write $\Lambda=\prod_{i=1}^{r}K_{i}$ where $K_{i}$ are finite field extensions of $\Frac(A)$. By Lemma~\ref{Finite extensions}, we know that $B$ decomposes as $B=\prod_{i=1}^{r}B_{i}$, where $B_{i}$ is the integral closure of $A$ in $K_{i}$. Since $A$ is a Krull domain and each $K_{i}$ is a finite field extension of $\Frac(A)$, the integral closure $B_{i}$ of $A$ inside $K_{i}$ is a Krull domain by \cite{BourbakiAlg}~Ch~VII, \S1, no.~8, Proposition~12. \end{proof}
\begin{proof}[Proof of Theorem~\ref{Integral extensions of Noetherian rings}]Replacing $A$ with $\overline{A}$, we may assume that $A$ is a Krull domain. Write \begin{equation*}B=\varinjlim_{i\geq0}C_{i},\end{equation*}where $C_{i}$ is finite over $A$ and $C_{0}=A$. Let $\mathcal{C}_{i}^{+}$ be the integral closure of $C_{i}$ in $\mathcal{C}_{i}=C_{i}[\varpi^{-1}]$. Note that $\mathcal{C}_{i}^{+}$ is also the integral closure of $A$ in $\mathcal{C}_{i}$. Since $B$ is integrally closed in $\mathcal{B}$, each $\mathcal{C}_{i}^{+}$ is contained in $B$. In particular, each $\mathcal{C}_{i}^{+}$ is reduced. Since $\mathcal{C}_{i}$ is finite over $\mathcal{A}=A[\varpi^{-1}]$, the total ring of fractions of $\mathcal{C}_{i}^{+}$ is a reduced $\Frac(A)$-algebra which is finite dimensional as a $\Frac(A)$-vector space. Hence the integral closure of $\mathcal{C}_{i}^{+}$ in its total ring of fractions is a finite product of Krull domains by Lemma~\ref{Krull domains and finite field extensions}. By Theorem~\ref{Reduced Noetherian rings are strongly Shilov}, this entails that each $\mathcal{C}_{i}^{+}$ is strongly $\varpi$-Krull. Since $A$ is a Krull domain, there are only finitely many minimal prime ideals of $A/(\varpi)_{A}$. Hence we can apply Corollary~\ref{Filtered colimits 2} to the directed system $(\mathcal{C}_{i}^{+})_{i\geq0}$ (with base ring $\mathcal{C}_{0}^{+}=A$) to conclude that \begin{equation*}B=\varinjlim_{i\geq0}\mathcal{C}_{i}^{+}\end{equation*}is strongly $\varpi$-Krull. \end{proof}
As an immediate consequence of Theorem \ref{Integral extensions of Noetherian rings}, we obtain a result on the completed absolute integral closure of any Noetherian domain.
\begin{cor}\label{Absolute integral closure}For every Noetherian domain $R$ and for every non-zero non-unit $\varpi\in R$, the absolute integral closure $R^{+}$ of $R$ is strongly $\varpi$-Krull and its $\varpi$-adic completion $\widehat{R^{+}}$ is strongly $\varpi$-Shilov.\end{cor}
\begin{rmk}By contrast, as we explained in Example~\ref{Example of Heitmann-Ma}, the ring $\widehat{R^{+}}$, for $R$ a complete Noetherian local domain of mixed characteristic $(0, p)$ and $\varpi=p$, is not strongly $\varpi$-Krull, by a result of Heitmann and Ma, \cite{Heitmann-Ma26}~Proposition~4.10.\end{rmk}  
Of course, Theorem~\ref{Integral extensions of Noetherian rings} implies that completed normal integral extensions of uniform Tate domains with Noetherian ring of definition satisfy Berkovich's description of the Shilov boundary (apply Theorem~\ref{Strongly Shilov}). With some more work, we can deduce from Theorem~\ref{Integral extensions of Noetherian rings} that such completed integral extensions actually satisfy Berkovich's description of the Shilov boundary in the strong sense.
\begin{thm}\label{Integral extensions of Noetherian rings 2}Let $\mathcal{A}$ be a uniform Tate domain with topologically nilpotent unit $\varpi$, admitting a Noetherian ring of definition which contains $\varpi$. Suppose that $\mathcal{A}\hookrightarrow\mathcal{B}$ is a torsion-free integral extension of reduced uniform Tate rings such that $\mathcal{B}^{\circ}$ is integral over $\mathcal{A}^{\circ}$. If either $\mathcal{A}$ or $\mathcal{B}$ is normal, then $\mathcal{B}$ (and, consequently, also its completion $\widehat{\mathcal{B}}$) satisfies Berkovich's description of the Shilov boundary in the strong sense.\end{thm}
We will deduce Theorem~\ref{Integral extensions of Noetherian rings 2} from a series of lemmas.
\begin{lemma}\label{Minimal primes and integral extensions 2}Let $A\hookrightarrow B$ be a torsion-free integral ring extension between a normal domain $A$ and a reduced ring $B$, let $\varpi\in A$ be a non-zero non-unit and let $(C_{i})_{i\geq0}$ be a direct system of subextensions of $B$ with \begin{equation*}B=\varinjlim_{i}C_{i}.\end{equation*}Consider the Tate rings $\mathcal{B}=B[\varpi^{-1}]$ and $\mathcal{C}_{i}=C_{i}[\varpi^{-1}]$. An element $v\in\mathcal{M}(\mathcal{B})$ belongs to $\spc^{-1}(\Min_{B}(\varpi))$ if and only if $v\vert_{\mathcal{C}_{i}}\in \spc^{-1}(\Min_{C_{i}}(\varpi))$ for every $i$ (or, equivalently, for some $i$).\end{lemma}
\begin{proof}Suppose that $v\in \mathcal{M}(\mathcal{B})$ is an element of the Berkovich spectrum $\mathcal{M}(\mathcal{B})$ such that $\spc(v)$ is a minimal prime of $\varpi$ in $B$. Let $\mathcal{A}=A[\varpi^{-1}]$. By going down, \begin{equation*}\spc(v\vert_{\mathcal{A}})=\spc(v)\cap A\end{equation*}is a minimal prime of $\varpi$ in $A$. By lying over and going up, \begin{equation*}\spc(v\vert_{\mathcal{C}_{i}})=\spc(v)\cap C_{i}\end{equation*}is then a minimal prime of $\varpi$ in $C_{i}$ for every $i$. Conversely, if $\spc(v\vert_{\mathcal{C}_{i}})=\spc(v)\cap C_{i}$ is minimal over $\varpi$ in $C_{i}$ for some $i$, then $\spc(v)$ itself is minimal over $\varpi$ in $B$, by lying over and going up. \end{proof}
\begin{lemma}\label{Closed subsets of the Berkovich spectrum}Let $A\hookrightarrow B$ be an injective ring map and let $\varpi\in A$ be a non-zero-divisor and non-unit in $A$ and in $B$. Let $(C_{i})_{i}$ be a direct system of $A$-subalgebras of $B$ with injective transition maps such that $B=\varinjlim_{i}C_{i}$. Consider the Tate rings $\mathcal{B}=B[\varpi^{-1}]$, $\mathcal{C}_{i}=C_{i}[\varpi^{-1}]$, and, for every $i$, let $\mathcal{S}_{i}\subseteq \mathcal{M}(\mathcal{C}_{i})$ be a closed subset of $\mathcal{M}(\mathcal{C}_{i})$. Then the subset \begin{equation*}\mathcal{S}=\{\, v\in\mathcal{M}(\mathcal{B})\mid v\vert_{\mathcal{C}_{i}}\in\mathcal{S}_{i}~\textrm{for all}~i\,\}\end{equation*}is a closed subset of $\mathcal{M}(\mathcal{B})$.\end{lemma}
\begin{proof}Let $(v_{n})_{n}$ be a net in $\mathcal{S}$ with limit $v$ in $\mathcal{M}(\mathcal{B})$. Thus $v_{n}(f)\to v(f)$ for all $f\in \mathcal{B}$. In particular, for every $i$, we have $v_{n}(f)\to v(f)$ for all $f\in \mathcal{C}_{i}$. This means that $v_{n}\vert_{\mathcal{C}_{i}}\to v\vert_{\mathcal{C}_{i}}$ in $\mathcal{M}(\mathcal{C}_{i})$. This means that $v\vert_{\mathcal{C}_{i}}\in \mathcal{S}_{i}$ for every $i$, since $\mathcal{S}_{i}$ is closed in $\mathcal{M}(\mathcal{C}_{i})$ for every $i$. In other words, $v\in\mathcal{S}$. \end{proof}
\begin{lemma}\label{Finitely many minimal primes}Let $A$ be a reduced ring whose integral closure in its total ring of fractions is a finite product of Krull domains. Then for any non-zero-divisor and non-unit $\varpi\in A$, the integral closure $\mathcal{A}^{+}$ of $A$ inside $\mathcal{A}=A[\varpi^{-1}]$ has only finitely many minimal prime ideals of $\varpi$.\end{lemma}
\begin{proof}The ring $\mathcal{A}^{+}$ is integral over $A$ with the same total ring of fractions, so it is contained in the integral closure $\overline{A}$ of $A$ in its total ring of fractions. By assumption, $\overline{A}$ is a finite product of Krull domains, so there exist only finitely many minimal primes of $\varpi$ in $\overline{A}$. But any minimal prime of $\varpi$ in $\mathcal{A}^{+}$ is the restriction of a minimal prime of $\varpi$ in its integral extension $\overline{A}$. \end{proof}
\begin{lemma}\label{Integral extensions and closed subsets}Let $A\hookrightarrow B$ be a torsion-free integral extension between a Krull domain $A$ and a reduced ring $B$. Let $\varpi\in A$ be a non-zero non-unit such that $B$ is integrally closed in $\mathcal{B}=B[\varpi^{-1}]$. Then $\spc^{-1}(\Min_{B}(\varpi))$ is closed in $\mathcal{M}(\mathcal{B})$.\end{lemma}
\begin{proof}Write $B$ as a direct limit of finite subextensions $C_{i}$ of $A$. Let $\mathcal{C}_{i}=C_{i}[\varpi^{-1}]$ and let $\mathcal{C}_{i}^{+}$ denote the integral closure of $C_{i}$ in $\mathcal{C}_{i}$. For every $i$, consider also the integral closure $\overline{C_{i}}$ of $C_{i}$ in its total ring of fractions. 
Set $\overline{\mathcal{C}_{i}}=\overline{C_{i}}[\varpi^{-1}]$. Using the assumption that $B$ is integrally closed in $\mathcal{B}$, we can write $B=\varinjlim_{i}\mathcal{C}_{i}^{+}$. By Lemma~\ref{Minimal primes and integral extensions 2} applied to the torsion-free integral extension \begin{equation*}A\hookrightarrow B,\end{equation*}we have \begin{equation*}\spc^{-1}(\Min_{B}(\varpi))=\{\, v\in \mathcal{M}(\mathcal{B})\mid v\vert_{\mathcal{C}_{i}}\in \spc^{-1}(\Min_{\mathcal{C}_{i}^{+}}(\varpi))~\textrm{for all}~i\,\}.\end{equation*}By Lemma~\ref{Closed subsets of the Berkovich spectrum}, to prove that $\spc^{-1}(\Min_{B}(\varpi))$ is closed in $\mathcal{M}(\mathcal{B})$, we only need to prove that, for every $i$,\begin{equation*}\spc^{-1}(\Min_{\mathcal{C}_{i}^{+}}(\varpi))\end{equation*}is closed in $\mathcal{M}(\mathcal{C}_{i})$. In particular, it suffices to show that these pre-images are finite. Recall from Proposition~\ref{Weakly Shilov} that every minimal prime ideal of $\varpi$ in a weakly $\varpi$-Shilov ring $C$ has a unique pre-image in the Berkovich spectrum of $C[\varpi^{-1}]$. Hence it suffices to show that, for every $i$, the ring $\mathcal{C}_{i}^{+}$ is weakly $\varpi$-Shilov and has only finitely many minimal prime ideals of $\varpi$.

Now, the total ring of fractions $\Frac(C_{i})$ of $C_{i}$ is a reduced finite-dimensional algebra over the field of fractions $\Frac(A)$ of $A$. By Lemma~\ref{Krull domains and finite field extensions}, the integral closure $\overline{C_{i}}$ of the Krull domain $A$ in $\Frac(C_{i})$ is a finite product of Krull domains. In particular, $\overline{C_{i}}$ has only finitely many minimal prime ideals of $\varpi$ and is strongly $\varpi$-Shilov. On the other hand, $\overline{C_{i}}$ is also the integral closure of $\mathcal{C}_{i}^{+}$ in $\Frac(C_{i})$. Hence $\mathcal{C}_{i}^{+}$ has only finitely many minimal prime ideals of $\varpi$ by Lemma~\ref{Finitely many minimal primes} and is strongly $\varpi$-Shilov by Theorem~\ref{Reduced Noetherian rings are strongly Shilov}. \end{proof}
\begin{proof}[Proof of Theorem~\ref{Integral extensions of Noetherian rings 2}]Let $A$ be a Noetherian ring of definition of $\mathcal{A}$ with $\varpi\in A$. Since $A$ is Noetherian, the integral closure of $A$ in $\mathcal{A}$ is completely integrally closed in $\mathcal{A}$. Thus $\mathcal{A}^{\circ}$ is an integral extension of $A$. By the Nagata-Mori Theorem, the integral closure $\overline{A}$ of $A$ in its field of fractions is a Krull domain and the hypothesis that either $\mathcal{A}$ or $\mathcal{B}$ is normal ensures that the integral map $A\hookrightarrow \mathcal{A}^{\circ}\hookrightarrow\mathcal{B}^{\circ}$ factors through an injective map $\overline{A}\hookrightarrow \mathcal{B}^{\circ}$. Applying Theorem~\ref{Integral extensions of Noetherian rings} to the torsion-free integral extension \begin{equation*}A\hookrightarrow \mathcal{A}^{\circ}\hookrightarrow \mathcal{B}^{\circ},\end{equation*}we see that $\mathcal{B}^{\circ}$ is strongly $\varpi$-Shilov. By Theorem~\ref{Strongly Shilov}, this implies that $\mathcal{B}$ satisfies Berkovich's description of the Shilov boundary, so every minimal prime of $\varpi$ in $\mathcal{B}^{\circ}$ has a unique pre-image in $\mathcal{M}(\mathcal{B})$ under $\spc$ and the Shilov boundary for $\mathcal{B}$ is given by the closure of $\spc^{-1}(\Min_{\mathcal{B}^{\circ}}(\varpi))$. On the other hand, we can apply Lemma~\ref{Integral extensions and closed subsets} to the integral extension $\overline{A}\hookrightarrow \mathcal{B}^{\circ}$ and conclude that the subset $\spc^{-1}(\Min_{\mathcal{B}^{\circ}}(\varpi))$ of $\mathcal{M}(\mathcal{B})$ is already closed.\end{proof}
\begin{rmk}This finishes the proof of case (2) of Theorem~\ref{Main theorem 1} in the introduction, so now the entirety of Theorem~\ref{Main theorem 1} is proved.\end{rmk}        
We now want to remove the Noetherian hypothesis in Theorem~\ref{Integral extensions of Noetherian rings}, which we are able to do under some mild assumptions (see Theorem~\ref{Strongly Shilov rings and integral extensions} below). We begin with some general observations concerning seminorms on integral ring extensions. Let $A\hookrightarrow B$ be an integral ring extension and let $\varpi\in A$ be an element which is a non-unit and non-zero-divisor on $A$ and on $B$. Recall that by Lemma~\ref{Spectral isometries} the spectral seminorm (relative to $\varpi$) on $\mathcal{B}$ restricts to the analogous spectral seminorm on $\mathcal{A}$. This allows us to denote both spectral seminorms by the same symbol $\vert\cdot\vert_{\spc,\varpi}$ without ambiguity.  
\begin{lemma}\label{Integral extensions and the Berkovich spectrum}Let $A\hookrightarrow B$ be an integral ring extension and let $\varpi\in A$ be a non-unit non-zero-divisor in $A$ and in $B$. For a multiplicative seminorm $w$ on $\mathcal{B}=B[\varpi^{-1}]$, the following are equivalent: \begin{enumerate}[(1)] \item $w$ is bounded with respect to $\vert\cdot\vert_{\spc,\varpi}$; \item the restriction $w\vert_{\mathcal{A}}$ of $w$ to $\mathcal{A}=A[\varpi^{-1}]$ is bounded with respect to $\vert\cdot\vert_{\spc,\varpi}$.\end{enumerate}\end{lemma}
\begin{proof}By Lemma~\ref{Spectral isometries}, we only need to prove (2)$\Rightarrow$(1). Suppose that $v=w\vert_{\mathcal{A}}$ is bounded with respect to $\vert\cdot\vert_{\spc,\varpi}$. Then $w(\varpi)\leq\vert\varpi\vert_{\spc,\varpi}=\frac{1}{2}$ and $w(\varpi)^{-1}=w(\varpi^{-1})\leq\vert\varpi^{-1}\vert_{\spc,\varpi}=2$. Thus, by Lemma~\ref{Boundedness vs. continuity}, it suffices to prove that $w$ is continuous with respect to the canonical extension $\lVert\cdot\rVert_{B,\varpi}$ of the $\varpi$-adic seminorm on $B$. Thus we have to prove that for every $r>0$ the subset $\{\, f\in B\mid w(f)<r\,\}$ is open with respect to the topology on $\mathcal{B}$ defined by the pair of definition $(B, \varpi)$. Given $r>0$, choose $n\geq1$ such that $\lVert\varpi^{n}\rVert_{A,\varpi}=2^{-n}<r$. Then $w(\varpi^{n})<r$, so \begin{equation*}\varpi^{n}B\subseteq \{\, f\in B\mid w(f)<r\,\},\end{equation*}provided that $w(f)\leq1$ for all $f\in B$. 

Hence it remains to prove that $w(f)\leq1$ for all $f\in B$. To this end, let \begin{equation*}f^{n}+a_{1}f^{n-1}+\dots+a_{n-1}f+a_{n}=0\end{equation*}be an equation of integral dependence of $f\in B$ over $A$. Then \begin{equation*}w(f^{n})\leq\max_{1\leq i\leq n}w(a_{i})w(f^{n-i}).\end{equation*}Choose $i\in\{\,1,\dots, n\,\}$ such that $w(f)^{n}=w(f^{n})\leq w(a_{i})w(f^{n-i})=w(a_{i})w(f)^{n-i}$, so that \begin{equation*}w(f)\leq w(a_{i})^{1/i}.\end{equation*}Since $w\vert_{\mathcal{A}}$ is bounded with respect to $\vert\cdot\vert_{\spc,\varpi}$, we have $w(a_{i})\leq1$ for all $i$. It follows that $w(f)\leq1$, as desired. \end{proof}
We deduce from the above lemma the following more general statement. 
\begin{lemma}\label{Integral extensions and the Berkovich spectrum 2}In the situation of the above lemma, the following are equivalent for a power-multiplicative seminorm $\vert\cdot\vert$ on $\mathcal{B}=B[\varpi^{-1}]$: \begin{enumerate}[(1)]\item $\vert\cdot\vert$ is bounded with respect to $\vert\cdot\vert_{\spc,\varpi}$; \item the restriction of $\vert\cdot\vert$ to $\mathcal{A}=A[\varpi^{-1}]$ is bounded with respect to $\vert\cdot\vert_{\spc,\varpi}$.\end{enumerate}\end{lemma}
\begin{proof}By Lemma~\ref{Bounded implies submetric}, a power-multiplicative seminorm $\vert\cdot\vert$ on a ring is bounded with respect to some seminorm $\lVert\cdot\rVert$ if and only if it is bounded above by $\lVert\cdot\rVert$, i.e., if and only if $\vert\cdot\vert\leq\lVert\cdot\rVert$. By virtue of \cite{Berkovich}~Theorem~1.3.1, a power-multiplicative seminorm $\vert\cdot\vert$ on a ring is bounded above by some seminorm $\lVert\cdot\rVert$ if and only if every multiplicative seminorm bounded above by $\vert\cdot\vert$ is bounded with respect to $\lVert\cdot\rVert$. Hence the assertion follows from Lemma~\ref{Integral extensions and the Berkovich spectrum}. \end{proof}
\begin{lemma}\label{Integral extensions and additive subgroups}Let $A\hookrightarrow B$ be an integral ring extension and let $\varpi\in A$ be a non-zero-divisor and non-unit in $A$ and in $B$. Suppose that $A$, $B$ are integrally closed in $\mathcal{A}=A[\varpi^{-1}]$, $\mathcal{B}=B[\varpi^{-1}]$, respectively. Set $U=\mathcal{B}^{\circ\circ}\cap \mathcal{A}$ and let $U^{\prime}$ be the subset of elements of $\mathcal{B}$ which satisfy an equation of integral dependence with coefficients in $U$. Let $\mathcal{B}^{+}$ be the integral closure of $\mathcal{A}^{\circ}$ inside $\mathcal{B}$. Then $U=\sqrt{(\varpi)_{\mathcal{A}^{\circ}}}$ and $U^{\prime}=\sqrt{(\varpi)_{\mathcal{B}^{+}}}$.\end{lemma}
\begin{proof}Since $A$, $B$ are integrally closed in $\mathcal{A}$, $\mathcal{B}$, we know from Lemma~\ref{Almost power-bounded} that the Tate rings $\mathcal{A}$, $\mathcal{B}$ are uniform, so $\vert\cdot\vert_{\spc,\varpi}$ is a seminorm defining the topology on each of them (see Lemma~\ref{Uniform Tate rings} and its proof). It follows from this and Lemma~\ref{Spectral isometries} that $\mathcal{B}^{\circ\circ}\cap \mathcal{A}=\mathcal{A}^{\circ\circ}=\sqrt{(\varpi)_{\mathcal{A}^{\circ}}}$. Consequently, $U^{\prime}$ is the subset of $\mathcal{B}^{+}$ consisting of elements satisfying an equation of integral dependence with coefficients in $\sqrt{(\varpi)_{\mathcal{A}^{\circ}}}$. By \cite{Cohen-Seidenberg}~Lemma~1, this means that $U^{\prime}=\sqrt{(\varpi)_{\mathcal{B}^{+}}}$.\end{proof}
\begin{prop}[Generalization of \cite{Guennebaud}~Ch.~II, Prop.~2]\label{Spectral seminorms on integral extensions 0}Let $\mathcal{A}\hookrightarrow \mathcal{B}$ be an integral extension between uniform Tate rings and let $\varpi\in \mathcal{A}$ be a topologically nilpotent unit. Suppose that $\mathcal{A}\hookrightarrow \mathcal{B}$ is an isometry with respect to the two spectral seminorms $\vert\cdot\vert_{\spc,\varpi}$ and that every power-multiplicative seminorm on $\mathcal{B}$ whose restriction to $\mathcal{A}$ is bounded above by $\vert\cdot\vert_{\spc,\varpi}$ on $\mathcal{A}$ is also bounded above by $\vert\cdot\vert_{\spc,\varpi}$ on $\mathcal{B}$. Then for all $f\in \mathcal{B}$ the spectral seminorm $\vert f\vert_{\spc,\varpi}$ of $f$ has the form \begin{equation*}\vert f\vert_{\spc,\varpi}=\inf(\max_{1\leq i\leq n}\vert a_{i}\vert_{\spc,\varpi}^{1/i}),\end{equation*}where the infimum is taken over all equations of integral dependence \begin{equation*}X^{n}+a_{1}X^{n-1}+\dots+a_{n-1}X+a_{n}=0,\end{equation*}$a_{i}\in \mathcal{A}$, satisfied by $f$.\end{prop}
\begin{proof}Our proof follows the strategy of the proof of \cite{Guennebaud}~Ch.~II, Proposition~2, using our results on generalized gauges from Section~\ref{sec:power-multiplicative seminorms} in place of the more classical results on gauges on $K$-vector spaces over a nonarchimedean field $K$ used by Guennebaud in loc.~cit.  

Define a function \begin{align*}\nu: \mathcal{B}\to \mathbb{R}_{\geq0}, \\ \nu(f)=\inf(\max_{1\leq i\leq n}\vert a_{i}\vert_{\spc,\varpi}^{1/i}),\end{align*}where the infimum is taken over all equations of integral dependence \begin{equation*}f^{n}+a_{1}f^{n-1}+\dots+a_{n-1}f+a_{n}=0\end{equation*}with $a_{i}\in \mathcal{A}$ satisfied by $f$. It is readily seen that $\vert\cdot\vert_{\spc,\varpi}\leq\nu$. Indeed, if $f^{n}+a_{1}f^{n-1}+\dots+a_{n-1}f+a_{n}=0$ is an equation of integral dependence over $\mathcal{A}=A[\varpi^{-1}]$, we have \begin{equation*}\vert f\vert_{\spc,\varpi}^{n}=\vert f^{n}\vert_{\spc,\varpi}\leq\max_{0\leq i\leq n-1}\vert a_{n-i}f^{i}\vert_{\spc,\varpi}\leq\max_{0\leq i\leq n-1}\vert a_{n-i}\vert_{\spc,\varpi}\vert f\vert_{\spc,\varpi}^{i},\end{equation*}so there exists some $i\in\{0,\dots,n-1\}$ with $\vert f\vert_{\spc,\varpi}^{n-i}\leq \vert a_{n-i}\vert_{\spc,\varpi}$. Since $\mathcal{A}\to \mathcal{B}$ is submetric (in fact, even isometric) with respect to the two spectral seminorms, it follows that \begin{equation*}\vert f\vert_{\spc,\varpi}\leq \max_{1\leq i\leq n}\vert a_{i}\vert_{\spc,\varpi}^{1/i},\end{equation*}as claimed. 

Let $U=\mathcal{B}^{\circ\circ}\cap \mathcal{A}=\mathcal{A}^{\circ\circ}$ and let $U^{\prime}$ be the set of $f\in \mathcal{B}$ which satisfy an equation of integral dependence with coefficients in $U$. In other words, $U^{\prime}=\mathcal{B}_{\nu<1}$. Let $\mathcal{B}^{+}$ be the integral closure of $\mathcal{A}^{\circ}$ in $\mathcal{B}$. Then $\mathcal{B}^{+}[\varpi^{-1}]$ is the integral closure of $\mathcal{A}=\mathcal{A}^{\circ}[\varpi^{-1}]$ in $\mathcal{B}$, so $\mathcal{B}^{+}[\varpi^{-1}]=\mathcal{B}$ as rings. Since $\varpi\in \mathcal{A}^{\circ}$ is not a unit and $\mathcal{A}^{\circ}\hookrightarrow \mathcal{B}^{+}$ is integral, $\varpi$ is not a unit in $\mathcal{B}^{+}$. By Lemma~\ref{Integral extensions and additive subgroups} applied to the integral extension $\mathcal{A}^{\circ}\hookrightarrow \mathcal{B}^{+}$, we see that $U^{\prime}$ is an ideal in the ring of definition $\mathcal{B}^{+}$ of the Tate ring $\mathcal{B}^{+}[\varpi^{-1}]$ containing $\varpi$. By Proposition~\ref{Properties of the generalized gauge 3}, this entails that \begin{equation*}p_{U^{\prime},\varpi}=\vert\cdot\vert_{\spc,\mathcal{B}^{+},\varpi},\end{equation*}where $\vert\cdot\vert_{\spc,\mathcal{B}^{+},\varpi}$ denotes the spectral seminorm of the uniform Tate ring $\mathcal{B}^{+}[\varpi^{-1}]$ (with pair of definition $(\mathcal{B}^{+}, \varpi)$). By Lemma~\ref{Spectral isometries} applied to the integral extension $\mathcal{A}^{\circ}\hookrightarrow \mathcal{B}^{+}$, the restriction of $\vert\cdot\vert_{\spc,\mathcal{B}^{+},\varpi}$ to $\mathcal{A}$ is equal to the spectral seminorm of $\mathcal{A}$. By our assumption on the Tate ring $\mathcal{B}$ and its spectral seminorm, this implies $\vert\cdot\vert_{\spc,\mathcal{B}^{+},\varpi}\leq \vert\cdot\vert_{\spc,\varpi}$. But on the other hand, $\vert\cdot\vert_{\spc,\varpi}\leq \vert\cdot\vert_{\spc,\mathcal{B}^{+},\varpi}$ since $\mathcal{B}^{+}\subseteq \mathcal{B}^{\circ}$. It follows that $\vert\cdot\vert_{\spc,\mathcal{B}^{+},\varpi}=\vert\cdot\vert_{\spc,\varpi}$ and thus \begin{equation*}p_{U^{\prime},\varpi}=\vert\cdot\vert_{\spc,\varpi}.\end{equation*}Therefore, to prove that $\nu\leq \vert\cdot\vert_{\spc,\varpi}$, it remains to prove that the function $\nu$ satisfies the hypotheses of Lemma~\ref{Explicit power-multiplicative seminorms 1}, so that $\nu\leq p_{U^{\prime},\varpi}$ by that lemma.

Since $U^{\prime}$ is an ideal in the subring $\mathcal{B}^{+}$ of $\mathcal{B}$, it is in particular a (multiplicative) subsemigroup of $\mathcal{B}$. Hence the third condition of Lemma~\ref{Explicit power-multiplicative seminorms 1} is satisfied and we only need to verify the first and second conditions of that lemma. To verify that $\nu(\varpi^{m}f)\leq 2^{-m}\nu(f)=\vert\varpi^{m}\vert_{\spc,\varpi}\nu(f)$ for all $m\in\mathbb{Z}$, $f\in \mathcal{B}$ (and hence actually $\nu(\varpi^{m}f)=\vert\varpi^{m}\vert_{\spc,\varpi}\nu(f)$ for all $m\in\mathbb{Z}$, $f\in \mathcal{B}$), let $f\in \mathcal{B}$, $m\in\mathbb{Z}$ and let $f^{n}+a_{1}f^{n-1}+\dots+a_{n-1}f+a_{n}=0$ be an equation of integral dependence of $f$ over $\mathcal{A}$. Then \begin{equation*}(\varpi^{m}f)^{n}+a_{1}\varpi^{m}(\varpi^{m}f)^{n-1}+\dots+a_{n-1}\varpi^{m(n-1)}(\varpi^{m}f)+a_{n}\varpi^{mn}=0\end{equation*}is an equation of integral dependence over $\mathcal{A}$ satisfied by $\varpi^{m}f$. Consequently, \begin{align*}\nu(\varpi^{m}f)\leq\max_{1\leq i\leq n}\vert \varpi^{mi}a_{i}\vert_{\spc,\varpi}^{1/i}\\ \leq \max_{1\leq i\leq n}\vert \varpi^{mi}\vert_{\spc,\varpi}^{1/i}\vert a_{i}\vert_{\spc,\varpi}^{1/i}=\vert\varpi^{m}\vert_{\spc,\varpi}\max_{1\leq i\leq n}\vert a_{i}\vert_{\spc,\varpi}^{1/i}.\end{align*}Since this holds for every equation of integral dependence of $f$ over $\mathcal{A}$, we conclude that $\nu(\varpi^{m}f)\leq\vert\varpi^{m}\vert_{\spc,\varpi}\nu(f)$ for all $f\in \mathcal{B}$, $m\in\mathbb{Z}$, as desired. 

Now we verify, following the last part of the proof of \cite{Guennebaud}~Ch.~II, Proposition~2, that $\nu$ satisfies the second condition in Lemma~\ref{Explicit power-multiplicative seminorms 1}, i.e., we prove that $\nu(f)^{m}\leq\nu(f^{m})$ for all $f\in \mathcal{B}$, $m\in\mathbb{Z}_{>0}$. Given an equation of integral dependence \begin{equation*}(f^{m})^{n}+a_{1}(f^{m})^{n-1}+\dots+a_{n-1}f^{m}+a_{n}=0\end{equation*}of $f^{m}$ over $\mathcal{A}$, we can rewrite this equation as \begin{equation*}f^{mn}+b_{m}f^{mn-m}+\dots+b_{nm-m}f^{m}+b_{nm}=0\end{equation*}with $b_{m(n-i)}=a_{n-i}$ for all $0\leq i\leq n-1$. We obtain: \begin{equation*}\nu(f)\leq\max_{0\leq i\leq n-1}\vert b_{m(n-i)}\vert_{\spc,\varpi}^{1/m(n-i)}=\max_{0\leq i\leq n-1}\vert a_{n-i}\vert_{\spc,\varpi}^{1/m(n-i)}=\max_{1\leq i\leq n}\vert a_{i}\vert_{\spc,\varpi}^{1/mi}.\end{equation*}Thus \begin{equation*}\nu(f)^{m}\leq\max_{1\leq i\leq n}\vert a_{i}\vert_{\spc,\varpi}^{1/i}.\end{equation*}Since this holds for every equation of integral dependence of $f^{m}$ over $\mathcal{A}$, we get the desired inequality $\nu(f)^{m}\leq\nu(f^{m})$. This finally allows us to apply Lemma~\ref{Explicit power-multiplicative seminorms 1} to the function $\nu$ and conclude that $\nu\leq p_{U^{\prime},\varpi}=\vert\cdot\vert_{\spc,\varpi}$. \end{proof}
\begin{cor}\label{Spectral seminorms on integral extensions}Let $A\hookrightarrow B$ be an integral extension of rings, let $\varpi\in A$ be a non-unit and a non-zero-divisor in both $A$ and $B$ and suppose that $A$, $B$ are integrally closed in $\mathcal{A}=A[\varpi^{-1}]$, $\mathcal{B}=B[\varpi^{-1}]$, respectively. Then for all $f\in \mathcal{B}$ the spectral seminorm $\vert f\vert_{\spc,\varpi}$ of $f$ has the form \begin{equation*}\vert f\vert_{\spc,\varpi}=\inf(\max_{1\leq i\leq n}\vert a_{i}\vert_{\spc,\varpi}^{1/i}),\end{equation*}where the infimum is taken over all equations of integral dependence \begin{equation*}X^{n}+a_{1}X^{n-1}+\dots+a_{n-1}X+a_{n}=0,\end{equation*}$a_{i}\in \mathcal{A}$, satisfied by $f$.\end{cor}
\begin{proof}By Lemma~\ref{Spectral isometries}, the map $\mathcal{A}\hookrightarrow \mathcal{B}$ is an isometry for the two spectral seminorms $\vert\cdot\vert_{\spc,\varpi}$. By Lemma~\ref{Integral extensions and the Berkovich spectrum 2}, every power-multiplicative seminorm on $\mathcal{B}$ whose restriction to $\mathcal{A}$ is bounded above by $\vert\cdot\vert_{\spc,\varpi}$ on $\mathcal{A}$ is also bounded above by $\vert\cdot\vert_{\spc,\varpi}$ on $\mathcal{B}$. Hence we can apply Proposition~\ref{Spectral seminorms on integral extensions 0} to conclude.  
\end{proof}
\begin{rmk}\label{Equivalence of the two versions}It follows from the proof of Proposition~\ref{Spectral seminorms on integral extensions 0} that the hypotheses of that proposition and the hypotheses of Corollary~\ref{Spectral seminorms on integral extensions} are actually equivalent. That is, if $\mathcal{A}\hookrightarrow \mathcal{B}$ is an integral extension between uniform Tate rings satisfying the assumptions in Proposition~\ref{Spectral seminorms on integral extensions 0}, then there exist rings of integral elements $A$, $B$ of $\mathcal{A}$, $\mathcal{B}$, respectively, such that $A$ maps into $B$ and the restricted map $A\hookrightarrow B$ is integral. 
Indeed, in order to see this, it suffices to prove that under the assumptions of Proposition~\ref{Spectral seminorms on integral extensions 0} the integral closure $\mathcal{B}^{+}$ of $\mathcal{A}^{\circ}$ in $\mathcal{B}$ is an open subring of $\mathcal{B}$. But the equality of spectral seminorms\begin{equation*}\vert\cdot\vert_{\spc,\varpi}=\vert\cdot\vert_{\spc,\mathcal{B}^{+},\varpi}\end{equation*}proved in the third paragraph of the proof of Proposition~\ref{Spectral seminorms on integral extensions 0} shows that the topology on $\mathcal{B}$ coincides with the topology defined by the pair $(\mathcal{B}^{+}, \varpi)$. \end{rmk}
Recall from \cite{Atiyah-MacDonald}~Proposition~5.15, that for any integral extension of domains $A\hookrightarrow B$ with $A$ a normal domain, the coefficients of the minimal polynomial (over the field of fractions of $A$) of any $f\in B$ lie in $A$. By a Tate domain we mean a Tate ring $\mathcal{A}$ whose underlying ring is an integral domain. 
\begin{lemma}\label{Spectral seminorms and the minimal polynomial}Let $\mathcal{A}\hookrightarrow \mathcal{B}$ be an integral extension between uniform Tate domains, where $\mathcal{A}$ is a normal domain, and let $\varpi\in \mathcal{A}$ be a topologically nilpotent unit. Suppose that $\mathcal{A}\hookrightarrow \mathcal{B}$ is an isometry with respect to the two spectral seminorms $\vert\cdot\vert_{\spc,\varpi}$ and that every power-multiplicative seminorm on $\mathcal{B}$ whose restriction to $\mathcal{A}$ is bounded above by $\vert\cdot\vert_{\spc,\varpi}$ on $\mathcal{A}$ is also bounded above by $\vert\cdot\vert_{\spc,\varpi}$ on $\mathcal{B}$. If $f\in \mathcal{B}$ satisfies $\vert f\vert_{\spc,\varpi}\leq1$, then the coefficients $a_{1},\dots, a_{n}$ of the minimal polynomial of $f$ over the fraction field of $\mathcal{A}$ satisfy $\vert a_{i}\vert_{\spc,\varpi}\leq1$ for all $i=1,\dots, n$.\end{lemma}
\begin{proof}We follow the first paragraphs of the proof of \cite{Guennebaud}~Ch.~II, Proposition~3. If $f\in\mathcal{B}$ satisfies $\vert f\vert_{\spc,\varpi}<1$, then by Proposition~\ref{Spectral seminorms on integral extensions 0} there exists an equation of integral dependence \begin{equation*}f^{m}+b_{1}f^{m-1}+\dots+b_{m-1}f+b_{m}=0\end{equation*}with $b_{i}\in \mathcal{A}$ and $\vert b_{i}\vert_{\spc,\varpi}<1$ for all $i$. In particular, $f$ belongs to the integral closure $\mathcal{B}^{+}$ of $\mathcal{A}^{\circ}$ inside $\mathcal{B}$, so the coefficients of the minimal polynomial of $f$ over the field of fractions of $\mathcal{A}$ belong to the normal domain $\mathcal{A}^{\circ}$. 

It remains to prove the assertion for $f\in\mathcal{B}$ with $\vert f\vert_{\spc,\varpi}=1$. Suppose that there exists an index $i_{0}\in\{1,\dots, n\}$ with $\vert a_{i_{0}}\vert_{\spc,\varpi}>1$. First assume that there exists an element $\alpha\in \mathcal{A}$ such that $\alpha$ is multiplicative with respect to $\vert\cdot\vert_{\spc,\varpi}$ and \begin{equation*}1>\vert\alpha\vert_{\spc,\varpi}>\vert a_{i_{0}}\vert_{\spc,\varpi}^{-1/i_{0}}.\end{equation*}Then, on one hand, we have $\vert \alpha f\vert_{\spc,\varpi}=\vert\alpha\vert_{\spc,\varpi}\vert f\vert_{\spc,\varpi}<1$ and, on the other hand, $\alpha^{i_{0}}a_{i_{0}}$ is a coefficient of the minimal polynomial of $\alpha f$ with $\vert \alpha^{i_{0}}a_{i_{0}}\vert_{\spc,\varpi}>1$, which contradicts what we proved in the last paragraph.

In the general case, we can adjoin an element $\alpha$ with the above properties by the following procedure. For any real number $r>0$ we can define a seminormed $\mathcal{A}$-algebra \begin{equation*}(\mathcal{A}[X], \vert\cdot\vert_{r})\end{equation*}by letting $\vert\cdot\vert_{r}$ be the $r$-Gauss norm \begin{equation*}\vert\sum_{j=0}^{n}c_{j}X^{j}\vert_{r}=\max_{0\leq j\leq n}\vert c_{j}\vert_{\spc,\varpi}r^{j},\end{equation*}and similarly for $(\mathcal{B}, \vert\cdot\vert_{\spc,\varpi})$. By construction, the maps of seminormed rings \begin{equation*}(\mathcal{A}, \vert\cdot\vert_{\spc,\varpi})\to (\mathcal{A}[X], \vert\cdot\vert_{r})\end{equation*}and $(\mathcal{B}, \vert\cdot\vert_{\spc,\varpi})\to (\mathcal{B}[X], \vert\cdot\vert_{r})$ are isometries and $X$ is a seminorm-multiplicative element in $\mathcal{A}[X]$ and in $\mathcal{B}[X]$ with seminorm $\vert X\vert_{r}=r$. We verify that the integral extension \begin{equation*}\mathcal{A}[X]\hookrightarrow \mathcal{B}[X]\end{equation*}again satisfies the hypotheses of the lemma when the source and target are viewed as uniform Tate rings by endowing them with the topologies defined by $\vert\cdot\vert_{r}$. 

Since $\mathcal{B}$ is a domain and $\mathcal{A}$ is a normal domain, $\mathcal{B}[X]$ is a domain and $\mathcal{A}[X]$ is a normal domain. The Gauss seminorms $\vert\cdot\vert_{r}$ are again power-multiplicative (combine the Gauss lemma with \cite{Berkovich}~Theorem~1.3.1). Thus the seminorms $\vert\cdot\vert_{r}$ on $\mathcal{A}[X]$ and $\mathcal{B}[X]$ are equal to the spectral seminorms of these Tate rings (relative to the same pseudo-uniformizer $\varpi$) by the uniqueness assertion of Lemma~\ref{Uniform Tate rings}. Since $\mathcal{A}\hookrightarrow \mathcal{B}$ was an isometry with respect to the spectral seminorms $\vert\cdot\vert_{\spc,\varpi}$, it is obvious from the definition of the seminorms $\vert\cdot\vert_{r}$ that the map of seminormed rings \begin{equation*}(\mathcal{A}[X], \vert\cdot\vert_{r})\hookrightarrow (\mathcal{B}[X], \vert\cdot\vert_{r})\end{equation*}is an isometry. Finally, if $\vert\cdot\vert$ is a power-multiplicative seminorm on $\mathcal{B}[X]$ which is bounded above by $\vert\cdot\vert_{r}$ on $\mathcal{A}[X]$, then it is also bounded above by $\vert\cdot\vert_{r}$ on $\mathcal{B}[X]$. Indeed, $\vert\cdot\vert_{r}$ restricts to the spectral seminorm $\vert\cdot\vert_{\spc,\varpi}$ on $\mathcal{A}$, so the restriction of $\vert\cdot\vert$ to $\mathcal{A}$ is bounded above by $\vert\cdot\vert_{\spc,\varpi}$. By our assumption on $\mathcal{A}\hookrightarrow \mathcal{B}$, this implies that the restriction of $\vert\cdot\vert$ to $\mathcal{B}$ is also bounded above by $\vert\cdot\vert_{\spc,\varpi}$. On the other hand, since $\vert\cdot\vert$ is bounded above by $\vert\cdot\vert_{r}$ on $\mathcal{A}[X]$, we also have $\vert X\vert\leq \vert X\vert_{r}=r$. We conclude that for every $\sum_{i}b_{i}X^{i}\in \mathcal{B}[X]$, \begin{equation*}\vert \sum_{i}b_{i}X^{i}\vert\leq \max_{i}\vert b_{i}\vert \vert X\vert^{i}\leq \max_{i}\vert b_{i}\vert_{\spc,\varpi}r^{i}=\vert \sum_{i}b_{i}X^{i}\vert_{r}.\end{equation*}This proves that $\mathcal{A}[X]\hookrightarrow \mathcal{B}[X]$ indeed satisfies all hypotheses of the lemma. 

Now, for every $f\in \mathcal{B}$, the minimal polynomial of $f$ over the field of fractions $\Frac(\mathcal{A})$ of $\mathcal{A}$ is equal to the minimal polynomial of $f$, viewed as an element of $\mathcal{B}[X]$, over the field of fractions of $\mathcal{A}[X]$. It follows that, to prove the inequality $\vert a_{i}\vert_{\spc,\varpi}\leq1$ for the coefficients of the minimal polynomial of some $f\in \mathcal{B}$, it suffices to prove the analogous assertion for the minimal polynomial of $f$ regarded as an element of the seminormed ring $(\mathcal{B}[X], \vert\cdot\vert_{r})$ for some choice of $r>0$. Choosing $r$ with $1>r>\vert a_{i_{0}}\vert_{\spc,\varpi}^{-1/i_{0}}$, we can thus apply the argument from the second paragraph of this proof to conclude.\end{proof}
\begin{prop}[Generalization of \cite{Guennebaud}~Ch.~II, Proposition~3]\label{Spectral seminorms and the minimal polynomial 1.5}Let $\mathcal{A}\hookrightarrow \mathcal{B}$ be an integral extension between uniform Tate domains, where $\mathcal{A}$ is a normal domain, and let $\varpi\in \mathcal{A}$ be a topologically nilpotent unit. Suppose that $\mathcal{A}\hookrightarrow \mathcal{B}$ is an isometry with respect to the two spectral seminorms $\vert\cdot\vert_{\spc,\varpi}$ and that every power-multiplicative seminorm on $\mathcal{B}$ whose restriction to $\mathcal{A}$ is bounded above by $\vert\cdot\vert_{\spc,\varpi}$ on $\mathcal{A}$ is also bounded above by $\vert\cdot\vert_{\spc,\varpi}$ on $\mathcal{B}$. Then for every $f\in \mathcal{B}$ the spectral seminorm $\vert f\vert_{\spc,\varpi}$ of $f$ can be described as \begin{equation*}\vert f\vert_{\spc,\varpi}=\max_{1\leq i\leq n}\vert a_{i}\vert_{\spc,\varpi}^{1/i},\end{equation*}where $X^{n}+a_{1}X^{n-1}+\dots+a_{n-1}X+a_{n}$ is the minimal polynomial of $f$ over the fraction field of $\mathcal{A}$. \end{prop}
\begin{proof}We continue following the proof of \cite{Guennebaud}~Ch.~II, Proposition~3. Suppose first that $\vert f\vert_{\spc,\varpi}=0$. Then $\vert f\vert_{\spc,\varpi}\leq \vert\varpi^{m}\vert_{\spc,\varpi}$ for all $m\geq1$. In other words, $\vert\varpi^{-m}f\vert_{\spc,\varpi}\leq1$ for all $m\geq1$. For any $m\geq1$, the coefficients of the minimal polynomial of $\varpi^{-m}f$ are $\varpi^{-mi}a_{i}$, so $\vert\varpi^{-mi}a_{i}\vert_{\spc,\varpi}\leq1$ for all $i=1,\dots, n$ and all $m\geq1$, by Lemma \ref{Spectral seminorms and the minimal polynomial}. It follows that $\vert a_{i}\vert_{\spc,\varpi}=0$ for all $i=1,\dots, n$, which proves the proposition for elements $f\in \mathcal{B}$ with $\vert f\vert_{\spc,\varpi}=0$. 

Now assume that $\vert f\vert_{\spc,\varpi}\neq0$. By the same procedure as in the last paragraphs of the proof of Lemma~\ref{Spectral seminorms and the minimal polynomial}, we may assume that there exists an element $\alpha\in \mathcal{A}$ which is multiplicative with respect to the seminorm $\vert\cdot\vert_{\spc,\varpi}$ and $\vert\alpha\vert_{\spc,\varpi}=\vert f\vert_{\spc,\varpi}^{-1}$. Under this assumption, the assertion of the proposition reduces to the special case when $\vert f\vert_{\spc,\varpi}=1$. In that case, we know that the coefficients $a_{i}$ of the minimal polynomial of $f$ satisfy $\vert a_{i}\vert_{\spc,\varpi}\leq1$, by Lemma~\ref{Spectral seminorms and the minimal polynomial}. But if $\max_{1\leq i\leq n}\vert a_{i}\vert_{\spc,\varpi}<1$, then, by Proposition~\ref{Spectral seminorms on integral extensions 0}, $\vert f\vert_{\spc,\varpi}<1$, a contradiction. It follows that $\max_{1\leq i\leq n}\vert a_{i}\vert_{\spc,\varpi}=1$, as desired.\end{proof}
\begin{cor}\label{Spectral seminorms and the minimal polynomial 2}Let $A\hookrightarrow B$ be an integral extension of integral domains and suppose that $A$ is a normal domain. Set $\mathcal{A}=A[\varpi^{-1}]$, $\mathcal{B}=B[\varpi^{-1}]$. Let $\varpi\in A$ be a non-zero non-unit in both $A$ and $B$ and suppose that $B$ is integrally closed in $\mathcal{B}$. Then for every $f\in \mathcal{B}$ the spectral seminorm $\vert f\vert_{\spc,\varpi}$ of $f$ can be described as \begin{equation*}\vert f\vert_{\spc,\varpi}=\max_{1\leq i\leq n}\vert a_{i}\vert_{\spc,\varpi}^{1/i},\end{equation*}where $X^{n}+a_{1}X^{n-1}+\dots+a_{n-1}X+a_{n}$ is the minimal polynomial of $f$ over the fraction field of $A$. \end{cor}
\begin{proof}By Lemma~\ref{Spectral isometries} and Lemma~\ref{Integral extensions and the Berkovich spectrum 2}, the integral extension $\mathcal{A}\hookrightarrow \mathcal{B}$ satisfies the hypotheses of Proposition~\ref{Spectral seminorms and the minimal polynomial 1.5}. We conclude by that proposition. \end{proof}
\begin{prop}\label{Complete integral closure and integral extensions}Let $A\hookrightarrow B$ be an integral extension of integral domains with $A$ normal and let $\varpi\in A$ be a non-zero non-unit. If $B$ is integrally closed in $\mathcal{B}=B[\varpi^{-1}]$ and $A$ is completely integrally closed in $\mathcal{A}=A[\varpi^{-1}]$, then $B$ is actually completely integrally closed in $\mathcal{B}$.\end{prop}
\begin{proof}Since $A$, $B$ are integrally closed in $\mathcal{A}$, $\mathcal{B}$, respectively, Lemma~\ref{Almost power-bounded} tells us that the Tate rings $\mathcal{A}$, $\mathcal{B}$ are uniform. By \cite{Dine22}~Lemma 2.24, we then know that $A$ (respectively, $B$) is completely integrally closed in $\mathcal{A}$ (respectively, in $\mathcal{B}$) if and only if it is equal to the closed unit ball with respect to $\vert\cdot\vert_{\spc,\varpi}$. We conclude by applying Corollary~\ref{Spectral seminorms and the minimal polynomial 2}. \end{proof}
\begin{rmk}\label{Krull's result}The above proposition can be thought of as a ``$\varpi$-local" analog of the following classical theorem of Krull (\cite{Krull36}~Satz~11): If $A$ is a completely integrally closed integral domain and $B$ is the integral closure of $A$ in an algebraic extension of its fraction field, then $B$ is completely integrally closed. The author is grateful to Kazuma Shimomoto for bringing this analogy to his attention.\end{rmk} 
We can now prove our main theorem on (completed) integral extensions.
\begin{thm}\label{Strongly Shilov rings and integral extensions}Let $A\hookrightarrow B$ be an integral extension of domains with $A$ a normal domain and let $\varpi\in A$ be a non-zero non-unit in $A$ and in $B$. Set $\mathcal{A}=A[\varpi^{-1}]$ and $\mathcal{B}=B[\varpi^{-1}]$. If $A$ is strongly $\varpi$-Shilov and $B$ is integrally closed in $\mathcal{B}$, then $\mathcal{B}$ (and hence also its completion $\widehat{\mathcal{B}}$) satisfies Berkovich's description of the Shilov boundary.\end{thm}
\begin{proof}By Corollary~\ref{Strongly Shilov implies completely integrally closed}, the assumption that $A$ is strongly $\varpi$-Shilov implies that $A$ is completely integrally closed in $\mathcal{A}$. By Proposition~\ref{Complete integral closure and integral extensions}, this implies that $B$ is completely integrally closed in $\mathcal{B}$. By Proposition~\ref{Integral extensions, first result}, $B$ is weakly $\varpi$-Shilov. By Proposition~\ref{Weakly Shilov}, this means that the pre-image under $\spc$ of the generic point of any irreducible component of $\Spf(B)$ consists of a single rank-$1$ point. By Lemma~\ref{Boundaries and Shilov boundaries}, to prove that $\mathcal{B}$ satisfies Berkovich's description of the Shilov boundary, it only remains to prove that the closure of the set of pre-images under $\spc$ of generic points of $\Spf(B)$ is a boundary for $\mathcal{B}$. Let $\mathcal{S}=\spc^{-1}(\Spf(B)_{\gen})$. We want to prove that $\vert f\vert_{\spc,\varpi}=\sup_{w\in\mathcal{S}}w(f)$ for all $f\in B$. If $\vert f\vert_{\spc,\varpi}=0$, there is nothing to show. If $\vert f\vert_{\spc,\varpi}>0$, let $r>0$ be such that $\vert f\vert_{\spc,\varpi}>r$. We want to show that there exists some $w\in \mathcal{S}$ with $w(f)>r$. Let $a_{1},\dots, a_n$ be the non-leading coefficients of the minimal polynomial of $f$ over the fraction field of $A$. Choose $i_{0}$ such that \begin{equation*}\vert a_{i_{0}}\vert_{\spc,\varpi}^{1/i_{0}}=\max_{1\leq i\leq n}\vert a_{i}\vert_{\spc,\varpi}^{1/i}.\end{equation*}By Corollary~\ref{Spectral seminorms and the minimal polynomial 2}, $r<\vert f\vert_{\spc,\varpi}=\vert a_{i_{0}}\vert_{\spc,\varpi}^{1/i_{0}}$. Since $A$ is strongly $\varpi$-Shilov, Theorem~\ref{Strongly Shilov} guarantees the existence of some $v\in\mathcal{M}(\mathcal{A})$ with $\mathfrak{p}=\spc(v)$ a minimal prime ideal over $\varpi$ such that $v(a_{i_{0}})>r^{i_{0}}$. 

Since $\widehat{A_{\mathfrak{p}}}$ is a valuation ring of rank $1$, the Berkovich spectrum $\mathcal{M}(A_{\mathfrak{p}}[\varpi^{-1}])$ consists of a single point $v_{\mathfrak{p}}$ (Lemma~\ref{Valuative rings and adic spectrum}) and this is precisely the valuation corresponding to $v$. By \cite{Berkovich}~Theorem~1.3.1, this implies \begin{equation*}v_{\mathfrak{p}}=\vert\cdot\vert_{\spc,A_{\mathfrak{p}},\varpi},\end{equation*}the spectral seminorm (relative to $\varpi$) on the uniform Tate ring $A_{\mathfrak{p}}[\varpi^{-1}]$. The localization $B_{\mathfrak{p}}=(A\setminus \mathfrak{p})^{-1}B$ is integral over $A_{\mathfrak{p}}$ and, since $B$ is integrally closed in $\mathcal{B}=B[\varpi^{-1}]$, it is integrally closed in $B_{\mathfrak{p}}[\varpi^{-1}]$. By Lemma~\ref{Spectral isometries}, the spectral seminorm $\vert\cdot\vert_{\spc,B_{\mathfrak{p}},\varpi}$ on $B_{\mathfrak{p}}[\varpi^{-1}]$ restricts to $v_{\mathfrak{p}}$ on $A_{\mathfrak{p}}[\varpi^{-1}]$. 

If $F$ is the minimal polynomial of $f$ over $\Frac(A)$, then $F$ is also the minimal polynomial of $\frac{f}{1}\in B_{\mathfrak{p}}$ over $\Frac(A)=\Frac(A_{\mathfrak{p}})$. By Corollary~\ref{Spectral seminorms and the minimal polynomial 2} applied to the integral extension $A_{\mathfrak{p}}\hookrightarrow B_{\mathfrak{p}}$ and the spectral seminorms $\vert\cdot\vert_{\spc,A_{\mathfrak{p}},\varpi}=v_{\mathfrak{p}}$ and $\vert\cdot\vert_{\spc,B_{\mathfrak{p}},\varpi}$, we see that \begin{equation*}\vert\frac{f}{1}\vert_{\spc,B_{\mathfrak{p}},\varpi}=\max_{1\leq i\leq n}v_{\mathfrak{p}}(\frac{a_{i}}{1})^{1/i}=\max_{1\leq i\leq n}v(a_{i})^{1/i}\geq v(a_{i_{0}})^{1/i_{0}}.\end{equation*}But $v(a_{i_{0}})^{1/i_{0}}>r$, by our choice of $i_{0}$ and $v$. Hence, by applying \cite{Berkovich}~Theorem~1.3.1, to the spectral seminorm $\vert\cdot\vert_{\spc,B_{\mathfrak{p}},\varpi}$ on $B_{\mathfrak{p}}[\varpi^{-1}]$, we see that there exists some $w_{\mathfrak{p}}\in\mathcal{M}(B_{\mathfrak{p}}[\varpi^{-1}])$ with $w_{\mathfrak{p}}(\frac{f}{1})>r$. Let $w\in\mathcal{M}(\mathcal{B})$ be the restriction of $w_{\mathfrak{p}}$ to $\mathcal{B}=B[\varpi^{-1}]$. Then the center $\mathfrak{q}=\spc(w)$ on $B$ corresponds to a prime ideal of $B_{\mathfrak{p}}$ containing $\varpi$. In particular, $\mathfrak{q}\cap A=\mathfrak{p}$, since $\mathfrak{p}$ is minimal over $\varpi$. If $\mathfrak{q}'$ is a prime ideal of $B$ containing $\varpi$ and contained in $\mathfrak{q}$, then, by the minimality of $\mathfrak{p}$ over $\varpi$, we have $\mathfrak{q}'\cap A=\mathfrak{p}$, and then $\mathfrak{q}'=\mathfrak{q}$ by ``incomparability" (\cite{Cohen-Seidenberg}~Theorem~4). Thus $\mathfrak{q}$ is a minimal prime ideal over $\varpi$, and we have found an element $w\in \spc^{-1}(\Spf(B)_{\gen})$ such that \begin{equation*}w(f)=w_{\mathfrak{p}}(\frac{f}{1})>r.\end{equation*}This proves that the closure of $\spc^{-1}(\Spf(B)_{\gen})$ in $\mathcal{M}(\mathcal{B})$ is a boundary for $\mathcal{B}$. \end{proof}
\begin{cor}\label{Integral extensions}Let $K$ be a nonarchimedean field, let $\mathcal{A}$ be an affinoid $K$-algebra in the sense of Tate whose underlying ring is a domain and let\begin{equation*}\mathcal{A}\hookrightarrow\mathcal{B}\end{equation*}be an extension of uniform complete Tate rings. Suppose that there exists a dense $\mathcal{A}$-subalgebra $\mathcal{B}'$ of $\mathcal{B}$ which is integral over $\mathcal{A}$ such that $\mathcal{B}^{\prime}$ is an integral domain and the integral closure $\mathcal{B}^{\prime+}$ of $\mathcal{A}^{\circ}$ inside $\mathcal{B}^{\prime}$ is open in $\mathcal{B}^{\prime}$. Then $\mathcal{B}$ satisfies Berkovich's description of the Shilov boundary in the strong sense.\end{cor}
\begin{proof}Using Noether normalization for affinoid algebras, we may assume $\mathcal{A}$ is a normal domain. Let $\varpi\in\mathcal{A}$ be a topologically nilpotent unit. Note that \begin{equation*}\mathcal{B}^{+}:=\widehat{\mathcal{B}^{\prime+}}\end{equation*}($\varpi$-adic completion) is an open and integrally closed subring of $\mathcal{B}$ (it is open since $\mathcal{B}^{\prime+}$ was assumed to be open in $\mathcal{B}^{\prime}$). Since $\mathcal{B}$ is uniform, this implies that $\mathcal{B}^{+}$ is a ring of definition of $\mathcal{B}$. By Theorem~\ref{Strongly Shilov} and Berkovich's \cite{Berkovich}~Proposition~2.4.4, $\mathcal{A}^{\circ}$ is strongly $\varpi$-Shilov (note that $\mathcal{A}^{\circ}/(\varpi)_{\mathcal{A}^{\circ}}$ clearly has finitely many minimal prime ideals, its nilreduction being Noetherian). By Theorem~\ref{Strongly Shilov rings and integral extensions} applied to the integral extension $\mathcal{A}^{\circ}\hookrightarrow\mathcal{B}^{\prime+}$, we see that $\mathcal{B}$ satisfies Berkovich's description of the Shilov boundary.

It remains to prove that $\spc^{-1}(\Min_{\mathcal{B}^{\prime+}}(\varpi))$ is closed in $\mathcal{M}(\mathcal{B}^{\prime})=\mathcal{M}(\mathcal{B})$. 
Write $\mathcal{B}^{\prime+}$ as a direct limit of finite subextensions $C_{i}$ of $\mathcal{A}^{\circ}$. By \cite{Kedlaya17}~Corollary~1.1.15, each of the finitely generated $\mathcal{A}$-modules $\mathcal{C}_{i}=C_{i}[\varpi^{-1}]$ is complete for its natural topology. Since $C_{i}$ is finitely generated as an $\mathcal{A}^{\circ}$-module, the natural topology on $\mathcal{C}_{i}$ coincides with the topology given by the pair of definition $(C_{i}, \varpi)$. Thus each $\mathcal{C}_{i}$, with topology defined by the pair $(C_{i}, \varpi)$, is an affinoid algebra in the sense of Tate. We know from Proposition~\ref{Complete integral closure and integral extensions}, or from \cite{BGR}~Theorem~6.3.5/1, that the integral closure of $C_{i}$ (or, equivalently, of $\mathcal{A}^{\circ}$) in $\mathcal{C}_{i}$ is completely integrally closed in $\mathcal{C}_{i}$, so it is equal to the subring of power-bounded elements $\mathcal{C}_{i}^{\circ}$. In particular, \begin{equation*}\mathcal{B}^{\prime+}=\varinjlim_{i\geq0}\mathcal{C}_{i}^{\circ}.\end{equation*}By Berkovich's description of the Shilov boundary for affinoid algebras (\cite{Berkovich}~Proposition~2.4.4), $\spc^{-1}(\Min_{\mathcal{C}_{i}^{\circ}}(\varpi))$ is equal to the Shilov boundary for $\mathcal{C}_{i}$ and is thus closed. By Lemma~\ref{Minimal primes and integral extensions 2} and Lemma~\ref{Closed subsets of the Berkovich spectrum}, this implies that $\spc^{-1}(\Min_{\mathcal{B}^{\prime+}}(\varpi))$ is closed in $\mathcal{M}(\mathcal{B}^{\prime})$, as required. \end{proof}
\begin{cor}\label{Integral extensions 2}Let \begin{equation*}\mathcal{A}\hookrightarrow\mathcal{B}\end{equation*}be an extension of uniform Tate rings, with topologically nilpotent unit $\varpi\in\mathcal{A}$, such that there exists a dense subring $\mathcal{B}^{\prime}$ in $\mathcal{B}$ containing $\mathcal{A}$ which is an integral domain and is integral over $\mathcal{A}$. Suppose that the integral closure $\mathcal{B}^{\prime+}$ of $\mathcal{A}^{\circ}$ inside $\mathcal{B}^{\prime}$ is an open subring of $\mathcal{B}^{\prime}$. Suppose further that $\mathcal{A}^{\circ}$ is a coherent normal domain, that $\mathcal{A}^{\circ}/(\varpi)_{\mathcal{A}^{\circ}}$ has only finitely many minimal prime ideals and that \begin{equation*}\vert \mathcal{A}\vert_{\spc,\varpi}\subseteq \sqrt{\vert\mathcal{A}^{\times,m}\vert_{\spc,\varpi}}\cup\{0\}.\end{equation*}Then $\mathcal{B}$ satisfies Berkovich's description of the Shilov boundary.\end{cor}
\begin{proof}By Theorem~\ref{Big theorem}, the hypotheses on $\mathcal{A}^{\circ}$ guarantee that $\mathcal{A}^{\circ}$ is strongly $\varpi$-Shilov. We see as in the proof of Corollary~\ref{Integral extensions} that the topological closure $\mathcal{B}^{+}$ of $\mathcal{B}^{\prime+}$ in $\mathcal{B}$ is a ring of definition of $\mathcal{B}$. We conclude by applying Theorem~\ref{Strongly Shilov rings and integral extensions} to the integral extension $\mathcal{A}^{\circ}\hookrightarrow \mathcal{B}^{\prime+}$. \end{proof}
Motivated by Theorem~\ref{Strongly Shilov rings and integral extensions}, we leave the following questions for the interested reader to explore.
\begin{question}\label{First open question}Let $A\hookrightarrow B$ be an integral extension of integral domains where $A$ is a normal domain, and let $\varpi\in A$ be a non-zero non-unit such that $B$ is integrally closed in $\mathcal{B}=B[\varpi^{-1}]$. If $A/(\varpi)_{A}$ satisfies condition S1, does $B/(\varpi)_{B}$ necessarily satisfy S1?\end{question}
\begin{question}\label{Second open question}Let $A\hookrightarrow B$ be an integral extension of integral domains where $A$ is a normal domain, and let $\varpi\in A$ be a non-zero non-unit such that $B$ is integrally closed in $\mathcal{B}=B[\varpi^{-1}]$. If $A$ is strongly $\varpi$-Shilov (respectively, strongly $\varpi$-Krull), is the same also true for $B$? \end{question}
Note that an affirmative answer to Question~\ref{First open question} would imply an affirmative answer to Question~\ref{Second open question} by virtue of Proposition~\ref{Integral extensions, first result}. We give an overview of the special cases when a positive answer to these questions is known.
\begin{itemize}\item Both questions have a positive answer when $A/(\varpi)_{A}$ has only finitely many minimal prime ideals and the integral extension $A\hookrightarrow B$ is flat (Theorem~\ref{Flat integral extensions} and its corollary). \item Both questions have a positive answer when $A$ is Noetherian (Theorem~\ref{Integral extensions of Noetherian rings}). \item Question~\ref{Second open question} has a positive answer when $B/(\varpi)_{B}$ has only finitely many minimal prime ideals (Theorem~\ref{Strongly Shilov rings and integral extensions} in conjunction with Theorem~\ref{Strongly Shilov}). \end{itemize}
\noindent{\textbf{Competing interest.} The author has no competing interest to declare.}

\bibliographystyle{plain} 
\bibliography{Bib}
\

\textsc{Department of Mathematics, University of California San Diego, La Jolla, CA 92093, United States} \newline 

E-mail address: \textsf{ddine@ucsd.edu}

\end{document}